\documentclass[10pt,reqno]{amsart}

\usepackage[utf8x]{inputenc}
\usepackage[english]{babel}
\usepackage{amsmath,amsfonts,amssymb,amsthm,shuffle}
\usepackage[T1]{fontenc}
\usepackage{lmodern}

\usepackage[top=3.5cm,bottom=3.5cm,left=3.6cm,right=3.6cm]{geometry}

\usepackage[dvipsnames]{xcolor}
\usepackage[hyperindex=true,frenchlinks=true,colorlinks=true,
citecolor=Mahogany,linkcolor=DarkOrchid,urlcolor=Tan,linktocpage,
pagebackref=true]{hyperref}

\usepackage{tikz}
\usetikzlibrary{shapes}
\usetikzlibrary{fit}
\usetikzlibrary{decorations.pathmorphing}

\usepackage{dsfont}
\usepackage{wasysym}
\usepackage{stmaryrd}
\usepackage{cite}
\usepackage{subfigure}
\usepackage{multirow}
\usepackage{enumitem}
\usepackage{multicol}

\title{Operads from posets and Koszul duality}
\date{\today}
\author{Samuele Giraudo}
\address{Laboratoire d'Informatique Gaspard-Monge, Université Paris-Est
    Marne-la-Vallée, 5 boulevard Descartes, Champs-sur-Marne,
    77454 Marne-la-Vallée cedex 2, France}
\email{samuele.giraudo@u-pem.fr}

\numberwithin{equation}{subsection}
\renewcommand{\leq}{\leqslant}
\renewcommand{\geq}{\geqslant}

\newtheorem{Theorem}{Theorem}[subsection]
\newtheorem{Proposition}[Theorem]{Proposition}
\newtheorem{Lemma}[Theorem]{Lemma}

\tikzstyle{Vertex}=[circle,draw=LimeGreen!80,fill=LimeGreen!8,
inner sep=1pt,minimum size=2mm,line width=1pt,font=\scriptsize]
\tikzstyle{Node}=[Vertex,draw=RoyalBlue!80,fill=RoyalBlue!8,inner sep=1.5pt]
\tikzstyle{Leaf}=[rectangle,draw=Black!70,fill=Black!16,
inner sep=0pt,minimum size=1mm,line width=1.25pt]
\tikzstyle{Edge}=[Maroon!80,cap=round,line width=1pt]
\tikzstyle{Mark1}=[draw=BrickRed!80,fill=BrickRed!8]
\tikzstyle{Mark2}=[draw=BurntOrange!80,fill=BurntOrange!8]
\tikzstyle{EdgeRew}=[->,RedOrange!80,cap=round,thick]

\newcommand{\N}{\mathbb{N}}
\newcommand{\K}{\mathbb{K}}

\newcommand{\Aca}{\mathcal{A}}

\newcommand{\Cca}{\mathcal{C}}
\newcommand{\Eca}{\mathcal{E}}
\newcommand{\Fca}{\mathcal{F}}

\newcommand{\Hca}{\mathcal{H}}
\newcommand{\Ica}{\mathcal{I}}
\newcommand{\Nca}{\mathcal{N}}
\newcommand{\Oca}{\mathcal{O}}
\newcommand{\Pca}{\mathcal{Q}}
\newcommand{\Sca}{\mathcal{S}}

\newcommand{\Xbb}{\mathbb{X}}
\newcommand{\Cfr}{\mathfrak{c}}
\newcommand{\Tfr}{\mathfrak{t}}
\newcommand{\Sfr}{\mathfrak{s}}
\newcommand{\Rfr}{\mathfrak{r}}

\newcommand{\La}{\mathtt{a}}
\newcommand{\Lb}{\mathtt{b}}
\newcommand{\Lc}{\mathtt{c}}

\newcommand{\Min}{\uparrow}
\newcommand{\As}{\mathsf{As}}
\newcommand{\Op}{\star}
\newcommand{\OpA}{\square}
\newcommand{\OpB}{\triangle}
\newcommand{\OpDual}{\bar \Op}
\newcommand{\OpDualA}{\bar \OpA}
\newcommand{\OpDualB}{\bar \OpB}
\newcommand{\Free}{\mathrm{Free}}
\newcommand{\FreeGen}{\mathfrak{G}}
\newcommand{\FreeRel}{\mathfrak{R}}

\newcommand{\Ord}{\preccurlyeq}
\newcommand{\OrdStrict}{\prec}
\newcommand{\Rew}{\to}
\newcommand{\RewTrans}{\overset{*}{\rightarrow}}
\newcommand{\RewEquiv}{\overset{*}{\leftrightarrow}}
\newcommand{\RewS}{\rightarrowtriangle}

\newcommand{\NbInterv}{\mathrm{int}}
\newcommand{\Hilbert}{\Hca}
\newcommand{\Alg}{\Aca}
\newcommand{\Unit}{\mathds{1}}
\newcommand{\UnitSet}{\Eca}
\newcommand{\Ideal}{\Ica}
\newcommand{\ASchr}{\mathsf{ASchr}}
\newcommand{\ASchrMap}{\mathrm{s}}
\newcommand{\Corolla}{\Cfr}
\newcommand{\NormalF}{\Nca}
\newcommand{\SetASchr}{\Sca}
\newcommand{\Pos}{\mathrm{pos}}
\newcommand{\TwoAs}{\mathit{2as}}

\newcommand{\AntiForestPattern}{%
\begin{tikzpicture}
        [baseline=(current bounding box.center),xscale=.2,yscale=.25]
    \node[Vertex](1)at(0,0){};
    \node[Vertex](2)at(-1,1){};
    \node[Vertex](3)at(1,1){};
    \draw[Edge](1)--(2);
    \draw[Edge](1)--(3);
\end{tikzpicture}}

\newcommand{\LinePattern}{%
\begin{tikzpicture}[baseline=(current bounding box.center),yscale=.28]
    \node[Vertex](1)at(0,0){};
    \node[Vertex](2)at(0,-1){};
    \draw[Edge](1)--(2);
\end{tikzpicture}}

\newcommand{\OnePoset}{%
\begin{tikzpicture}[baseline=(current bounding box.center)]
    \node[Vertex](1)at(0,0){};
\end{tikzpicture}}

\newcommand{\PosetEmpty}{\emptyset}

\newcommand{\Hide}[1]{}
\newcommand{\Sloane}[1]{\href{http://oeis.org/#1}{{\bf #1}}}

\begin{document}

\maketitle

\begin{abstract}
    We introduce a functor $\As$ from the category of posets to the
    category of nonsymmetric binary and quadratic operads, establishing
    a new connection between these two categories. Each operad obtained
    by the construction $\As$ provides a generalization of the
    associative operad because all of its generating operations are
    associative. This construction has a very singular property: the
    operads obtained from $\As$ are almost never basic. Besides, the
    properties of the obtained operads, such as Koszulity, basicity,
    associative elements, realization, and dimensions, depend on
    combinatorial properties of the starting posets. Among others, we
    show that the property of being a forest for the Hasse diagram of
    the starting poset implies that the obtained operad is Koszul.
    Moreover, we show that the construction $\As$ restricted to a certain
    family of posets with Hasse diagrams satisfying some combinatorial
    properties is closed under Koszul duality.
\end{abstract}

\tableofcontents

\section*{Introduction}
Very recent examples show that the theory of operads, since its
introduction~\cite{May72,BV73} in the 1970s, is developing year after
year many fruitful connections with combinatorics. In fact, there are at
least three important instances to warrant this assertion. First, the
Hilbert series of a Koszul operad and its Koszul dual are related by an
inversion operation on series~\cite{GK94,BB10}, a classical combinatorial
operation. Moreover, the Koszulity property for operads is strongly
related to the theory of rewrite rules on trees, this last providing a
sufficient combinatorial condition to establish that an operad is
Koszul~\cite{Hof10,DK10,LV12}. Finally, another strategy to prove that
an operad is Koszul consists in constructing a family of posets from an
operad~\cite{MY91}, so that the Koszulity of the considered operad is a
consequence of a combinatorial property of the posets of the obtained
family~\cite{Val07} (see also~\cite{LV12}).
\medskip

Moreover, by endowing a set of combinatorial objects with an operad
structure, the theory of operads helps to establish combinatorial
properties of its objects and leads to new ways of understanding it.
Indeed, since operads are algebraic structures modeling the composition
of operators, operads on combinatorial objects allow to decompose objects
into elementary pieces. This may potentially bring and highlight
combinatorial properties. For instance, several operads on various sorts
of trees emphasize some properties of trees~\cite{Cha08}, the definition
of an operad structure on a sort of noncrossing configurations in
polygons leads to a formula for their enumeration~\cite{CG14}, and the
definition of several operads on combinatorial objects such as Motzkin
paths and directed animals leads to see these as gluings of small
objects~\cite{Gir15a} and to derive properties.
\medskip

On the other side, purely combinatorial observations lead to algebraic
properties on operads. For instance, computations in the dendriform
operad, introduced by Loday~\cite{Lod01}, can be performed by
considering intervals of the Tamari order~\cite{Tam62} (see~\cite{LR02}).
Furthermore, pre-Lie algebras, introduced independently by
Vinberg~\cite{Vin63} and by Gerstenhaber~\cite{Ger63}, form an important
class of algebras appearing among others in the study of homogeneous cones
and affine manifolds. The pre-Lie operad, defined by Chapoton and
Livernet~\cite{CL01}, encodes the category of pre-Lie algebras. This
operad involves rooted trees and its partial composition is described by
means of tree grafts. The existence of this operad and of its combinatorial
description in terms of rooted trees lead to the construction of free
pre-Lie algebras.
\medskip

This work is devoted to enrich these connections between operads and
combinatorics by establishing a new link between posets and nonsymmetric
operads by means of a construction associating a nonsymmetric operad
$\As(\Pca)$, called {\em $\Pca$-associative operad}, with any finite
poset~$\Pca$. This construction is a functor $\As$ from the category of
finite posets to the category of nonsymmetric binary and quadratic
operads. The will to generalize two families of operads Koszul dual to
each other, constructed in a previous work~\cite{Gir15b}, is the first
impetus of this paper. The operads of these families, called
multiassociative operads and dual multiassociative operads, are
respectively denoted by $\As_\gamma$ and $\As_\gamma^!$, and depend on a
nonnegative integer parameter $\gamma$. In this present work, we retrieve
$\As_\gamma$ by applying the construction $\As$ on the total order on a
set of $\gamma$ elements and we retrieve $\As_\gamma^!$ by applying the
construction $\As$ on the trivial order on the same set.
\medskip

Let us describe some main properties of $\As$. First, each operad
obtained by our construction provides a generalization of the
associative operad since all its generating operations are associative.
Besides, many combinatorial properties of the starting poset $\Pca$ lead
to algebraic properties for $\As(\Pca)$ (see Table~\ref{tab:properties}).
\begin{table}[ht]
    \centering
    \begin{tabular}{c|c|c}
        Properties of the poset $\Pca$ & Properties of the operad $\As(\Pca)$
                & Statement \\ \hline \hline
        None & Binary and quadratic
                & Definition, Section~\ref {subsubsec:definition_as}
                \\ \hline
        Forest & Koszul & Theorem~\ref{thm:koszulity} \\ \hline
        Thin forest & Closed under Koszul duality
                & Theorem~\ref{thm:isomorphism_thin_forests_dual_operad}
                \\ \hline
        Trivial & Basic & Proposition~\ref{prop:basic_operad}
    \end{tabular}
    \bigskip
    \caption{Summary of the properties satisfied by a poset $\Pca$
    implying properties for the operad $\As(\Pca)$. Note that any trivial
    poset is also a thin forest poset, and that a thin forest poset is
    also a forest poset. In particular, if $\Pca$ is a trivial poset,
    $\As(\Pca)$ has all properties mentioned in the middle column.}
    \label{tab:properties}
\end{table}
For instance, when $\Pca$ is a forest (with the meaning that no element
of $\Pca$ covers two different elements), $\As(\Pca)$ is a Koszul operad.
Moreover, when $\Pca$ is not a trivial poset, $\As(\Pca)$ is not a basic
operad (property defined in~\cite{Val07}). This last property seems to
be interesting since almost all common operads are basic, such as
the associative operad or the diassociative operad~\cite{Lod01} (see
additionally~\cite{Zin12} and the references therein for a census of the
most famous operads and their main properties). This gives to our
construction a very unique flavor.
\medskip

The further study of the operads obtained by the construction $\As$ is
driven by computer exploration. Indeed, computer experiments bring us
the observation that some operads obtained by the construction $\As$ are
Koszul dual to each other. This observation raises several questions.
The first one consists in describing the family of posets such that the
construction $\As$ restrained to this family is closed under Koszul duality.
The second one consists in defining an operation $^\perp$ on this family
of posets such that for any poset $\Pca$ of this family,
$\As(\Pca^\perp)$ is isomorphic to the Koszul dual $\As(\Pca)^!$ of
$\As(\Pca)$. The last one relies on an expression of an explicit
isomorphism between $\As(\Pca)^!$ and $\As(\Pca^\perp)$. We answer all
these questions in this work, forming its main result.
\medskip

This paper and the presented results are organized as follows.
Section~\ref{sec:definitions_lemmas} contains basic definitions about
posets, rewrite rules on trees, and operads. We state therein some
lemmas used in the sequel about these objects. Most important are
Lemma~\ref{lem:confluent_few_nodes} establishing a criterion to prove
that a rewrite rule is confluent, and
Lemma~\ref{lem:koszulity_criterion_pbw} establishing a criterion to
prove that an operad is Koszul. These results are already known: the
first one belongs to folklore and we provide a proof for it, and the
second one is a consequence of works of Dotsenko and
Khoroshkin~\cite{DK10}, and Hoffbeck~\cite{Hof10}.
\medskip

Section~\ref{sec:posets_to_operads} is concerned with the description of
the construction $\As$ and the first general properties of the obtained
operads. We begin by showing that $\As$ is functorial
(Theorem~\ref{thm:as_poset_functor}). Next, for any poset $\Pca$, we
describe the associative elements of $\As(\Pca)$
(Proposition~\ref{prop:associative_elements}) and show that $\As(\Pca)$
is basic if and only if $\Pca$ is a trivial poset
(Proposition~\ref{prop:basic_operad}). We end this section by describing
algebras over the operads $\As(\Pca)$, called
{\em $\Pca$-associative algebras}, and by introducing a notion of units
for these algebras accompanied with certain properties
(Proposition~\ref{prop:units}).
\medskip

In Section~\ref{sec:forest_posets}, we focus on the case where the poset
$\Pca$ as input of the construction $\As$ is a forest poset. We show
that in this case, $\As(\Pca)$ is a Koszul operad
(Theorem~\ref{thm:koszulity}) and derive some consequences. In
particular, we give a functional equation for the Hilbert series of
$\As(\Pca)$ (Proposition~\ref{prop:hilbert_series}) and provide a
combinatorial realization of $\As(\Pca)$ involving Schröder
trees~\cite{Sta11} endowed with a labeling satisfying some constraints
(Theorem~\ref{thm:realization}). We next use this realization of
$\As(\Pca)$ to describe free algebras over $\As(\Pca)$ over one generator.
We end this section by establishing two presentations for the Koszul
dual of $\As(\Pca)$ (Propositions~\ref{prop:presentation_koszul_dual}
and~\ref{prop:alternative_presentation_koszul_dual}).
\medskip

This work ends by introducing in Section~\ref{sec:thin_forest_posets}
a class of posets whose elements are called {\em thin forest posets}.
The construction $\As$ restricted to this class of posets has the
property to be closed under Koszul duality. This property relies upon
an involution $^\perp$ on thin forest posets. We show
(Theorem~\ref{thm:isomorphism_thin_forests_dual_operad}) that if $\Pca$
is a thin forest poset, there is a thin forest poset $\Pca^\perp$ such
that $\As(\Pca)^!$ and $\As(\Pca^\perp)$ are isomorphic. This addresses
the questions previously raised.
\bigskip

{\it Notations and general conventions.}
All the algebraic structures of this article have a field of
characteristic zero $\K$ as ground field. For any integers $a$ and $c$,
$[a, c]$ denotes the set $\{b \in \N : a \leq b \leq c\}$ and $[n]$, the
set $[1, n]$. The cardinality of a finite set $S$ is denoted by $\# S$.
If $u$ is a word, its letters are indexed from left to right from $1$ to
its length $|u|$ and for any $i \in [|u|]$, $u_i$ is the letter of $u$ at
position $i$. If $\Op$ is a generator of an operad $\Oca$, we denote by
$\OpDual$ the associated generator in the Koszul dual of $\Oca$.
\medskip

{\it Acknowledgements.} The author warmly thanks the referees for their
very careful reading and their suggestions, improving the quality of the
paper.
\medskip

\section{Elementary definitions and lemmas}%
\label{sec:definitions_lemmas}
In this preliminary section, basic notions about posets, trees, and
operads are recalled. In particular, we describe the notions of patterns
in posets and pattern avoidance, rewrite rules on trees, and Koszulity
of operads. Besides, this section contains four lemmas, used further in
the text.
\medskip

\subsection{Posets and poset patterns}
Unless otherwise specified, we shall use in the sequel the standard
terminology about posets~\cite{Sta11}. For the sake of completeness,
let us recall the most important definitions and set our notations.
\medskip

\subsubsection{Posets}
Let $(\Pca, \Ord_\Pca)$ be a poset. We denote by $\OrdStrict_\Pca$ the
strict order relation obtained from $\Ord_\Pca$ and by $\not \Ord_\Pca$
the complementary binary relation of $\Ord_\Pca$. We denote by
$\Min_\Pca$ the binary partial operation $\min$ on $\Pca$. All posets
considered in this work are finite, so that $\Pca$ is a finite set and
has a cardinality, called {\em size} of $\Pca$.
\medskip

A {\em chain} of $\Pca$ is a subset $\{c_1, \dots, c_\ell\}$ of $\Pca$
such that $c_i \OrdStrict_\Pca c_{i + 1}$ for all $i \in [\ell - 1]$. An
{\em antichain} of $\Pca$ is a subset $\{a_1, \dots, a_\ell\}$ of $\Pca$
such that $a_i \Ord_\Pca a_j$ implies $i = j$ for all $i, j \in [\ell]$.
If $a$ and $c$ are two elements of $\Pca$ such that $a \Ord_\Pca c$,
the {\em interval} $[a, c]$ is the set
$\{b \in \Pca : a \Ord_\Pca b \Ord_\Pca c\}$. The number of intervals of
$\Pca$ is denoted by $\NbInterv(\Pca)$. An {\em order filter} of $\Pca$
is a subset $\Fca$ of $\Pca$ such that for all $a \in \Fca$ and all
$b \in \Pca$ satisfying $a \Ord_\Pca b$, $b$ is in $\Fca$. A
{\em subposet} of $\Pca$ is a subset $\Pca'$ of $\Pca$ endowed with the
order relation of $\Pca$ restricted to $\Pca'$.
\medskip

We shall define posets $\Pca$ by drawing Hasse diagrams, where minimal
elements are drawn uppermost and vertices are labeled by the elements
of $\Pca$. For instance, the Hasse diagram
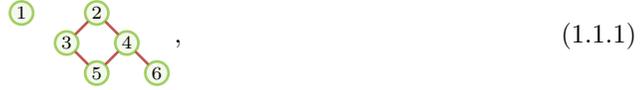
\begin{equation}
    \begin{tikzpicture}
        [baseline=(current bounding box.center),xscale=.4,yscale=.4]
        \node[Vertex](1)at(-1.5,1){\begin{math}1\end{math}};
        \node[Vertex](2)at(1,1){\begin{math}2\end{math}};
        \node[Vertex](3)at(0,0){\begin{math}3\end{math}};
        \node[Vertex](4)at(2,0){\begin{math}4\end{math}};
        \node[Vertex](5)at(1,-1){\begin{math}5\end{math}};
        \node[Vertex](6)at(3,-1){\begin{math}6\end{math}};
        \draw[Edge](2)--(3);
        \draw[Edge](2)--(4);
        \draw[Edge](3)--(5);
        \draw[Edge](4)--(5);
        \draw[Edge](4)--(6);
    \end{tikzpicture}\,,
\end{equation}
denotes the poset $([6], \Ord)$ satisfying among others $3 \Ord 5$
and~$2 \Ord 6$.
\medskip

Let $(\Pca_1, \Ord_{\Pca_1})$ and $(\Pca_2, \Ord_{\Pca_1})$ be two
posets. A {\em morphism of posets} is a map $\phi : \Pca_1 \to \Pca_2$
such that $x \Ord_{\Pca_1} y$ implies $\phi(x) \Ord_{\Pca_2} \phi(y)$.
\medskip

\begin{Lemma} \label{lem:morphism_posets_min}
    Let $\Pca_1$ and $\Pca_2$ be two posets and
    $\phi : \Pca_1 \to \Pca_2$ be a morphism of posets. Then, for all
    comparable elements $a$ and $b$ of $\Pca_1$,
    \begin{equation}
        \phi\left(a \Min_{\Pca_1} b\right) =
        \phi(a) \Min_{\Pca_2} \phi(b).
    \end{equation}
\end{Lemma}
\medskip

This easy property pointed out by Lemma~\ref{lem:morphism_posets_min}
intervenes later to show that our construction from posets to operads is
functorial.
\medskip

\subsubsection{Poset patterns}
Let $(\Pca_1, \Ord_{\Pca_1})$ and $(\Pca_2, \Ord_{\Pca_1})$ be two
posets. We say that $\Pca_1$ admits an {\em occurrence} of (the
{\em pattern}) $\Pca_2$ if there is an isomorphism of posets
$\phi : \Pca_1' \to \Pca_2$ where $\Pca_1'$ is a subposet of $\Pca_1$.
Conversely, we say that $\Pca_1$ {\em avoids} $\Pca_2$ if there is no
occurrence of $\Pca_2$ in~$\Pca_1$. Since only the isomorphism class of
a pattern is important to say if a poset admits an occurrence of it, we
shall draw unlabeled Hasse diagrams to specify patterns. For instance,
the poset
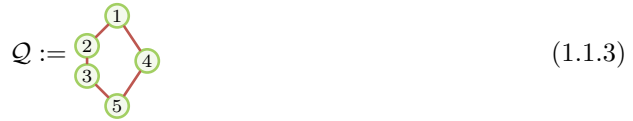
\begin{equation}
    \Pca :=
    \begin{tikzpicture}
        [baseline=(current bounding box.center),xscale=.4,yscale=.4]
        \node[Vertex](1)at(0,0){\begin{math}1\end{math}};
        \node[Vertex](2)at(-1,-1){\begin{math}2\end{math}};
        \node[Vertex](3)at(-1,-2){\begin{math}3\end{math}};
        \node[Vertex](4)at(1,-1.5){\begin{math}4\end{math}};
        \node[Vertex](5)at(0,-3){\begin{math}5\end{math}};
        \draw[Edge](1)--(2);
        \draw[Edge](2)--(3);
        \draw[Edge](3)--(5);
        \draw[Edge](1)--(4);
        \draw[Edge](4)--(5);
    \end{tikzpicture}
\end{equation}
admits an occurrence of the pattern
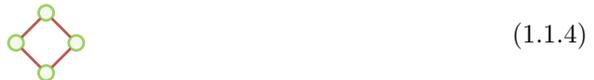
\begin{equation}
    \begin{tikzpicture}
        [baseline=(current bounding box.center),xscale=.4,yscale=.4]
        \node[Vertex](1)at(1,1){};
        \node[Vertex](2)at(0,0){};
        \node[Vertex](3)at(2,0){};
        \node[Vertex](4)at(1,-1){};
        \draw[Edge](1)--(2);
        \draw[Edge](1)--(3);
        \draw[Edge](2)--(4);
        \draw[Edge](3)--(4);
    \end{tikzpicture}
\end{equation}
since $1 \Ord_\Pca  3$, $1 \Ord_\Pca 4$, $3 \Ord_\Pca 5$, and
$4 \Ord_\Pca 5$. Moreover, $\Pca$ avoids the pattern
\begin{equation}
    \begin{tikzpicture}
        [baseline=(current bounding box.center),xscale=.4,yscale=.4]
        \node[Vertex](1)at(0,0){};
        \node[Vertex](2)at(1,0){};
        \node[Vertex](3)at(2,0){};
    \end{tikzpicture}
\end{equation}
since $\Pca$ has no antichain of cardinality~$3$.
\medskip

We call {\em forest poset} any poset avoiding the pattern
$\AntiForestPattern$. In other words, a forest poset is a poset for
which its Hasse diagram is a forest of rooted trees (where roots are
minimal elements).
\medskip

\begin{Lemma} \label{lem:forest_poset_max_three_elements}
    Let $\Pca$ be a forest poset and $a$, $b$, and $c$ be three elements
    of $\Pca$ such that $a$ and $b$ are comparable and $b$ and $c$ are
    comparable. Then, $a \Min_\Pca b \Min_\Pca c$ is a well-defined
    element of~$\Pca$.
\end{Lemma}
\medskip

This easy property pointed out by Lemma~\ref{lem:forest_poset_max_three_elements}
is crucial to establish the Koszulity of some operads constructed further.
\medskip

\subsection{Trees and rewrite rules}
Unless otherwise specified, we use in the sequel the standard
terminology ({\em i.e.}, {\em node}, {\em edge}, {\em root}, {\em parent},
{\em child}, {\em ancestor}, {\em etc.}) about planar rooted
trees~\cite{Knu97}. For the sake of completeness, let us recall the
most important definitions and set our notations.
\medskip

\subsubsection{Trees}
Let $\Tfr$ be a planar rooted tree. The {\em arity} of a node of $\Tfr$
is its number of children. An {\em internal node} (resp. a {\em leaf})
of $\Tfr$ is a node with a nonzero (resp. null) arity. Internal nodes
can be {\em labeled}, that is, each internal node of a tree is associated
with an element of a certain set. Given an internal node $x$ of $\Tfr$,
due to the planarity of $\Tfr$, the children of $x$ are totally ordered
from left to right and are thus indexed from $1$ to the arity of $x$.
Similarly, the leaves of $\Tfr$ are totally ordered from left to right
and thus are indexed from $1$ to the number of its leaves. In our
graphical representations, each planar rooted tree is depicted so that
its root is the uppermost node. Since we consider in the sequel only
planar rooted trees, we shall call these simply {\em trees}.
\medskip

Let $\Tfr$ be a tree. A tree $\Sfr$ is a {\em bottom subtree} of $\Tfr$
if there exists a node $x$ of $\Tfr$ such that the tree consisting in
the nodes of $\Tfr$ having $x$ as ancestor is $\Sfr$. A tree $\Sfr$ is
a {\em top subtree} of~$\Tfr$ if $\Sfr$ can be obtained from $\Tfr$ by
replacing certain of nodes of $\Tfr$ by leaves and by forgetting their
descendants. A tree $\Sfr$ is a {\em middle subtree} of $\Tfr$ if $\Sfr$
is a top subtree of a bottom subtree of~$\Tfr$. The {\em position}
$\Pos_\Tfr(x)$ of a node $x$ of $\Tfr$ is the word of positive integers
recursively defined by $\Pos_\Tfr(x) := \epsilon$ if $x$ is the root of
$\Tfr$ and otherwise
\begin{math}
    \Pos_\Tfr(x) := i\, \Pos_\Sfr(x),
\end{math}
where $x$ is in the bottom subtree $\Sfr$ rooted at the $i$th child of
the root of $\Tfr$.
\medskip

If $\Tfr$ is a labeled tree wherein all internal nodes have exactly two
children, the {\em infix reading word} of $\Tfr$ is the word obtained
recursively by concatenating the infix reading word of the bottom
subtree of $\Tfr$ rooted at the first child of its root, the label
of the root of $\Tfr$, and the infix reading word of the bottom
subtree of $\Tfr$ rooted at the second child of its root.
\medskip

Let $\Tfr$ and $\Sfr$ be two trees. An {\em occurrence} of (the
{\em pattern}) $\Sfr$ in $\Tfr$ is a position $u$ of a node $x$ of $\Tfr$
such that $\Sfr$ is a middle subtree of $\Tfr$ rooted at $x$.
Conversely, we say that $\Tfr$ {\em avoids} $\Sfr$ if there is no
occurrence of $\Sfr$ in $\Tfr$. Let $\Sfr'$ be a tree with the same
number of leaves as $\Sfr$ and $u$ be an occurrence of $\Sfr$ in $\Tfr$.
The {\em replacement} of the occurrence $u$ of $\Sfr$ by $\Sfr'$ in
$\Tfr$ is the tree obtained by replacing the middle subtree of $\Tfr$
rooted at a node of position $u$ equal to $\Sfr$ by $\Sfr'$.
\medskip

\subsubsection{Rewrite rules}
Let $S$ be a set of trees. A {\em rewrite rule} on $S$ is a binary
relation $\Rew'$ on $S$ whenever for all trees $\Sfr$ and $\Sfr'$ of $S$,
$\Sfr \Rew' \Sfr'$ only if $\Sfr$ and $\Sfr'$ have the same number of
leaves. We say that a tree $\Tfr$ is {\em rewritable in one step} by
$\Rew'$ into $\Tfr'$ if there exist two trees $\Sfr$ and $\Sfr'$
satisfying $\Sfr \Rew' \Sfr'$ and $\Tfr$ admits an occurrence $u$ of
$\Sfr$ such that the tree obtained by replacing the occurrence $u$ of
$\Sfr$ by $\Sfr'$ in $\Tfr$ is $\Tfr'$. We denote by $\Tfr \Rew \Tfr'$
this property, so that $\Rew$ is a binary relation on $S$. When
$\Tfr = \Tfr'$ or when there exists a sequence of trees
$(\Tfr_1, \dots, \Tfr_{k - 1})$ with $k \geq 1$ such that
\begin{math}
    \Tfr \Rew \Tfr_1 \Rew \cdots \Rew \Tfr_{k - 1} \Rew \Tfr',
\end{math}
we say that $\Tfr$ is {\em rewritable} by $\Rew$ into $\Tfr'$ and we
denote this property by $\Tfr \RewTrans \Tfr'$. In other words,
$\RewTrans$ is the reflexive and transitive closure of $\Rew$. We denote
by $\RewEquiv$ the reflexive, transitive, and symmetric closure of
$\Rew$. The {\em vector space induced} by $\Rew'$ is the subspace of the linear
span of all trees of $S$ generated by the family of all $\Tfr - \Tfr'$
such that~$\Tfr \RewEquiv \Tfr'$.
\medskip

For instance, let $S$ be the set of all trees where internal nodes are
labeled on $\{\La, \Lb, \Lc\}$ and consider the rewrite rule $\Rew'$ on
$S$ satisfying
\begin{equation}
    \begin{split}
    \begin{tikzpicture}[xscale=.4,yscale=.23]
        \node(0)at(0.00,-2.00){};
        \node(2)at(1.00,-2.00){};
        \node(3)at(2.00,-2.00){};
        \node(1)at(1.00,0.00){\begin{math}\Lb\end{math}};
        \draw(0)--(1);
        \draw(2)--(1);
        \draw(3)--(1);
        \node(r)at(1.00,1.75){};
        \draw(r)--(1);
    \end{tikzpicture}
    \end{split}
    \begin{split} \enspace \Rew' \enspace \end{split}
    \begin{split}
    \begin{tikzpicture}[xscale=.27,yscale=.25]
        \node(0)at(0.00,-3.33){};
        \node(2)at(2.00,-3.33){};
        \node(4)at(4.00,-1.67){};
        \node(1)at(1.00,-1.67){\begin{math}\La\end{math}};
        \node(3)at(3.00,0.00){\begin{math}\La\end{math}};
        \draw(0)--(1);
        \draw(1)--(3);
        \draw(2)--(1);
        \draw(4)--(3);
        \node(r)at(3.00,1.5){};
        \draw(r)--(3);
    \end{tikzpicture}
    \end{split}
    \qquad \mbox{and} \qquad
    \begin{split}
    \begin{tikzpicture}[xscale=.27,yscale=.25]
        \node(0)at(0.00,-3.33){};
        \node(2)at(2.00,-3.33){};
        \node(4)at(4.00,-1.67){};
        \node(1)at(1.00,-1.67){\begin{math}\La\end{math}};
        \node(3)at(3.00,0.00){\begin{math}\Lc\end{math}};
        \draw(0)--(1);
        \draw(1)--(3);
        \draw(2)--(1);
        \draw(4)--(3);
        \node(r)at(3.00,1.5){};
        \draw(r)--(3);
    \end{tikzpicture}
    \end{split}
    \begin{split} \enspace \Rew' \enspace \end{split}
    \begin{split}
    \begin{tikzpicture}[xscale=.27,yscale=.25]
        \node(0)at(0.00,-1.67){};
        \node(2)at(2.00,-3.33){};
        \node(4)at(4.00,-3.33){};
        \node(1)at(1.00,0.00){\begin{math}\La\end{math}};
        \node(3)at(3.00,-1.67){\begin{math}\Lc\end{math}};
        \draw(0)--(1);
        \draw(2)--(3);
        \draw(3)--(1);
        \draw(4)--(3);
        \node(r)at(1.00,1.5){};
        \draw(r)--(1);
    \end{tikzpicture}
    \end{split}\,.
\end{equation}
We then have the following steps of rewritings by $\Rew'$:
\begin{equation}
    \begin{split}
    \begin{tikzpicture}[xscale=.2,yscale=.16]
        \node(0)at(0.00,-6.50){};
        \node(10)at(8.00,-6.50){};
        \node(12)at(10.00,-6.50){};
        \node(2)at(1.00,-9.75){};
        \node(4)at(3.00,-9.75){};
        \node(5)at(4.00,-9.75){};
        \node(7)at(5.00,-9.75){};
        \node(8)at(6.00,-9.75){};
        \node[text=BrickRed](1)at(2.00,-3.25){\begin{math}\Lb\end{math}};
        \node(11)at(9.00,-3.25){\begin{math}\La\end{math}};
        \node(3)at(2.00,-6.50){\begin{math}\Lc\end{math}};
        \node(6)at(5.00,-6.50){\begin{math}\Lb\end{math}};
        \node(9)at(7.00,0.00){\begin{math}\Lc\end{math}};
        \draw[draw=BrickRed](0)--(1);
        \draw[draw=BrickRed](1)--(9);
        \draw(10)--(11);
        \draw(11)--(9);
        \draw(12)--(11);
        \draw(2)--(3);
        \draw[draw=BrickRed](3)--(1);
        \draw(4)--(3);
        \draw(5)--(6);
        \draw[draw=BrickRed](6)--(1);
        \draw(7)--(6);
        \draw(8)--(6);
        \node(r)at(7.00,2.5){};
        \draw(r)--(9);
    \end{tikzpicture}
    \end{split}
    \begin{split} \enspace \Rew \enspace \end{split}
    \begin{split}
    \begin{tikzpicture}[xscale=.18,yscale=.18]
        \node(0)at(0.00,-8.40){};
        \node(11)at(10.00,-5.60){};
        \node(13)at(12.00,-5.60){};
        \node(2)at(2.00,-11.20){};
        \node(4)at(4.00,-11.20){};
        \node(6)at(6.00,-8.40){};
        \node(8)at(7.00,-8.40){};
        \node(9)at(8.00,-8.40){};
        \node(1)at(1.00,-5.60){\begin{math}\La\end{math}};
        \node[text=BrickRed](10)at(9.00,0.00){\begin{math}\Lc\end{math}};
        \node(12)at(11.00,-2.80){\begin{math}\La\end{math}};
        \node(3)at(3.00,-8.40){\begin{math}\Lc\end{math}};
        \node[text=BrickRed](5)at(5.00,-2.80){\begin{math}\La\end{math}};
        \node(7)at(7.00,-5.60){\begin{math}\Lb\end{math}};
        \draw(0)--(1);
        \draw[draw=BrickRed](1)--(5);
        \draw(11)--(12);
        \draw[draw=BrickRed](12)--(10);
        \draw(13)--(12);
        \draw(2)--(3);
        \draw(3)--(1);
        \draw(4)--(3);
        \draw[draw=BrickRed](5)--(10);
        \draw(6)--(7);
        \draw[draw=BrickRed](7)--(5);
        \draw(8)--(7);
        \draw(9)--(7);
        \node(r)at(9.00,2.25){};
        \draw[draw=BrickRed](r)--(10);
    \end{tikzpicture}
    \end{split}
    \begin{split} \enspace \Rew \enspace \end{split}
    \begin{split}
    \begin{tikzpicture}[xscale=.18,yscale=.15]
        \node(0)at(0.00,-7.00){};
        \node(11)at(10.00,-10.50){};
        \node(13)at(12.00,-10.50){};
        \node(2)at(2.00,-10.50){};
        \node(4)at(4.00,-10.50){};
        \node(6)at(6.00,-10.50){};
        \node(8)at(7.00,-10.50){};
        \node(9)at(8.00,-10.50){};
        \node(1)at(1.00,-3.50){\begin{math}\La\end{math}};
        \node(10)at(9.00,-3.50){\begin{math}\Lc\end{math}};
        \node(12)at(11.00,-7.00){\begin{math}\La\end{math}};
        \node(3)at(3.00,-7.00){\begin{math}\Lc\end{math}};
        \node(5)at(5.00,0.00){\begin{math}\La\end{math}};
        \node[text=BrickRed](7)at(7.00,-7.00){\begin{math}\Lb\end{math}};
        \draw(0)--(1);
        \draw(1)--(5);
        \draw(10)--(5);
        \draw(11)--(12);
        \draw(12)--(10);
        \draw(13)--(12);
        \draw(2)--(3);
        \draw(3)--(1);
        \draw(4)--(3);
        \draw[draw=BrickRed](6)--(7);
        \draw[draw=BrickRed](7)--(10);
        \draw[draw=BrickRed](8)--(7);
        \draw[draw=BrickRed](9)--(7);
        \node(r)at(5.00,2.75){};
        \draw(r)--(5);
    \end{tikzpicture}
    \end{split}
    \begin{split} \enspace \Rew \enspace \end{split}
    \begin{split}
    \begin{tikzpicture}[xscale=.2,yscale=.16]
        \node(0)at(0.00,-6.00){};
        \node(10)at(10.00,-9.00){};
        \node(12)at(12.00,-9.00){};
        \node(14)at(14.00,-9.00){};
        \node(2)at(2.00,-9.00){};
        \node(4)at(4.00,-9.00){};
        \node(6)at(6.00,-12.00){};
        \node(8)at(8.00,-12.00){};
        \node(1)at(1.00,-3.00){\begin{math}\La\end{math}};
        \node(11)at(11.00,-3.00){\begin{math}\Lc\end{math}};
        \node(13)at(13.00,-6.00){\begin{math}\La\end{math}};
        \node(3)at(3.00,-6.00){\begin{math}\Lc\end{math}};
        \node(5)at(5.00,0.00){\begin{math}\La\end{math}};
        \node(7)at(7.00,-9.00){\begin{math}\La\end{math}};
        \node(9)at(9.00,-6.00){\begin{math}\La\end{math}};
        \draw(0)--(1);
        \draw(1)--(5);
        \draw(10)--(9);
        \draw(11)--(5);
        \draw(12)--(13);
        \draw(13)--(11);
        \draw(14)--(13);
        \draw(2)--(3);
        \draw(3)--(1);
        \draw(4)--(3);
        \draw(6)--(7);
        \draw(7)--(9);
        \draw(8)--(7);
        \draw(9)--(11);
        \node(r)at(5.00,2.25){};
        \draw(r)--(5);
    \end{tikzpicture}
    \end{split}\,.
\end{equation}
\medskip

We shall use the standard terminology ({\em terminating},
{\em normal form}, {\em confluent}, {\em convergent},
{\em critical pair}, {\em etc.}) about rewrite rules~\cite{BN98}. Let us
recall the most important definitions. Let $\Rew'$ be a rewrite rule.
The {\em degree} of $\Rew'$ is the maximal number of internal nodes of
the trees appearing as left members of $\Rew'$. Besides, $\Rew'$ is
{\em terminating} if there is no infinite chain
$\Tfr \Rew \Tfr_1 \Rew \Tfr_2 \Rew \cdots$. In this case, any tree $\Tfr$
of $S$ that that cannot be rewritten by $\Rew'$ is a {\em normal form}
for $\Rew'$. We say that $\Rew'$ is {\em confluent} if for any trees
$\Tfr$, $\Rfr_1$, and $\Rfr_2$ such that $\Tfr \RewTrans \Rfr_1$ and
$\Tfr \RewTrans \Rfr_2$, there exists a tree $\Tfr'$ such that
$\Rfr_1 \RewTrans \Tfr'$ and $\Rfr_2 \RewTrans \Tfr'$.
We call {\em critical tree} of $\Rew'$ any tree $\Tfr$ such
that there exist two different trees $\Rfr_1$ and $\Rfr_2$ satisfying
$\Tfr \Rew \Rfr_1$ and $\Tfr \Rew \Rfr_2$. In this case, the pair
$\{\Rfr_1, \Rfr_2\}$ is a {\em critical pair} of $\Rew'$ associated with
$\Tfr$. Moreover, the critical pair $\{\Rfr_1, \Rfr_2\}$ is
{\em joinable} if there exists a tree $\Tfr'$ such that
$\Rfr_1 \RewTrans \Tfr'$ and $\Rfr_2 \RewTrans \Tfr'$.
By the diamond lemma~\cite{New42}, when $\Rew'$ is terminating and when
for any tree $\Tfr$, all critical pairs associated with $\Tfr$ are
joinable, $\Rew'$ is confluent. When $\Rew'$ is both terminating and
confluent, we say that $\Rew'$ is {\em convergent}.
\medskip

\begin{Lemma} \label{lem:confluent_few_nodes}
    Let $\Rew'$ be a terminating rewrite rule of degree $\ell$. Then,
    $\Rew'$ is confluent if and only if all critical pairs of $\Rew'$
    consisting in trees with $2 \ell - 1$ internal nodes or less are
    joinable.
\end{Lemma}
\begin{proof}
    If $\Rew'$ is confluent, by definition of the confluence, all
    critical pairs of $\Rew'$ are joinable. In particular, critical
    pairs of trees with $2 \ell - 1$ internal nodes of less are joinable.
    \smallskip

    Conversely, assume that all critical pairs of trees with $2 \ell - 1$
    internal nodes of less are joinable. Let $\Tfr$ be a critical tree
    of $\Rew'$ and $\{\Rfr_1, \Rfr_2\}$ be a critical pair associated
    with $\Tfr$. We have thus $\Tfr \Rew \Rfr_1$ and $\Tfr \Rew \Rfr_2$.
    By definition of $\Rew'$, there are four trees $\Tfr'_1$, $\Rfr'_1$,
    $\Tfr'_2$, and $\Rfr'_2$ such that $\Tfr'_1 \Rew' \Rfr'_1$,
    $\Tfr'_2 \Rew' \Rfr'_2$, and $\Rfr_1$ (resp. $\Rfr_2$) is obtained
    by replacing an occurrence $u_1$ (resp. $u_2$) of $\Tfr'_1$
    (resp. $\Tfr'_2$) by $\Rfr'_1$ (resp. $\Rfr'_2$) in $\Tfr$. We have
    now two cases to consider, depending on the positions $u_1$ and $u_2$
    of $\Tfr'_1$ and $\Tfr'_2$ in $\Tfr$.
    \begin{enumerate}[fullwidth,label={\it Case \arabic*.}]
        \item If the occurrences of $\Tfr'_1$ and $\Tfr'_2$ at positions
        $u_1$ and $u_2$ in $\Tfr$ do not share any internal node of
        $\Tfr$, $\Rfr_1$ (resp. $\Rfr_2$) admits the considered
        occurrence of $\Tfr'_2$ (resp. $\Tfr'_1$). For these reasons, let
        $\Tfr'$ be the tree obtained by replacing the considered
        occurrence of $\Tfr'_2$ by $\Rfr'_2$ in $\Rfr_1$. Equivalently,
        since the considered occurrences of $\Tfr'_1$ and $\Tfr'_2$
        do not overlap in $\Tfr$, $\Tfr'$ is the tree obtained by
        replacing the considered occurrence of $\Tfr'_1$ by $\Rfr'_1$ in
        $\Rfr_2$. Thereby, we have $\Rfr_1 \Rew \Tfr'$ and
        $\Rfr_2 \Rew \Tfr'$, showing that the critical pair
        $\{\Rfr_1, \Rfr_2\}$ is joinable.
        \smallskip

        \item Otherwise, the occurrences of $\Tfr'_1$ and $\Tfr'_2$ at
        positions $u_1$ and $u_2$ in $\Tfr$ share at least one internal
        node of $\Tfr$. Denote by $\Sfr$ the middle subtree of $\Tfr$
        consisting in the smallest number of internal nodes such that
        all the internal nodes of the considered occurrences of $\Tfr'_1$
        and $\Tfr'_2$ in $\Tfr$ are in $\Sfr$. Let $\Sfr_1$ be the tree
        obtained by replacing the considered occurrence of $\Tfr'_1$ by
        $\Rfr'_1$ in $\Sfr$  and let $\Sfr_2$ be the tree obtained by
        replacing the considered  occurrence of $\Tfr'_2$ by $\Rfr'_2$
        in $\Sfr$. Since the occurrences of $\Tfr'_1$ and $\Tfr'_2$
        overlap in $\Tfr$, by the minimality of the number of nodes of
        $\Sfr$, $\Sfr$ has at most $2 \ell - 1$ internal nodes. Now,
        since we have $\Sfr \Rew \Sfr_1$ and $\Sfr \Rew \Sfr_2$, $\Sfr$
        is a critical tree of $\Rew'$ and $\{\Sfr_1, \Sfr_2\}$ is a
        critical pair associated with $\Sfr$. By hypothesis,
        $\{\Sfr_1, \Sfr_2\}$ is joinable, so that there is a tree $\Sfr'$
        such that $\Sfr_1 \RewTrans \Sfr'$ and $\Sfr_2 \RewTrans \Sfr'$.
        By setting $\Tfr'$ as the tree obtained by replacing the
        considered occurrence of $\Sfr$ by $\Sfr'$ in $\Tfr$, we finally
        have $\Rfr_1 \RewTrans \Tfr'$ and $\Rfr_2 \RewTrans \Tfr'$,
        showing that $\{\Rfr_1, \Rfr_2\}$ is joinable.
    \end{enumerate}
    We have shown that all critical pairs of $\Rew'$ are joinable.
    Since $\Rew'$ is terminating, this implies by the diamond
    lemma~\cite{New42} that $\Rew'$ is a confluent rewrite rule.
\end{proof}
\medskip

Lemma~\ref{lem:confluent_few_nodes} provides a criterion to prove that
a terminating rewrite rule $\Rew'$ is confluent by considering only
critical trees and associated critical pairs with a bounded number of
internal nodes. We shall work with rewrite rules arising in the context
of quadratic operads, and thus with degree~$2$. In this case, we will
consider only critical trees with at most three internal nodes.
\medskip

\subsection{Operads and Koszulity}
We adopt most of notations and conventions of~\cite{LV12} about operads.
For the sake of completeness, we recall here the elementary notions
about operads employed thereafter.
\medskip

\subsubsection{Nonsymmetric operads}
A {\em nonsymmetric operad in the category of vector spaces}, or a
{\em nonsymmetric operad} for short, is a graded vector space
$\Oca := \bigoplus_{n \geq 1} \Oca(n)$ together with linear maps
\begin{equation}
    \circ_i : \Oca(n) \otimes \Oca(m) \to \Oca(n + m - 1),
    \qquad n, m \geq 1, i \in [n],
\end{equation}
called {\em partial compositions}, and a distinguished element
$\Unit \in \Oca(1)$, the {\em unit} of $\Oca$. This data has to satisfy
the three relations
\begin{subequations}
\begin{equation} \label{equ:operad_axiom_1}
    (x \circ_i y) \circ_{i + j - 1} z = x \circ_i (y \circ_j z),
    \qquad x \in \Oca(n), y \in \Oca(m),
    z \in \Oca(k), i \in [n], j \in [m],
\end{equation}
\begin{equation} \label{equ:operad_axiom_2}
    (x \circ_i y) \circ_{j + m - 1} z = (x \circ_j z) \circ_i y,
    \qquad x \in \Oca(n), y \in \Oca(m),
    z \in \Oca(k), i < j \in [n],
\end{equation}
\begin{equation} \label{equ:operad_axiom_3}
    \Unit \circ_1 x = x = x \circ_i \Unit,
    \qquad x \in \Oca(n), i \in [n].
\end{equation}
\end{subequations}
Since we consider in this paper only nonsymmetric operads, we shall
call these simply {\em operads}. Moreover, in this work, we shall only
consider operads $\Oca$ where $\Oca(1)$ has dimension $1$.
\medskip

If $x$ is an element of $\Oca$ such that $x \in \Oca(n)$ for a $n \geq 1$,
we say that $n$ is the {\em arity} of $x$ and we denote it by $|x|$. An
element $x$ of $\Oca$ of arity $2$ is {\em associative} if
$x \circ_1 x = x \circ_2 x$. If $\Oca_1$ and $\Oca_2$ are two operads, a
linear map $\phi : \Oca_1 \to \Oca_2$ is an {\em operad morphism} if it
respects arities, sends the unit of $\Oca_1$ to the unit of $\Oca_2$,
and commutes with partial composition maps. We say that $\Oca_2$ is a
{\em suboperad} of $\Oca_1$ if $\Oca_2$ is a graded subspace of $\Oca_1$,
and $\Oca_1$ and $\Oca_2$ have the same unit and the same partial
compositions. For any set $G \subseteq \Oca$, the {\em operad generated
by} $G$ is the smallest suboperad of $\Oca$ containing $G$. When the
operad generated by $G$ is $\Oca$ itself and $G$ is minimal with respect
to inclusion among the subsets of $\Oca$ satisfying this property, $G$
is a {\em generating set} of $\Oca$ and its elements are {\em generators}
of $\Oca$. An {\em operad ideal} of $\Oca$ is a graded subspace $I$ of
$\Oca$ such that, for any $x \in \Oca$ and $y \in I$, $x \circ_i y$ and
$y \circ_j x$ are in $I$ for all valid integers $i$ and $j$. Given an
operad ideal $I$ of $\Oca$, one can define the {\em quotient operad}
$\Oca/_I$ of $\Oca$ by $I$ in the usual way. When $\Oca$ is such that
all $\Oca(n)$ are finite for all $n \geq 1$, the {\em Hilbert series} of
$\Oca$ is the series $\Hca_\Oca(t)$ defined by
\begin{equation}
    \Hca_\Oca(t) := \sum_{n \geq 1} \dim \Oca(n)\, t^n.
\end{equation}
\medskip

\subsubsection{Syntax trees and free operads}
Let $\FreeGen := \sqcup_{n \geq 1} \FreeGen(n)$ be a graded set. We say
that the {\em arity} of an element $x$ of $\FreeGen$ is $n$ provided
that $x \in \FreeGen(n)$. A {\em syntax tree on $\FreeGen$} is a planar
rooted tree such that its internal nodes of arity $n$ are labeled on
elements of arity $n$ of~$\FreeGen$. The {\em degree} (resp. {\em arity})
of a syntax tree is its number of internal nodes (resp. leaves). For
instance, if $\FreeGen := \FreeGen(2) \sqcup \FreeGen(3)$ with
$\FreeGen(2) := \{\La, \Lc\}$ and $\FreeGen(3) := \{\Lb\}$,
\begin{equation}
    \begin{split}
    \begin{tikzpicture}[xscale=.26,yscale=.15]
        \node(0)at(0.00,-6.50){};
        \node(10)at(8.00,-9.75){};
        \node(12)at(10.00,-9.75){};
        \node(2)at(2.00,-6.50){};
        \node(4)at(3.00,-3.25){};
        \node(5)at(4.00,-9.75){};
        \node(7)at(5.00,-9.75){};
        \node(8)at(6.00,-9.75){};
        \node(1)at(1.00,-3.25){\begin{math}\Lc\end{math}};
        \node(11)at(9.00,-6.50){\begin{math}\La\end{math}};
        \node(3)at(3.00,0.00){\begin{math}\Lb\end{math}};
        \node(6)at(5.00,-6.50){\begin{math}\Lb\end{math}};
        \node(9)at(7.00,-3.25){\begin{math}\La\end{math}};
        \node(r)at(3.00,2.75){};
        \draw(0)--(1); \draw(1)--(3); \draw(10)--(11); \draw(11)--(9);
        \draw(12)--(11); \draw(2)--(1); \draw(4)--(3); \draw(5)--(6);
        \draw(6)--(9); \draw(7)--(6); \draw(8)--(6); \draw(9)--(3);
        \draw(r)--(3);
    \end{tikzpicture}
    \end{split}
\end{equation}
is a syntax tree on $\FreeGen$ of degree $5$ and arity $8$. Its root is
labeled by $\Lb$ and has arity $3$.
\medskip

The {\em free operad over $\FreeGen$} is the operad  $\Free(\FreeGen)$
wherein for any $n \geq 1$, $\Free(\FreeGen)(n)$ is the linear span of
the syntax trees on $\FreeGen$ of arity $n$, the partial composition
$\Sfr \circ_i \Tfr$ of two syntax trees $\Sfr$ and $\Tfr$ on $\FreeGen$
consists in grafting the root of $\Tfr$ on the $i$th leaf of $\Sfr$, and
its unit is the tree consisting in one leaf. For instance, if
$\FreeGen := \FreeGen(2) \sqcup \FreeGen(3)$ with
$\FreeGen(2) := \{\La, \Lc\}$ and $\FreeGen(3) := \{\Lb\}$, one has in
$\Free(\FreeGen)$,
\begin{equation}\begin{split}\end{split}
    \begin{split}
    \begin{tikzpicture}[xscale=.28,yscale=.17]
        \node(0)at(0.00,-5.33){};
        \node(2)at(2.00,-5.33){};
        \node(4)at(4.00,-5.33){};
        \node(6)at(5.00,-5.33){};
        \node(7)at(6.00,-5.33){};
        \node[text=RoyalBlue](1)at(1.00,-2.67){\begin{math}\La\end{math}};
        \node[text=RoyalBlue](3)at(3.00,0.00){\begin{math}\La\end{math}};
        \node[text=RoyalBlue](5)at(5.00,-2.67){\begin{math}\Lb\end{math}};
        \node(r)at(3.00,2.25){};
        \draw[draw=RoyalBlue](0)--(1); \draw[draw=RoyalBlue](1)--(3);
        \draw[draw=RoyalBlue](2)--(1); \draw[draw=RoyalBlue](4)--(5);
        \draw[draw=RoyalBlue](5)--(3); \draw[draw=RoyalBlue](6)--(5);
        \draw[draw=RoyalBlue](7)--(5); \draw[draw=RoyalBlue](r)--(3);
    \end{tikzpicture}
    \end{split}
    \enspace \circ_3 \enspace
    \begin{split}
    \begin{tikzpicture}[xscale=.26,yscale=.23]
        \node(0)at(0.00,-1.67){};
        \node(2)at(2.00,-3.33){};
        \node(4)at(4.00,-3.33){};
        \node[text=BrickRed](1)at(1.00,0.00){\begin{math}\Lc\end{math}};
        \node[text=BrickRed](3)at(3.00,-1.67){\begin{math}\La\end{math}};
        \node(r)at(1.00,1.75){};
        \draw[draw=BrickRed](0)--(1); \draw[draw=BrickRed](2)--(3);
        \draw[draw=BrickRed](3)--(1); \draw[draw=BrickRed](4)--(3);
        \draw[draw=BrickRed](r)--(1);
    \end{tikzpicture}
    \end{split}
    \enspace = \enspace
    \begin{split}
    \begin{tikzpicture}[xscale=.25,yscale=.18]
        \node(0)at(0.00,-4.80){};
        \node(10)at(9.00,-4.80){};
        \node(11)at(10.00,-4.80){};
        \node(2)at(2.00,-4.80){};
        \node(4)at(4.00,-7.20){};
        \node(6)at(6.00,-9.60){};
        \node(8)at(8.00,-9.60){};
        \node[text=RoyalBlue](1)at(1.00,-2.40){\begin{math}\La\end{math}};
        \node[text=RoyalBlue](3)at(3.00,0.00){\begin{math}\La\end{math}};
        \node[text=BrickRed](5)at(5.00,-4.80){\begin{math}\Lc\end{math}};
        \node[text=BrickRed](7)at(7.00,-7.20){\begin{math}\La\end{math}};
        \node[text=RoyalBlue](9)at(9.00,-2.40){\begin{math}\Lb\end{math}};
        \node(r)at(3.00,2.40){};
        \draw[draw=RoyalBlue](0)--(1); \draw[draw=RoyalBlue](1)--(3);
        \draw[draw=RoyalBlue](10)--(9); \draw[draw=RoyalBlue](11)--(9);
        \draw[draw=RoyalBlue](2)--(1); \draw[draw=BrickRed](4)--(5);
        \draw[draw=RoyalBlue!50!BrickRed!50](5)--(9); \draw[draw=BrickRed](6)--(7);
        \draw[draw=BrickRed](7)--(5); \draw[draw=BrickRed](8)--(7);
        \draw[draw=RoyalBlue](9)--(3); \draw[draw=RoyalBlue](r)--(3);
    \end{tikzpicture}
    \end{split}\,.
\end{equation}
\medskip

We denote by $\Corolla : \FreeGen \to \Free(\FreeGen)$ the inclusion map,
sending any $x$ of $\FreeGen$ to the {\em corolla} labeled by $x$, that
is the syntax tree consisting in one internal node labeled by $x$
attached to a required number of leaves. In the sequel, if required by
the context, we shall implicitly see any element $x$ of $\FreeGen$ as
the corolla $\Corolla(x)$ of $\Free(\FreeGen)$. For instance, when $x$
and $y$ are two elements of $\FreeGen$, we shall simply denote by
$x \circ_i y$ the syntax tree $\Corolla(x) \circ_i \Corolla(y)$ for all
valid integers $i$.
\medskip

\subsubsection{Presentations by generators and relations}
A {\em presentation} of an operad $\Oca$ consists in a pair
$(\FreeGen, \FreeRel)$ such that
$\FreeGen := \sqcup_{n \geq 1} \FreeGen(n)$ is a graded set, $\FreeRel$
is a subspace of $\Free(\FreeGen)$, and  $\Oca$ is isomorphic to
$\Free(\FreeGen)/_{\langle \FreeRel \rangle}$,  where
$\langle \FreeRel \rangle$ is the operad ideal of $\Free(\FreeGen)$
generated by $\FreeRel$. We call $\FreeGen$ the {\em set of generators}
and $\FreeRel$ the {\em space of relations} of $\Oca$. We say that
$\Oca$ is {\em quadratic} if there is a presentation
$(\FreeGen, \FreeRel)$ of $\Oca$ such that $\FreeRel$ is a homogeneous
subspace of $\Free(\FreeGen)$ consisting in syntax trees of degree $2$.
Besides, we say that $\Oca$ is {\em binary} if there is a presentation
$(\FreeGen, \FreeRel)$ of $\Oca$ such that $\FreeGen$ is concentrated in
arity~$2$. Furthermore, if $\Oca$ admits a presentation
$(\FreeGen, \FreeRel)$ and $\Rew'$ is a rewrite rule on $\Free(\FreeGen)$
such that the space induced by $\Rew'$ is $\langle \FreeRel \rangle$, we
say that $\Rew'$ is an {\em orientation} of $\FreeRel$.
\medskip

\subsubsection{Koszul duality and Koszulity}%
\label{subsubsec:koszul_duality_koszulity_criterion}
In~\cite{GK94}, Ginzburg and Kapranov extended the notion of Koszul
duality of quadratic associative algebras to quadratic operads. Starting
with a binary and quadratic operad $\Oca$ admitting a presentation
$(\FreeGen, \FreeRel)$ where $\FreeGen$ is finite, the {\em Koszul dual}
of $\Oca$ is the operad $\Oca^!$, isomorphic to the operad admitting the
presentation $(\FreeGen, \FreeRel^\perp)$ where $\FreeRel^\perp$ is the
annihilator of $\FreeRel$ in $\Free(\FreeGen)$ with respect to the scalar
product
\begin{equation} \label{equ:scalar_product_koszul}
    \langle -, - \rangle :
    \Free(\FreeGen)(3) \otimes \Free(\FreeGen)(3) \to \K
\end{equation}
linearly defined, for all $x, x', y, y' \in \FreeGen(2)$, by
\begin{equation}
    \left\langle x \circ_i y, x' \circ_{i'} y' \right\rangle :=
    \begin{cases}
        1 & \mbox{if }
            x = x', y = y', \mbox{ and } i = i' = 1, \\
        -1 & \mbox{if }
            x = x', y = y', \mbox{ and } i = i' = 2, \\
        0 & \mbox{otherwise}.
    \end{cases}
\end{equation}
Then, with knowledge of a presentation of $\Oca$, one can compute a
presentation of $\Oca^!$.
\medskip

Besides, a quadratic operad $\Oca$ is {\em Koszul} if its Koszul complex
is acyclic~\cite{GK94,LV12}. Furthermore, when $\Oca$ is Koszul and
admits an Hilbert series, the Hilbert series of $\Oca$ and of its Koszul
dual $\Oca^!$ are related~\cite{GK94} by
\begin{equation}
    \Hca_\Oca\left(-\Hca_{\Oca^!}(-t)\right) = t.
\end{equation}
\medskip

In this work, to prove the Koszulity of an operad $\Oca$, we shall make
use of a tool introduced by Dotsenko and Khoroshkin~\cite{DK10} in
the context of Gröbner bases for operads, which reformulates in our
context, by using rewrite rules on syntax trees, in the following way.
\begin{Lemma} \label{lem:koszulity_criterion_pbw}
    Let $\Oca$ be an operad admitting a quadratic presentation
    $(\FreeGen, \FreeRel)$. If there exists an orientation $\Rew'$ of
    $\FreeRel$ such that $\Rew'$ is a convergent rewrite rule, then
    $\Oca$ is Koszul.
\end{Lemma}
\medskip

When $\Rew'$ satisfies the conditions contained in the statement of
Lemma~\ref{lem:koszulity_criterion_pbw}, the set of the normal forms
of $\Rew'$ forms a basis of $\Oca$, called
{\em Poincaré-Birkhoff-Witt basis} and arising in the work of
Hoffbeck~\cite{Hof10} (see also~\cite{LV12}).
\medskip

\subsubsection{Algebras over an operad}
Any operad $\Oca$ encodes a category of algebras whose objects are
called {\em $\Oca$-algebras}. An $\Oca$-algebra $\Alg_\Oca$ is a vector
space endowed with a left action
\begin{equation}
    \cdot : \Oca(n) \otimes \Alg_\Oca^{\otimes n} \to \Alg_\Oca,
    \qquad n \geq 1,
\end{equation}
satisfying the relations imposed by the structure of $\Oca$, that are
\begin{multline} \label{equ:algebra_over_operad}
    (x \circ_i y) \cdot (e_1 \otimes \dots \otimes e_{n + m - 1})
    = \\
    x \cdot \left(e_1 \otimes \dots
        \otimes e_{i - 1} \otimes
        y \cdot (e_i \otimes \dots \otimes e_{i + m - 1})
        \otimes e_{i + m} \otimes
        \dots \otimes e_{n + m - 1}\right),
\end{multline}
for all $x \in \Oca(n)$, $y \in \Oca(m)$, $i \in [n]$, and
\begin{math}
    e_1 \otimes \dots \otimes e_{n + m - 1}
    \in \Alg_\Oca^{\otimes {n + m - 1}}.
\end{math}
Notice that, by~\eqref{equ:algebra_over_operad}, if $G$ is a generating
set of $\Oca$, it is enough to define the action of each $x \in G$ on
$\Alg_\Oca^{\otimes |x|}$ to wholly define~$\cdot$.
\medskip

\section{From posets to operads} \label{sec:posets_to_operads}
This section is devoted to the introduction of our construction
producing an operad from a poset. We also establish here some of its
basic general properties. We end this section by presenting algebras
over our operads and some of their properties.
\medskip

\subsection{Construction}
Let us describe the construction $\As$, associating with any poset a
binary and quadratic operad presentation, and prove that it is functorial.
\medskip

\subsubsection{Operad presentations from posets}%
\label{subsubsec:definition_as}
For any poset $(\Pca, \Ord_\Pca)$, we define the
{\em $\Pca$-associative operad} $\As(\Pca)$ as the operad admitting the
presentation $(\FreeGen_\Pca^\Op, \FreeRel_\Pca^\Op)$ where
$\FreeGen_\Pca^\Op$ is the set of generators
\begin{equation}
    \FreeGen_\Pca^\Op := \FreeGen_\Pca^\Op(2) := \{\Op_a : a \in \Pca\},
\end{equation}
and $\FreeRel_\Pca^\Op$ is the space of relations generated by
\begin{subequations}
\begin{equation} \label{equ:relation_1}
    \Op_a \circ_1 \Op_b
    - \Op_{a \Min_\Pca b} \circ_2 \Op_{a \Min_\Pca b},
    \qquad
    a, b \in \Pca \mbox{ and }
        (a \Ord_\Pca b \mbox{ or } b \Ord_\Pca a),
\end{equation}
\begin{equation} \label{equ:relation_2}
    \Op_{a \Min_\Pca b} \circ_1 \Op_{a \Min_\Pca b}
    - \Op_a \circ_2 \Op_b,
    \qquad
    a, b \in \Pca \mbox{ and }
        (a \Ord_\Pca b \mbox{ or } b \Ord_\Pca a).
\end{equation}
\end{subequations}
By definition, $\As(\Pca)$ is a binary and quadratic operad.
\medskip

\begin{Lemma} \label{lem:relations}
    Let $\Pca$ be a poset. For any $a \in \Pca$, let $R_a$ be the set
    \begin{equation}
        R_a := \{\Op_a \circ_1 \Op_b, \; \Op_b \circ_1 \Op_a, \;
        \Op_b \circ_2 \Op_a, \; \Op_a \circ_2 \Op_b
        : b \in \Pca \mbox{ and } a \Ord_\Pca b\}.
    \end{equation}
    Then, for all $x, y \in R_a$, $x - y$ is an element of the space of
    relations $\FreeRel_\Pca^\Op$ of $\As(\Pca)$.
\end{Lemma}
\begin{proof}
    This statement is a consequence of the definition of the space
    $\FreeRel_\Pca^\Op$ through the elements~\eqref{equ:relation_1}
    and~\eqref{equ:relation_2}. For instance, for
    $x := \Op_a \circ_1 \Op_b$ and $y := \Op_{b'} \circ_1 \Op_a$ where
    $b$ and $b'$ are elements of $\Pca$ such that $a \Ord_\Pca b$ and
    $a \Ord_\Pca b'$, by~\eqref{equ:relation_1}, we have that
    $\Op_a \circ_1 \Op_b - \Op_a \circ_2 \Op_a$ and
    $\Op_{b'} \circ_1 \Op_a - \Op_a \circ_2 \Op_a$ are in
    $\FreeRel_\Pca^\Op$. Thus,
    \begin{math}
        (\Op_a \circ_1 \Op_b - \Op_a \circ_2 \Op_a)
        - (\Op_{b'} \circ_1 \Op_a - \Op_a \circ_2 \Op_a) =
        x - y
    \end{math}
    also is. The proof for the other possibilities for $x$ and $y$
    as elements of $R_a$ is analogous.
\end{proof}
\medskip

By Lemma~\ref{lem:relations}, we observe that all $\Op_a$, $a \in \Pca$,
are associative. For this reason, $\As(\Pca)$ is a generalization of the
associative operad on several generating operations. As we will see in
the sequel, this very simple way to produce operads has many
combinatorial and algebraic properties.
\medskip

\subsubsection{Functoriality}
For any morphism of posets $\phi : \Pca_1 \to \Pca_2$, we denote by
$\As(\phi)$ the map
\begin{equation}
    \As(\phi) : \As(\Pca_1)(2) \to \As(\Pca_2)(2)
\end{equation}
defined by
\begin{equation}
    \As(\phi)\left(\pi_1(\Op_x)\right) :=
    \pi_2\left(\Op_{\phi(x)}\right)
\end{equation}
for all $x \in \Pca_1$, where
$\pi_1 : \Free(\FreeGen_{\Pca_1}^\Op) \to \As(\Pca_1)$ and
$\pi_2 : \Free(\FreeGen_{\Pca_2}^\Op) \to \As(\Pca_2)$ are canonical
projections.
\medskip

\begin{Lemma} \label{lem:as_poset_map}
    Let $\Pca_1$ and $\Pca_2$ be two posets and
    $\phi : \Pca_1 \to \Pca_2$ be a morphism of posets. Then, the map
    $\As(\phi)$ uniquely extends into an operad morphism from
    $\As(\Pca_1)$ to $\As(\Pca_2)$.
\end{Lemma}
\begin{proof}
    Let us denote by
    $\pi_1 : \Free(\FreeGen_{\Pca_1}^\Op) \to \As(\Pca_1)$ and by
    $\pi_2 : \Free(\FreeGen_{\Pca_2}^\Op) \to \As(\Pca_2)$ the canonical
    projections. Let $r$ be an element of $\FreeRel_{\Pca_1}^\Op$
    of the form~\eqref{equ:relation_1}. Then,
    \begin{math}
        r = \Op_a \circ_1 \Op_b -
        \Op_{a \Min_{\Pca_1} b} \circ_2 \Op_{a \Min_{\Pca_1} b}
    \end{math}
    where $a$ and $b$ are two comparable elements of $\Pca_1$. We have,
    by Lemma~\ref{lem:morphism_posets_min},
    \begin{equation}\begin{split}
        \As(\phi)(\pi_1(\Op_a)) \circ_1 & \As(\phi)(\pi_1(\Op_b))
        -  \As(\phi)(\pi_1(\Op_{a \Min_{\Pca_1} b})) \circ_2
            \As(\phi)(\pi_1(\Op_{a \Min_{\Pca_1} b}))
        \\ & =
        \pi_2\left(\Op_{\phi(a)}\right) \circ_1 \pi_2(\Op_{\phi(b)})
        -
        \pi_2\left(\Op_{\phi(a \Min_{\Pca_1} b)}\right) \circ_2
            \pi_2\left(\Op_{\phi(a \Min_{\Pca_1} b)}\right)
        \\ & =
        \pi_2\left(\Op_{\phi(a)}\right) \circ_1 \pi_2\left(\Op_{\phi(b)}\right)
        -
        \pi_2\left(\Op_{\phi(a) \Min_{\Pca_2} \phi(b)}\right) \circ_2
            \pi_2\left(\Op_{\phi(a) \Min_{\Pca_2} \phi(b)}\right)
        \\ & =
        \pi_2\left(\Op_{\phi(a)} \circ_1 \Op_{\phi(b)}
        -
        \Op_{\phi(a) \Min_{\Pca_2} \phi(b)} \circ_2
            \Op_{\phi(a) \Min_{\Pca_2} \phi(b)}\right)
        \\ & = 0,
    \end{split}\end{equation}
    showing that $\As(\phi)(\pi_1(r)) = 0$. An analogous computation shows
    that the same property holds for all elements $r$ of
    $\FreeRel_{\Pca_1}^\Op$ of the form~\eqref{equ:relation_2}. This
    shows that $\As(\phi)$ extends in a unique way into an operad morphism.
\end{proof}
\medskip

\begin{Theorem} \label{thm:as_poset_functor}
    The construction $\As$ is a functor from the category of posets to
    the category of binary and quadratic operads.
\end{Theorem}
\begin{proof}
    By definition, the construction $\As$ sends any poset $\Pca$ to an
    operad $\As(\Pca)$, defined by its presentation which is binary and
    quadratic. Moreover, by Lemma~\ref{lem:as_poset_map}, $\As$ sends
    any morphism of posets $\phi$ to a morphism of operads $\As(\phi)$.
    It follows the statement of the theorem.
\end{proof}
\medskip

\subsubsection{First examples} \label{subsubsec:first_examples}
Let us use Theorem~\ref{thm:as_poset_functor} to exhibit some examples
of constructions of binary an quadratic operads from posets.
\medskip

From the poset
\begin{equation}
    \Pca :=
    \begin{tikzpicture}
        [baseline=(current bounding box.center),xscale=.4,yscale=.4]
        \node[Vertex](1)at(0,0){\begin{math}1\end{math}};
        \node[Vertex](2)at(2,0){\begin{math}2\end{math}};
        \node[Vertex](3)at(1,-1){\begin{math}3\end{math}};
        \node[Vertex](4)at(3,-1){\begin{math}4\end{math}};
        \draw[Edge](1)--(3);
        \draw[Edge](2)--(3);
        \draw[Edge](2)--(4);
    \end{tikzpicture}\,,
\end{equation}
the operad $\As(\Pca)$ is binary and quadratic, generated by the set
$\FreeGen_\Pca^\Op = \{\Op_1, \Op_2, \Op_3, \Op_4\}$, and,
by Lemma~\ref{lem:relations}, submitted to the relations
\begin{subequations}
\begin{equation}
    \Op_1 \circ_1 \Op_1
    = \Op_1 \circ_1 \Op_3 = \Op_3 \circ_1 \Op_1
    = \Op_3 \circ_2 \Op_1 = \Op_1 \circ_2 \Op_3
    = \Op_1 \circ_2 \Op_1,
\end{equation}
\begin{equation}\begin{split}
    \Op_2 \circ_1 \Op_2
    = \Op_2 \circ_1 \Op_3 & = \Op_2 \circ_1 \Op_4
    = \Op_3 \circ_1 \Op_2 = \Op_4 \circ_1 \Op_2
    \\ & = \Op_4 \circ_2 \Op_2 = \Op_3 \circ_2 \Op_2
    = \Op_2 \circ_2 \Op_4 = \Op_2 \circ_2 \Op_3
    = \Op_2 \circ_2 \Op_2,
\end{split}\end{equation}
\begin{equation}
    \Op_3 \circ_1 \Op_3 = \Op_3 \circ_2 \Op_3,
\end{equation}
\begin{equation}
    \Op_4 \circ_1 \Op_4 = \Op_4 \circ_2 \Op_4.
\end{equation}
\end{subequations}
\medskip

Besides, when $\Pca$ is the trivial poset on the set $[\ell]$,
$\ell \geq 0$, the operad $\As(\Pca)$ is generated by the set
$\FreeGen_\Pca^\Op = \{\Op_1, \dots, \Op_\ell\}$ and, by
Lemma~\ref{lem:relations}, submitted to the relations
\begin{equation}
    \Op_a \circ_1 \Op_a = \Op_a \circ_2 \Op_a,
    \qquad a \in [\ell].
\end{equation}
In particular, when $\ell = 0$, $\As(\Pca)$ is the trivial operad, when
$\ell = 1$, $\As(\Pca)$ is the associative operad, and when $\ell = 2$,
$\As(\Pca)$ is the operad $\TwoAs$~\cite{LR06}. These operads, for a generic
$\ell \geq 0$, have been introduced in~\cite{Gir15b} under the name of
{\em dual multiassociative operads}. These operads can be realized using
Schröder trees~\cite{Sta11} with labelings satisfying some conditions.
In Section~\ref{subsubsec:realization_as_poset_forest}, we shall
describe a generalized version of this realization.
\medskip

Furthermore, when $\Pca$ is the total order on the set $[\ell]$,
$\ell \geq 0$, the operad $\As(\Pca)$ is generated by the set
$\FreeGen_\Pca^\Op = \{\Op_1, \dots, \Op_\ell\}$ and, by
Lemma~\ref{lem:relations}, submitted to the relations
\begin{equation}
    \Op_a \circ_1 \Op_a =
    \Op_a \circ_1 \Op_b = \Op_b \circ_1 \Op_a =
    \Op_b \circ_2 \Op_a = \Op_a \circ_2 \Op_b =
    \Op_a \circ_2 \Op_a,
    \qquad a \leq b \in [\ell].
\end{equation}
In particular, when $\ell = 0$, $\As(\Pca)$ is the trivial operad and
when $\ell = 1$, $\As(\Pca)$ is the associative operad. These operads,
for a generic $\ell \geq 0$, have been introduced and studied
in~\cite{Gir15b} under the name of {\em multiassociative operads}. They
have the particularity to have stationary dimensions since
$\dim \As(\Pca)(1) = 1$ and $\dim \As(\Pca)(n) = \ell$ for all~$n \geq 2$.
\medskip

\subsection{General properties}
Let us now list some general properties of the operad $\As(\Pca)$ where
$\Pca$ is a poset without particular requirements. We provide the
dimension of the space of relations of $\As(\Pca)$, describe its
associative elements, and give a necessary and sufficient condition for
its basicity.
\medskip

\subsubsection{Space of relations dimensions}

\begin{Proposition} \label{prop:dimension_relations}
    Let $\Pca$ be a poset. Then, the dimension of the space
    $\FreeRel_\Pca^\Op$ of relations of $\As(\Pca)$ satisfies
    \begin{equation} \label{equ:dimension_relations}
        \dim \FreeRel_\Pca^\Op = 4 \, \NbInterv(\Pca) - 3 \, \# \Pca.
    \end{equation}
\end{Proposition}
\begin{proof}
    Consider the space $\FreeRel_1$ generated by the elements
    of~\eqref{equ:relation_1} with $a \ne b$. Since these elements are
    linearly independent and are totally specified by two different
    elements $a$ and $b$ of $\Pca$ such that $a$ and $b$ are comparable
    in $\Pca$, we have
    \begin{equation}
        \dim \FreeRel_1 = 2 \left(\NbInterv(\Pca) - \# \Pca\right).
    \end{equation}
    For the same reason, the dimension of the space $\FreeRel_2$
    generated by the elements of~\eqref{equ:relation_2} with $a \ne b$
    satisfies $\dim \FreeRel_2 = \dim \FreeRel_1$. The space $\FreeRel_3$
    generated by the elements of~\eqref{equ:relation_1}
    and~\eqref{equ:relation_2} with $a = b$ consists in the space
    generated by $\Op_a \circ_1 \Op_a - \Op_a \circ_2 \Op_a$,
    $a \in \Pca$, and is hence of dimension $\# \Pca$. Therefore, since
    we have
    \begin{equation}
        \FreeRel_\Pca^\Op =
        \FreeRel_1 \oplus \FreeRel_2 \oplus \FreeRel_3,
    \end{equation}
    we obtain the stated formula~\eqref{equ:dimension_relations} for
    $\dim \FreeRel_\Pca^\Op$ by summing the dimensions of $\FreeRel_1$,
    $\FreeRel_2$, and~$\FreeRel_3$.
\end{proof}
\medskip

\subsubsection{Associative elements}

\begin{Proposition} \label{prop:associative_elements}
    Let $\Pca$ be a poset and
    $C := \{c_1 \OrdStrict_\Pca \dots \OrdStrict_\Pca c_\ell\}$
    be a chain of $\Pca$. Then, the linear span $\Cca_C$ of the set
    $\{\pi(\Op_{c_1}), \dots, \pi(\Op_{c_\ell})\}$ contains only
    associative elements of $\As(\Pca)$, where
    $\pi : \Free(\FreeGen_\Pca^\Op) \to \As(\Pca)$ is the canonical
    projection. Conversely, any associative element of $\As(\Pca)$
    is an element of $\Cca_C$ for a chain $C$ of $\Pca$.
\end{Proposition}
\begin{proof}
    Consider the element
    \begin{equation} \label{equ:associative_elements_x}
        x := \sum_{a \in \Pca} \lambda_a \, \Op_a
    \end{equation}
    of $\Free(\FreeGen_\Pca^\Op)$, where $\lambda_a \in \K$ for all
    $a \in \Pca$, such that $\pi(x)$ is associative in $\As(\Pca)$. Then,
    since we have $\pi(r) = 0$ for all elements $r$ of $\FreeRel_\Pca^\Op$,
    by Lemma~\ref{lem:relations}, we have
    \begin{equation}\begin{split} \label{equ:associative_elements}
        \pi(x \circ_1 x - x \circ_2 x)
            & = \sum_{a, b \in \Pca}
            \lambda_a \lambda_b \,
            \pi(\Op_a \circ_1 \Op_b - \Op_a \circ_2 \Op_b) \\
            & = \sum_{\substack{a, b \in \Pca \\ a \Ord_\Pca b}}
            \lambda_a \lambda_b \,
            \pi(\Op_a \circ_1 \Op_b - \Op_a \circ_2 \Op_b)
            +
            \sum_{\substack{a, b \in \Pca \\ b \OrdStrict_\Pca a}}
            \lambda_a \lambda_b \,
            \pi(\Op_a \circ_1 \Op_b - \Op_a \circ_2 \Op_b) \\
            & \qquad +
            \sum_{\substack{a, b \in \Pca \\ a \not \Ord_\Pca b \\
                b \not \Ord_\Pca a}}
            \lambda_a \lambda_b \, \pi(\Op_a \circ_1 \Op_b -
                \Op_a \circ_2 \Op_b) \\
            & = \sum_{\substack{a, b \in \Pca \\ a \not \Ord_\Pca b \\
                b \not \Ord_\Pca a}}
             \lambda_a \lambda_b \,
             \pi(\Op_a \circ_1 \Op_b - \Op_a \circ_2 \Op_b).
    \end{split}\end{equation}
    Now, the fact that $\pi(x)$ is associative, that is
    $\pi(x \circ_1 x - x \circ_2 x) = 0$, is equivalent to the fact that
    the relations
    \begin{equation} \label{equ:associative_elements_coefficients_condition}
        \lambda_a \lambda_b = 0,
        \qquad a, b \in \Pca, a \not \Ord_\Pca b,
        \mbox{ and } b \not \Ord_\Pca a
    \end{equation}
    are satisfied. This implies that $x$ expresses as
    \begin{equation} \label{equ:associative_elements_sum_chain}
        x = \sum_{c \in C} \lambda_c \, \Op_c,
    \end{equation}
    where $C$ is a chain of $\Pca$. Therefore, $\pi(x)$ is an element of
    $\Cca_C$. Conversely, let $x$ be an element of
    $\Free(\FreeGen_\Pca^\Op)$ of the
    form~\eqref{equ:associative_elements_sum_chain} for a chain $C$ of
    $\Pca$. Since the coefficients of $x$
    satisfy~\eqref{equ:associative_elements_coefficients_condition},
    by~\eqref{equ:associative_elements},
    $\pi(x \circ_1 x - x \circ_2 x) = 0$. Hence, $x$ is associative.
\end{proof}
\medskip

\subsubsection{Basicity}
An operad $\Oca$ is {\em basic}~\cite{Val07} if for any
nonzero element $y$ of $\Oca(m)$, $m \geq 1$, all the maps
\begin{equation}
    \circ^y_i : \Oca(n) \to \Oca(n + m - 1),
    \qquad i \in [n],
\end{equation}
linearly defined by
\begin{equation}
    \circ^y_i(x) := x \circ_i y,
    \qquad x \in \Oca(n),
\end{equation}
are injective. Notice that this definition of the basicity property for
operads is a slight but equivalent variant of the one appearing
in~\cite{Val07}.
\medskip

\begin{Proposition} \label{prop:basic_operad}
    Let $\Pca$ be a poset. The operad $\As(\Pca)$ is basic if and only
    if $\Pca$ is a trivial poset.
\end{Proposition}
\begin{proof}
    Assume that $\Pca$ is not a trivial poset. Then, there are two
    elements $a$ and $b$ in $\Pca$ such that $a \OrdStrict_\Pca b$. By
    denoting by $\pi : \Free(\FreeGen_\Pca^\Op) \to \As(\Pca)$ the
    canonical projection, we have
    \begin{equation}
        \circ^{\pi(\Op_a)}_1(\pi(\Op_a)) =
        \pi(\Op_a \circ_1 \Op_a)
    \end{equation}
    and
    \begin{equation} \label{equ:basic_operad}
        \circ^{\pi(\Op_a)}_1(\pi(\Op_b)) =
        \pi(\Op_b \circ_1 \Op_a) =
        \pi(\Op_a \circ_1 \Op_a).
    \end{equation}
    The second equality of~\eqref{equ:basic_operad} is a consequence
    of the fact that, by Lemma~\ref{lem:relations}, the element
    $\Op_b \circ_1 \Op_a - \Op_a \circ_1 \Op_a$ belongs to the space
    $\FreeRel_\Pca^\Op$ of relations of $\As(\Pca)$. Since $a \ne b$,
    $\pi(\Op_a) \ne \pi(\Op_b)$, showing that the map
    $\circ^{\pi(\Op_a)}_1$ is not injective and thus, that $\As(\Pca)$
    is not basic.
    \smallskip

    Conversely assume that $\Pca$ is a trivial poset. The space
    $\FreeRel_\Pca^\Op$ of relations $\As(\Pca)$ is hence generated by
    the $\Op_a \circ_1 \Op_a - \Op_a \circ_2 \Op_a$, $a \in \Pca$.
    Therefore, for any syntax trees $\Sfr$ and $\Tfr$ of
    $\Free(\FreeGen_\Pca^\Op)$, $\Sfr - \Tfr$ is in
    $\langle \FreeRel_\Pca^\Op \rangle$ only if the roots of $\Sfr$ and
    $\Tfr$ have the same label. This implies that for all nonzero
    elements $y$ of $\Oca$, all the maps $\circ^y_i$ are injective. Thus,
    $\As(\Pca)$ is basic.
\end{proof}
\medskip

\subsection{Algebras over poset associative operads}
Let $\Pca$ be a poset. From the presentation
$(\FreeGen_\Pca^\Op, \FreeRel_\Pca^\Op)$ of the operad $\As(\Pca)$
provided by its definition in Section~\ref{subsubsec:definition_as}, an
$\As(\Pca)$-algebra is a vector space $\Alg$ endowed with linear
operations
\begin{equation}
    \Op_a : \Alg \otimes \Alg \to \Alg,
    \qquad a \in \Pca,
\end{equation}
satisfying, for all $x, y, z \in \Alg$, the relations
\begin{equation} \label{equ:relations_algebras}
    (x \Op_a y) \Op_b z
    = x \Op_a (y \Op_b z)
    = (x \Op_c y) \Op_a z
    = x \Op_c (y \Op_a z),
    \qquad a, b, c \in \Pca
    \mbox{ and } a \Ord_\Pca b \mbox{ and } a \Ord_\Pca c.
\end{equation}
We call {\em $\Pca$-associative algebra} any $\As(\Pca)$-algebra.
\medskip

We shall exhibit two examples of $\Pca$-associative algebras in the
sequel: in Section~\ref{subsubsec:free_as_forest_poset_algebras}, free
$\Pca$-associative algebras over one generator when $\Pca$ is a forest
poset and in Section~\ref{subsubsec:antichains_algebra},
$\Pca$-associative algebras involving the antichains of the poset $\Pca$.
\medskip

\subsubsection{Units}
Let $\Pca$ be a poset and $\Alg$ be a $\Pca$-associative algebra. An
{\em $a$-unit}, $a \in \Pca$, of $\Alg$ is an element $\Unit_a$ of
$\Alg$ satisfying
\begin{equation}
    \Unit_a \Op_a x = x = x \Op_a \Unit_a
\end{equation}
for all $x \in \Alg$. Obviously, for any $a \in \Pca$ there is at most
one $a$-unit in $\Alg$.
\medskip

Besides, for any element $x$ of $\Alg$, we denote by $\UnitSet_\Alg(x)$
the set of elements $a$ of $\Pca$ such that $x$ is an $a$-unit of $\Alg$.
Obviously, if $\Unit_a$ is an $a$-unit of $\Alg$,
$a \in \UnitSet_\Alg(\Unit_a)$.
\medskip

\begin{Proposition} \label{prop:units}
    Let $\Pca$ be a poset and $\Alg$ be a $\Pca$-associative algebra.
    Then:
    \begin{enumerate}[label={\it (\roman*)}]
        \item \label{item:units_1}
        for any element $x$ of $\Alg$, $\UnitSet_\Alg(x)$ is an order
        filter of $\Pca$;
        \item \label{item:units_2}
        for all elements $x$ and $y$ of $\Alg$ such that $x \ne y$, the
        sets $\UnitSet_\Alg(x)$ and $\UnitSet_\Alg(y)$ are disjoint.
    \end{enumerate}
\end{Proposition}
\begin{proof}
    Assume that $x$ is an element of $\Alg$ such that there is an element
    $a$ of $\Pca$ satisfying $a \in \UnitSet_\Alg(x)$. Then, by
    definition of $\UnitSet_\Alg(x)$, $x$ is an $a$-unit of $\Alg$. Thus,
    for all $b$ of $\Pca$ such that $a \Ord_\Pca b$ and all elements $y$
    of $\Alg$, we have by~\eqref{equ:relations_algebras},
    \begin{equation}
        x \Op_b y = (x \Op_a x) \Op_b y
        = (x \Op_a x) \Op_a y
        = x \Op_a y
        = y.
    \end{equation}
    By using similar arguments, we obtain that $y \Op_b x = y$, showing
    that $x$ is a $b$-unit and hence, that $b \in \UnitSet_\Alg(x)$.
    This establishes~\ref{item:units_1}.
    \smallskip

    Let us now consider two elements $x$ and $y$ of $\Alg$ such that
    there is an element $a$ of $\Pca$ satisfying
    $a \in \UnitSet_\Alg(x) \cap \UnitSet_\Alg(y)$. Then, by definition
    of $\UnitSet_\Alg(x)$ and $\UnitSet_\Alg(y)$, $x$ and $y$ are both
    $a$-units of $\Alg$, implying that $x = y$. This
    implies~\ref{item:units_2}.
\end{proof}
\medskip

Proposition~\ref{prop:units} implies that the sets $\UnitSet_\Alg(x)$,
$x \in \Alg$, form a partition of an order filter of~$\Pca$ where each
part is itself an order filter of $\Pca$.
\medskip

\subsubsection{Antichains algebra}%
\label{subsubsec:antichains_algebra}
Let $\Pca$ be a poset and set $\Xbb_\Pca := \{x_a : a \in \Pca\}$ as
a set of commutative parameters and consider the commutative and
associative polynomial algebra $\K[\Xbb_\Pca]/_{\Ideal_\Pca}$, where
$\Ideal_\Pca$ is the ideal of $\K[\Xbb_\Pca]$ generated by
\begin{equation}
    x_a x_b - x_a, \qquad a \Ord_\Pca b \in \Pca.
\end{equation}
Then, one observes that $x_{a_1} \dots x_{a_k}$ is a reduced monomial of
$\K[\Xbb_\Pca]/_{\Ideal_\Pca}$ if and only if the set
$\{a_1, \dots, a_k\}$ is an antichain of $\Pca$ of size $k$.
\medskip

We endow $\K[\Xbb_\Pca]/_{\Ideal_\Pca}$ with linear operations
\begin{equation}
    \Op_a :
    \K[\Xbb_\Pca]/_{\Ideal_\Pca} \otimes
    \K[\Xbb_\Pca]/_{\Ideal_\Pca} \to
    \K[\Xbb_\Pca]/_{\Ideal_\Pca},
    \qquad a \in \Pca,
\end{equation}
defined, for all reduced monomials $x_{b_1} \dots x_{b_k}$ and
$x_{c_1} \dots x_{c_\ell}$ of $\K[\Xbb]/_{\Ideal_\Pca}$, by
\begin{equation}
    x_{b_1} \dots x_{b_k} \Op_a x_{c_1} \dots x_{c_\ell} :=
    \pi(x_{b_1} \dots x_{b_k} x_a x_{c_1} \dots x_{c_\ell}),
\end{equation}
where $\pi : \K[\Xbb_\Pca] \to \K[\Xbb_\Pca]/_{\Ideal_\Pca}$ is the
canonical projection. These operations $\Op_a$, $a \in \Pca$, endow
$\K[\Xbb_\Pca]/_{\Ideal_\Pca}$ with a structure of a $\Pca$-associative
algebra.
\medskip

Consider for instance the poset
\begin{equation}
    \Pca :=
    \begin{tikzpicture}
        [baseline=(current bounding box.center),xscale=.45,yscale=.45]
        \node[Vertex](1)at(0,0){\begin{math}1\end{math}};
        \node[Vertex](2)at(-1,-1){\begin{math}2\end{math}};
        \node[Vertex](3)at(0,-1){\begin{math}3\end{math}};
        \node[Vertex](4)at(1,-1){\begin{math}4\end{math}};
        \node[Vertex](5)at(0,-2){\begin{math}5\end{math}};
        \draw[Edge](1)--(2);
        \draw[Edge](1)--(3);
        \draw[Edge](1)--(4);
        \draw[Edge](2)--(5);
        \draw[Edge](4)--(5);
    \end{tikzpicture}\,.
\end{equation}
The space $\K[\Xbb_\Pca]/_{\Ideal_\Pca}$ is the linear span of the
reduced monomials
\begin{equation}
    x_1, \; x_2, \; x_3, \; x_4, \; x_5, \;
    x_2 x_3, \; x_2 x_4, \; x_3 x_4, \; x_3 x_5, \;
    x_2 x_3 x_4,
\end{equation}
and one has for instance
\begin{subequations}
\begin{equation}
    x_2 \Op_3 x_4 = x_2 x_3 x_4,
\end{equation}
\begin{equation}
    x_2 x_3 \Op_1 x_4 = x_1,
\end{equation}
\begin{equation}
    x_2 x_3 \Op_5 x_4 = x_2 x_3 x_4.
\end{equation}
\end{subequations}
\medskip

\section{Forest posets, Koszul duality, and Koszulity}%
\label{sec:forest_posets}
Here, we focus on the construction $\As$ when the input poset $\Pca$
of the construction is a forest poset. In this case, we show that
$\As(\Pca)$ is Koszul, provide a realization of $\As(\Pca)$, and obtain
a functional equation for its Hilbert series. We end this section by
computing presentations of the Koszul dual of $\As(\Pca)$.
\medskip

\subsection{Koszulity and Poincaré-Birkhoff-Witt bases}
We prove here that when $\Pca$ is a forest poset, $\As(\Pca)$ is Koszul.
For that, we consider an orientation of the space of relations
$\FreeRel_\Pca^\Op$ of $\As(\Pca)$ and show that this orientation is a
convergent rewrite rule. This strategy to prove that $\As(\Pca)$ is
Koszul is deduced, as explained in
Section~\ref{subsubsec:koszul_duality_koszulity_criterion}, from the
works of Hoffbeck~\cite{Hof10} and Dotsenko and Khoroshkin~\cite{DK10}.
\medskip

\subsubsection{Orientation of the space of relations}
Let $\Pca$ be a poset (not necessarily a forest poset just now) and
$\Rew_\Pca'$ be the rewrite rule on $\Free(\FreeGen_\Pca^\Op)$
satisfying
\begin{subequations}
\begin{equation} \label{equ:rewrite_1}
    \Op_a \circ_1 \Op_b
    \Rew_\Pca'
    \Op_{a \Min_\Pca b} \circ_2 \Op_{a \Min_\Pca b},
    \qquad a, b \in \Pca \mbox{ and }
        (a \Ord_\Pca b \mbox{ or } b \Ord_\Pca a),
\end{equation}
\begin{equation} \label{equ:rewrite_2}
    \Op_a \circ_2 \Op_b
    \Rew_\Pca'
    \Op_{a \Min_\Pca b} \circ_2 \Op_{a \Min_\Pca b},
    \qquad a, b \in \Pca \mbox{ and }
        (a \OrdStrict_\Pca b \mbox{ or } b \OrdStrict_\Pca a).
\end{equation}
\end{subequations}
\medskip

\subsubsection{Convergent rewrite rule}

\begin{Lemma} \label{lem:rewrite_rule_space_relations}
    Let $\Pca$ be a poset. Then, the space induced by the rewrite rule
    $\Rew_\Pca'$ on $\Free(\FreeGen_\Pca^\Op)$ is the operad ideal of
    $\Free(\FreeGen_\Pca^\Op)$ generated by the space of relations
    $\FreeRel_\Pca^\Op$ of $\As(\Pca)$.
\end{Lemma}
\begin{proof}
    Let $\Sfr$ and $\Tfr$ be two syntax trees such that
    $\Sfr \RewEquiv_\Pca \Tfr$. To prove that
    $\Sfr - \Tfr \in \langle\FreeRel_\Pca^\Op\rangle$, it is enough to
    prove that $\Sfr \Rew_\Pca' \Tfr$ implies
    $\Sfr - \Tfr \in \FreeRel_\Pca^\Op$. By a direct inspection
    of~\eqref{equ:rewrite_1} and~\eqref{equ:rewrite_2}, we observe
    by Lemma~\ref{lem:relations} that $\Sfr - \Tfr$ is
    in~$\FreeRel_\Pca^\Op$.
    \smallskip

    Conversely, let $\Sfr$ and $\Tfr$ be two syntax trees such that
    $\Sfr - \Tfr \in \langle\FreeRel_\Pca^\Op\rangle$. It is enough to
    prove that $\Sfr - \Tfr \in \FreeRel_\Pca^\Op$ implies
    $\Sfr \RewEquiv_\Pca \Tfr$. Since $\FreeRel_\Pca^\Op$ is generated
    by~\eqref{equ:relation_1} and~\eqref{equ:relation_2}, we have two
    cases to consider. If $\Sfr = \Op_a \circ_1 \Op_b$ and
    $\Tfr = \Op_{a \Min_\Pca b} \circ_2 \Op_{a \Min_\Pca b}$  where $a$
    and $b$ are two comparable elements of $\Pca$,
    by~\eqref{equ:rewrite_1}, $\Sfr \Rew_\Pca' \Tfr$. Otherwise,
    $\Sfr = \Op_{a \Min_\Pca b} \circ_1 \Op_{a \Min_\Pca b}$ and
    $\Tfr = \Op_a \circ_2 \Op_b$  where $a$ and $b$ are two comparable
    elements of $\Pca$. By setting
    $\Rfr := \Op_{a \Min_\Pca b} \circ_2 \Op_{a \Min_\Pca b}$, we have
    by~\eqref{equ:rewrite_1}, $\Sfr \Rew_\Pca' \Rfr$ and,
    by~\eqref{equ:rewrite_2}, $\Tfr \Rew_\Pca' \Rfr$. Then,
    $\Sfr \RewEquiv_\Pca \Tfr$.
\end{proof}
\medskip

Lemma~\ref{lem:rewrite_rule_space_relations} shows that the rewrite rule
$\Rew_\Pca'$ is an orientation of the space of relations
$\FreeRel_\Pca^\Op$ of~$\As(\Pca)$.
\medskip

\begin{Lemma} \label{lem:terminating_rewrite_rule}
    Let $\Pca$ be a poset. Then, the rewrite rule $\Rew_\Pca'$ on
    $\Free(\FreeGen_\Pca^\Op)$ is terminating.
\end{Lemma}
\begin{proof}
    Let $n \geq 1$ and consider the map
    \begin{equation}
        \phi : \Free(\FreeGen_\Pca^\Op)(n) \to \N \times \Pca^{n - 1}
    \end{equation}
    such that for any syntax tree $\Tfr$ of $\Free(\FreeGen_\Pca^\Op)(n)$,
    $\phi(\Tfr) := (\alpha, u)$ where $\alpha$ is the sum, for all
    internal nodes $x$ of $\Tfr$ of the number of internal nodes in the
    right subtree of $x$, and $u$ is the infix reading word of $\Tfr$. Let
    $\leq$ be the order relation on $\N \times \Pca^{n - 1}$ satisfying
    $(\alpha, u) \leq (\alpha', u')$ if $\alpha \leq \alpha'$ and
    $u'_i \Ord_\Pca u_i$ for all $i \in [n - 1]$. For all comparable
    elements $a$ and $b$ of $\Pca$, we have
    \begin{equation}
        \phi(\Op_a \circ_1 \Op_b)
        = (0, ab)
        < (1, (a \Min_\Pca b) \, (a \Min_\Pca b))
        = \phi(\Op_{a \Min_\Pca b} \circ_2 \Op_{a \Min_\Pca b}),
    \end{equation}
    and when in addition $a \ne b$,
    \begin{equation}
        \phi(\Op_a \circ_2 \Op_b)
        = (1, ab)
        < (1, (a \Min_\Pca b) \, (a \Min_\Pca b))
        = \phi(\Op_{a \Min_\Pca b} \circ_2 \Op_{a \Min_\Pca b}).
    \end{equation}
    Therefore, for all syntax trees $\Sfr$ and $\Tfr$ such that
    $\Sfr \Rew_\Pca' \Tfr$, $\phi(\Sfr) < \phi(\Tfr)$. This implies that
    for all syntax trees $\Sfr$ and $\Tfr$ such that $\Sfr \ne \Tfr$ and
    $\Sfr \RewTrans_\Pca \Tfr$, $\phi(\Sfr) < \phi(\Tfr)$. Now, since
    $\Rew_\Pca'$ preserves the arities of syntax trees and since the set
    $\Free(\FreeGen_\Pca^\Op)(n)$ is finite, $\Rew_\Pca'$ is terminating.
\end{proof}
\medskip

\begin{Lemma} \label{lem:normal_forms}
    Let $\Pca$ be a poset. Then, the set of normal forms of the rewrite
    rule $\Rew_\Pca'$ on $\Free(\FreeGen_\Pca^\Op)$ is the set of the
    syntax trees $\Tfr$ on $\Free(\FreeGen_\Pca^\Op)$ such that for any
    internal node of $\Tfr$ labeled by $\Op_a$ having a left (resp. right)
    child labeled by $\Op_b$, $a$ and $b$ are incomparable (resp. are
    equal or are incomparable) in $\Pca$.
\end{Lemma}
\begin{proof}
    By Lemma~\ref{lem:terminating_rewrite_rule}, $\Rew_\Pca'$ is
    terminating. Therefore, $\Rew_\Pca'$ admits normal forms, which are
    by definition the syntax trees of $\Free(\FreeGen_\Pca^\Op)$ that
    cannot be rewritten by $\Rew_\Pca'$. These syntax trees are exactly
    the ones avoiding the patterns appearing as left members
    in~\eqref{equ:rewrite_1} and~\eqref{equ:rewrite_2}, and are those
    described in the statement of the lemma.
\end{proof}
\medskip

Let us denote by $\NormalF(\Pca)$ the set of normal forms of $\Rew_\Pca'$,
described in the statement of Lemma~\ref{lem:normal_forms}. Moreover,
we denote by $\NormalF(\Pca)(n)$, $n \geq 1$, the set $\NormalF(\Pca)$
restricted to syntax trees with exactly~$n$ leaves. From their
description provided by Lemma~\ref{lem:normal_forms}, any tree $\Tfr$ of
$\NormalF(\Pca)$ different from the leaf is of the recursive unique
general form
\begin{equation} \label{equ:normal_forms}
    \Tfr =
    \begin{tikzpicture}
        [baseline=(current bounding box.center),xscale=.5,yscale=.45]
        \node(0)at(0.00,-1.67){\begin{math}\Sfr_1\end{math}};
        \node(2)at(2.00,-3.33){\begin{math}\Sfr_{\ell - 1}\end{math}};
        \node(4)at(4.00,-3.33){\begin{math}\Sfr_\ell\end{math}};
        \node(1)at(1.00,0.00){\begin{math}\Op_a\end{math}};
        \node(3)at(3.00,-1.67){\begin{math}\Op_a\end{math}};
        \draw(0)--(1);
        \draw(2)--(3);
        \draw[densely dashed](3)--(1);
        \draw(4)--(3);
        \node(r)at(1.00,1){};
        \draw(r)--(1);
    \end{tikzpicture}\,,
\end{equation}
where $a \in \Pca$ and the dashed edge denotes a right comb tree wherein
internal nodes are labeled by $\Op_a$ and for any $i \in [\ell]$,
$\Sfr_i$ is a tree of $\NormalF(\Pca)$ such that $\Sfr_i$ is the leaf or
its root is labeled by a~$\Op_b$, $b \in\Pca$, so that $a$ and $b$ are
incomparable in~$\Pca$.
\medskip

\begin{Lemma} \label{lem:confluent_rewrite_rule}
    Let $\Pca$ be a forest poset. Then, the rewrite rule $\Rew_\Pca'$ on
    $\Free(\FreeGen_\Pca^\Op)$ is confluent.
\end{Lemma}
\begin{proof}
    During this proof, to gain clarity in our drawings of trees, we
    will represent labels of internal nodes of syntax trees of
    $\Free(\FreeGen_\Pca^\Op)$ only by their subscripts. For instance,
    an internal node labeled by $a$, $a \in \Pca$, denotes an internal
    node labeled by $\Op_a$. Moreover, a word $a_1 \dots a_k$ of
    elements of $\Pca$ shall represent the element
    $a_1 \Min_\Pca \dots \Min_\Pca a_k$. For instance, an internal node
    labeled by $a b$ denotes an internal node labeled by
    $\Op_{a \Min_\Pca b}$. Under these conventions, the syntax trees
    \begin{equation} \label{equ:confluent_rewrite_rule_critical_trees}
        \begin{tikzpicture}
            [baseline=(current bounding box.center),xscale=.26,yscale=.22]
            \node(0)at(0.00,-5.25){};
            \node(2)at(2.00,-5.25){};
            \node(4)at(4.00,-3.50){};
            \node(6)at(6.00,-1.75){};
            \node(1)at(1.00,-3.50){\begin{math}a\end{math}};
            \node(3)at(3.00,-1.75){\begin{math}b\end{math}};
            \node(5)at(5.00,0.00){\begin{math}c\end{math}};
            \draw(0)--(1);
            \draw(1)--(3);
            \draw(2)--(1);
            \draw(3)--(5);
            \draw(4)--(3);
            \draw(6)--(5);
            \node(r)at(5.00,1.75){};
            \draw(r)--(5);
        \end{tikzpicture}\,,\quad
        \begin{tikzpicture}
            [baseline=(current bounding box.center),xscale=.3,yscale=.25]
            \node(0)at(0.00,-3.50){};
            \node(2)at(2.00,-5.25){};
            \node(4)at(4.00,-5.25){};
            \node(6)at(6.00,-1.75){};
            \node(1)at(1.00,-1.75){\begin{math}a\end{math}};
            \node(3)at(3.00,-3.50){\begin{math}b\end{math}};
            \node(5)at(5.00,0.00){\begin{math}c\end{math}};
            \draw(0)--(1);
            \draw(1)--(5);
            \draw(2)--(3);
            \draw(3)--(1);
            \draw(4)--(3);
            \draw(6)--(5);
            \node(r)at(5.00,1.5){};
            \draw(r)--(5);
        \end{tikzpicture}\,,\quad
        \begin{tikzpicture}
            [baseline=(current bounding box.center),xscale=.3,yscale=.22]
            \node(0)at(0.00,-4.67){};
            \node(2)at(2.00,-4.67){};
            \node(4)at(4.00,-4.67){};
            \node(6)at(6.00,-4.67){};
            \node(1)at(1.00,-2.33){\begin{math}a\end{math}};
            \node(3)at(3.00,0.00){\begin{math}b\end{math}};
            \node(5)at(5.00,-2.33){\begin{math}c\end{math}};
            \draw(0)--(1);
            \draw(1)--(3);
            \draw(2)--(1);
            \draw(4)--(5);
            \draw(5)--(3);
            \draw(6)--(5);
            \node(r)at(3.00,1.75){};
            \draw(r)--(3);
        \end{tikzpicture}\,,\quad
        \begin{tikzpicture}
            [baseline=(current bounding box.center),xscale=.3,yscale=.25]
            \node(0)at(0.00,-1.75){};
            \node(2)at(2.00,-5.25){};
            \node(4)at(4.00,-5.25){};
            \node(6)at(6.00,-3.50){};
            \node(1)at(1.00,0.00){\begin{math}a\end{math}};
            \node(3)at(3.00,-3.50){\begin{math}b\end{math}};
            \node(5)at(5.00,-1.75){\begin{math}c\end{math}};
            \draw(0)--(1);
            \draw(2)--(3);
            \draw(3)--(5);
            \draw(4)--(3);
            \draw(5)--(1);
            \draw(6)--(5);
            \node(r)at(1.00,1.5){};
            \draw(r)--(1);
        \end{tikzpicture}\,,\quad
        \begin{tikzpicture}
            [baseline=(current bounding box.center),xscale=.26,yscale=.22]
            \node(0)at(0.00,-1.75){};
            \node(2)at(2.00,-3.50){};
            \node(4)at(4.00,-5.25){};
            \node(6)at(6.00,-5.25){};
            \node(1)at(1.00,0.00){\begin{math}a\end{math}};
            \node(3)at(3.00,-1.75){\begin{math}b\end{math}};
            \node(5)at(5.00,-3.50){\begin{math}c\end{math}};
            \draw(0)--(1);
            \draw(2)--(3);
            \draw(3)--(1);
            \draw(4)--(5);
            \draw(5)--(3);
            \draw(6)--(5);
            \node(r)at(1.00,1.75){};
            \draw(r)--(1);
        \end{tikzpicture}\,,
    \end{equation}
    are critical trees of $\Rew_\Pca'$ where in the first, third, and
    last trees, $a$ and $b$ are comparable and $b$ and $c$ are
    comparable, in the second tree, $a$ and $b$ are comparable and $a$
    and $c$ are comparable, and in the fourth tree, $a$ and $c$ are
    comparable and $b$ and $c$ are comparable. Moreover, all critical
    trees of $\Rew_\Pca'$ consisting in three or less internal nodes are
    one of the five trees
    of~\eqref{equ:confluent_rewrite_rule_critical_trees}.
    \smallskip

    Because given a tree $\Tfr$
    of~\eqref{equ:confluent_rewrite_rule_critical_trees}, there are
    exactly two ways to rewrite $\Tfr$ in one step by~$\Rew_\Pca'$, each
    critical tree of $\Rew_\Pca'$ gives rise to one critical pair. By
    denoting by~$\raisebox{.2em}{\tikz{\draw[EdgeRew](0,0)--(.4,0);}}$
    the reflexive closure of the relation $\Rew_\Pca$, the fact that
    $\Pca$ is a forest poset and
    Lemma~\ref{lem:forest_poset_max_three_elements} imply that we have
    the five graphs of rewritings by $\Rew_\Pca'$

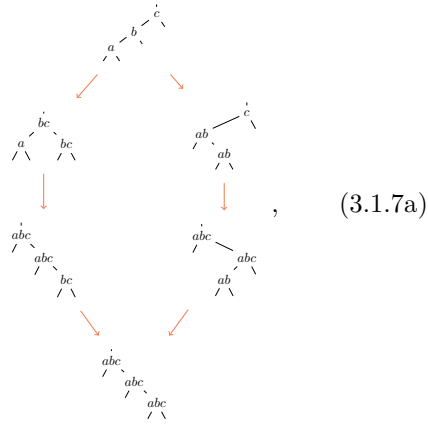
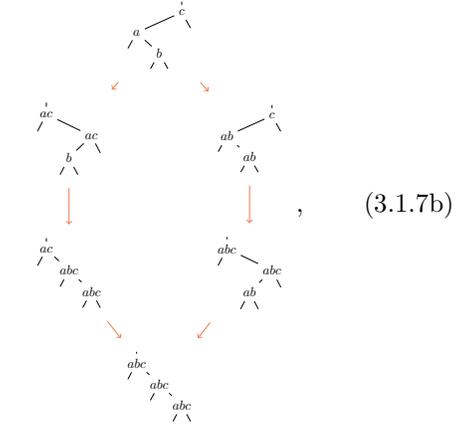
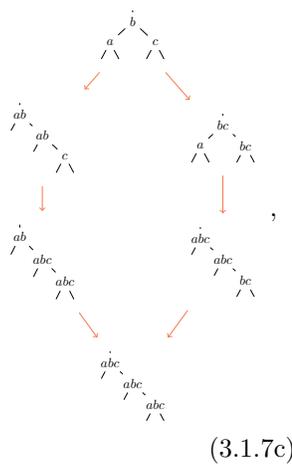
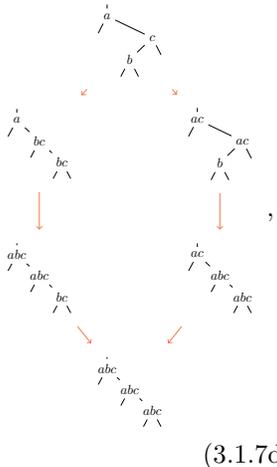
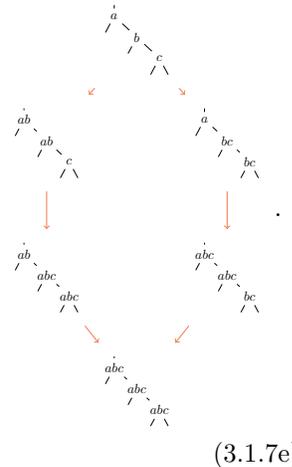
    \begin{subequations}
    \begin{minipage}[c]{.45\linewidth}
    \begin{equation}
    \scalebox{.5}{\begin{tikzpicture}
    [baseline=(current bounding box.center),xscale=.8,yscale=1.1]
        \node(A)at(0,0){
            \begin{tikzpicture}
            [baseline=(current bounding box.center),xscale=.3,yscale=.26]
            \begin{scope}
                \node(0)at(0.00,-5.25){};
                \node(2)at(2.00,-5.25){};
                \node(4)at(4.00,-3.50){};
                \node(6)at(6.00,-1.75){};
                \node(1)at(1.00,-3.50){\begin{math}a\end{math}};
                \node(3)at(3.00,-1.75){\begin{math}b\end{math}};
                \node(5)at(5.00,0.00){\begin{math}c\end{math}};
                \draw(0)--(1);
                \draw(1)--(3);
                \draw(2)--(1);
                \draw(3)--(5);
                \draw(4)--(3);
                \draw(6)--(5);
                \node(r)at(5.00,1.5){};
                \draw(r)--(5);
            \end{scope}\end{tikzpicture}};
        \node(B)at(-3,-2.5){
            \begin{tikzpicture}
            [baseline=(current bounding box.center),xscale=.3,yscale=.24]
            \begin{scope}
                \node(0)at(0.00,-4.67){};
                \node(2)at(2.00,-4.67){};
                \node(4)at(4.00,-4.67){};
                \node(6)at(6.00,-4.67){};
                \node(1)at(1.00,-2.33){\begin{math}a\end{math}};
                \node(3)at(3.00,0.00){\begin{math}b c\end{math}};
                \node(5)at(5.00,-2.33){\begin{math}b c\end{math}};
                \draw(0)--(1);
                \draw(1)--(3);
                \draw(2)--(1);
                \draw(4)--(5);
                \draw(5)--(3);
                \draw(6)--(5);
                \node(r)at(3.00,1.75){};
                \draw(r)--(3);
            \end{scope}\end{tikzpicture}};
        \node(C)at(3,-2.5){
            \begin{tikzpicture}
            [baseline=(current bounding box.center),xscale=.3,yscale=.31]
            \begin{scope}
                \node(0)at(0.00,-3.50){};
                \node(2)at(2.00,-5.25){};
                \node(4)at(4.00,-5.25){};
                \node(6)at(6.00,-1.75){};
                \node(1)at(1.00,-1.75){\begin{math}a b\end{math}};
                \node(3)at(3.00,-3.50){\begin{math}a b\end{math}};
                \node(5)at(5.00,0.00){\begin{math}c\end{math}};
                \draw(0)--(1);
                \draw(1)--(5);
                \draw(2)--(3);
                \draw(3)--(1);
                \draw(4)--(3);
                \draw(6)--(5);
                \node(r)at(5.00,1.31){};
                \draw(r)--(5);
            \end{scope}\end{tikzpicture}};
        \node(D)at(-3,-5.5){
            \begin{tikzpicture}
            [baseline=(current bounding box.center),xscale=.3,yscale=.34]
            \begin{scope}
                \node(0)at(0.00,-1.75){};
                \node(2)at(2.00,-3.50){};
                \node(4)at(4.00,-5.25){};
                \node(6)at(6.00,-5.25){};
                \node(1)at(1.00,0.00){\begin{math}a b c\end{math}};
                \node(3)at(3.00,-1.75){\begin{math}a b c\end{math}};
                \node(5)at(5.00,-3.50){\begin{math}b c\end{math}};
                \draw(0)--(1);
                \draw(2)--(3);
                \draw(3)--(1);
                \draw(4)--(5);
                \draw(5)--(3);
                \draw(6)--(5);
                \node(r)at(1.00,1.31){};
                \draw(r)--(1);
        \end{scope}\end{tikzpicture}};
        \node(E)at(3,-5.5){
            \begin{tikzpicture}
            [baseline=(current bounding box.center),xscale=.3,yscale=.33]
            \begin{scope}
                \node(0)at(0.00,-1.75){};
                \node(2)at(2.00,-5.25){};
                \node(4)at(4.00,-5.25){};
                \node(6)at(6.00,-3.50){};
                \node(1)at(1.00,0.00){\begin{math}a b c\end{math}};
                \node(3)at(3.00,-3.50){\begin{math}a b\end{math}};
                \node(5)at(5.00,-1.75){\begin{math}a b c\end{math}};
                \draw(0)--(1);
                \draw(2)--(3);
                \draw(3)--(5);
                \draw(4)--(3);
                \draw(5)--(1);
                \draw(6)--(5);
                \node(r)at(1.00,1.31){};
                \draw(r)--(1);
            \end{scope}\end{tikzpicture}};
        \node(F)at(0,-8.5){
            \begin{tikzpicture}
            [baseline=(current bounding box.center),xscale=.3,yscale=.32]
            \begin{scope}
                \node(0)at(0.00,-1.75){};
                \node(2)at(2.00,-3.50){};
                \node(4)at(4.00,-5.25){};
                \node(6)at(6.00,-5.25){};
                \node(1)at(1.00,0.00){\begin{math}a b c\end{math}};
                \node(3)at(3.00,-1.75){\begin{math}a b c\end{math}};
                \node(5)at(5.00,-3.50){\begin{math}a b c\end{math}};
                \draw(0)--(1);
                \draw(2)--(3);
                \draw(3)--(1);
                \draw(4)--(5);
                \draw(5)--(3);
                \draw(6)--(5);
                \node(r)at(1.00,1.31){};
                \draw(r)--(1);
            \end{scope}\end{tikzpicture}};
        \draw[EdgeRew](A)--(B);
        \draw[EdgeRew](A)--(C);
        \draw[EdgeRew](B)--(D);
        \draw[EdgeRew](C)--(E);
        \draw[EdgeRew](D)--(F);
        \draw[EdgeRew](E)--(F);
    \end{tikzpicture}}\,,
    \end{equation}
    \end{minipage}
    \qquad
    \begin{minipage}[c]{.45\linewidth}
    \begin{equation}
    \scalebox{.5}{\begin{tikzpicture}
    [baseline=(current bounding box.center),xscale=.8,yscale=1.1]
        \node(A)at(0,0){
            \begin{tikzpicture}
            [baseline=(current bounding box.center),xscale=.3,yscale=.31]
            \begin{scope}
                \node(0)at(0.00,-3.50){};
                \node(2)at(2.00,-5.25){};
                \node(4)at(4.00,-5.25){};
                \node(6)at(6.00,-1.75){};
                \node(1)at(1.00,-1.75){\begin{math}a\end{math}};
                \node(3)at(3.00,-3.50){\begin{math}b\end{math}};
                \node(5)at(5.00,0.00){\begin{math}c\end{math}};
                \draw(0)--(1);
                \draw(1)--(5);
                \draw(2)--(3);
                \draw(3)--(1);
                \draw(4)--(3);
                \draw(6)--(5);
                \node(r)at(5.00,1.31){};
                \draw(r)--(5);
            \end{scope}\end{tikzpicture}};
        \node(B)at(-3,-2.5){
            \begin{tikzpicture}
            [baseline=(current bounding box.center),xscale=.3,yscale=.33]
            \begin{scope}
                \node(0)at(0.00,-1.75){};
                \node(2)at(2.00,-5.25){};
                \node(4)at(4.00,-5.25){};
                \node(6)at(6.00,-3.50){};
                \node(1)at(1.00,0.00){\begin{math}a c\end{math}};
                \node(3)at(3.00,-3.50){\begin{math}b\end{math}};
                \node(5)at(5.00,-1.75){\begin{math}a c\end{math}};
                \draw(0)--(1);
                \draw(2)--(3);
                \draw(3)--(5);
                \draw(4)--(3);
                \draw(5)--(1);
                \draw(6)--(5);
                \node(r)at(1.00,1.31){};
                \draw(r)--(1);
            \end{scope}\end{tikzpicture}};
        \node(C)at(3,-2.5){
            \begin{tikzpicture}
            [baseline=(current bounding box.center),xscale=.3,yscale=.31]
            \begin{scope}
                \node(0)at(0.00,-3.50){};
                \node(2)at(2.00,-5.25){};
                \node(4)at(4.00,-5.25){};
                \node(6)at(6.00,-1.75){};
                \node(1)at(1.00,-1.75){\begin{math}a b\end{math}};
                \node(3)at(3.00,-3.50){\begin{math}a b\end{math}};
                \node(5)at(5.00,0.00){\begin{math}c\end{math}};
                \draw(0)--(1);
                \draw(1)--(5);
                \draw(2)--(3);
                \draw(3)--(1);
                \draw(4)--(3);
                \draw(6)--(5);
                \node(r)at(5.00,1.31){};
                \draw(r)--(5);
            \end{scope}\end{tikzpicture}};
        \node(D)at(-3,-5.75){
            \begin{tikzpicture}
            [baseline=(current bounding box.center),xscale=.3,yscale=.32]
            \begin{scope}
                \node(0)at(0.00,-1.75){};
                \node(2)at(2.00,-3.50){};
                \node(4)at(4.00,-5.25){};
                \node(6)at(6.00,-5.25){};
                \node(1)at(1.00,0.00){\begin{math}a c\end{math}};
                \node(3)at(3.00,-1.75){\begin{math}a b c\end{math}};
                \node(5)at(5.00,-3.50){\begin{math}a b c\end{math}};
                \draw(0)--(1);
                \draw(2)--(3);
                \draw(3)--(1);
                \draw(4)--(5);
                \draw(5)--(3);
                \draw(6)--(5);
                \node(r)at(1.00,1.31){};
                \draw(r)--(1);
        \end{scope}\end{tikzpicture}};
        \node(E)at(3,-5.75){
            \begin{tikzpicture}
            [baseline=(current bounding box.center),xscale=.3,yscale=.33]
            \begin{scope}
                \node(0)at(0.00,-1.75){};
                \node(2)at(2.00,-5.25){};
                \node(4)at(4.00,-5.25){};
                \node(6)at(6.00,-3.50){};
                \node(1)at(1.00,0.00){\begin{math}a b c\end{math}};
                \node(3)at(3.00,-3.50){\begin{math}a b\end{math}};
                \node(5)at(5.00,-1.75){\begin{math}a b c\end{math}};
                \draw(0)--(1);
                \draw(2)--(3);
                \draw(3)--(5);
                \draw(4)--(3);
                \draw(5)--(1);
                \draw(6)--(5);
                \node(r)at(1.00,1.31){};
                \draw(r)--(1);
            \end{scope}\end{tikzpicture}};
        \node(F)at(0,-8.5){
            \begin{tikzpicture}
            [baseline=(current bounding box.center),xscale=.3,yscale=.32]
            \begin{scope}
                \node(0)at(0.00,-1.75){};
                \node(2)at(2.00,-3.50){};
                \node(4)at(4.00,-5.25){};
                \node(6)at(6.00,-5.25){};
                \node(1)at(1.00,0.00){\begin{math}a b c\end{math}};
                \node(3)at(3.00,-1.75){\begin{math}a b c\end{math}};
                \node(5)at(5.00,-3.50){\begin{math}a b c\end{math}};
                \draw(0)--(1);
                \draw(2)--(3);
                \draw(3)--(1);
                \draw(4)--(5);
                \draw(5)--(3);
                \draw(6)--(5);
                \node(r)at(1.00,1.31){};
                \draw(r)--(1);
            \end{scope}\end{tikzpicture}};
        \draw[EdgeRew](A)--(B);
        \draw[EdgeRew](A)--(C);
        \draw[EdgeRew](B)--(D);
        \draw[EdgeRew](C)--(E);
        \draw[EdgeRew](D)--(F);
        \draw[EdgeRew](E)--(F);
    \end{tikzpicture}}\,,
    \end{equation}
    \end{minipage}

    \begin{minipage}[c]{.29\linewidth}
    \begin{equation}
    \scalebox{.5}{\begin{tikzpicture}
        [baseline=(current bounding box.center),xscale=.8,yscale=1.1]
        \node(A)at(0,0){
            \begin{tikzpicture}
            [baseline=(current bounding box.center),xscale=.3,yscale=.24]
            \begin{scope}
                \node(0)at(0.00,-4.67){};
                \node(2)at(2.00,-4.67){};
                \node(4)at(4.00,-4.67){};
                \node(6)at(6.00,-4.67){};
                \node(1)at(1.00,-2.33){\begin{math}a\end{math}};
                \node(3)at(3.00,0.00){\begin{math}b\end{math}};
                \node(5)at(5.00,-2.33){\begin{math}c\end{math}};
                \draw(0)--(1);
                \draw(1)--(3);
                \draw(2)--(1);
                \draw(4)--(5);
                \draw(5)--(3);
                \draw(6)--(5);
                \node(r)at(3.00,1.75){};
                \draw(r)--(3);
            \end{scope}\end{tikzpicture}};
        \node(B)at(-3,-2.5){
            \begin{tikzpicture}
            [baseline=(current bounding box.center),xscale=.3,yscale=.32]
            \begin{scope}
                \node(0)at(0.00,-1.75){};
                \node(2)at(2.00,-3.50){};
                \node(4)at(4.00,-5.25){};
                \node(6)at(6.00,-5.25){};
                \node(1)at(1.00,0.00){\begin{math}a b\end{math}};
                \node(3)at(3.00,-1.75){\begin{math}a b\end{math}};
                \node(5)at(5.00,-3.50){\begin{math}c\end{math}};
                \draw(0)--(1);
                \draw(2)--(3);
                \draw(3)--(1);
                \draw(4)--(5);
                \draw(5)--(3);
                \draw(6)--(5);
                \node(r)at(1.00,1.31){};
                \draw(r)--(1);
            \end{scope}\end{tikzpicture}};
        \node(C)at(3,-2.5){
            \begin{tikzpicture}
            [baseline=(current bounding box.center),xscale=.3,yscale=.24]
            \begin{scope}
                \node(0)at(0.00,-4.67){};
                \node(2)at(2.00,-4.67){};
                \node(4)at(4.00,-4.67){};
                \node(6)at(6.00,-4.67){};
                \node(1)at(1.00,-2.33){\begin{math}a\end{math}};
                \node(3)at(3.00,0.00){\begin{math}b c\end{math}};
                \node(5)at(5.00,-2.33){\begin{math}b c\end{math}};
                \draw(0)--(1);
                \draw(1)--(3);
                \draw(2)--(1);
                \draw(4)--(5);
                \draw(5)--(3);
                \draw(6)--(5);
                \node(r)at(3.00,1.75){};
                \draw(r)--(3);
            \end{scope}\end{tikzpicture}};
        \node(D)at(-3,-5.5){
            \begin{tikzpicture}
            [baseline=(current bounding box.center),xscale=.3,yscale=.34]
            \begin{scope}
                \node(0)at(0.00,-1.75){};
                \node(2)at(2.00,-3.50){};
                \node(4)at(4.00,-5.25){};
                \node(6)at(6.00,-5.25){};
                \node(1)at(1.00,0.00){\begin{math}a b\end{math}};
                \node(3)at(3.00,-1.75){\begin{math}a b c\end{math}};
                \node(5)at(5.00,-3.50){\begin{math}a b c\end{math}};
                \draw(0)--(1);
                \draw(2)--(3);
                \draw(3)--(1);
                \draw(4)--(5);
                \draw(5)--(3);
                \draw(6)--(5);
                \node(r)at(1.00,1.31){};
                \draw(r)--(1);
        \end{scope}\end{tikzpicture}};
        \node(E)at(3,-5.5){
            \begin{tikzpicture}
            [baseline=(current bounding box.center),xscale=.3,yscale=.32]
            \begin{scope}
                \node(0)at(0.00,-1.75){};
                \node(2)at(2.00,-3.50){};
                \node(4)at(4.00,-5.25){};
                \node(6)at(6.00,-5.25){};
                \node(1)at(1.00,0.00){\begin{math}a b c\end{math}};
                \node(3)at(3.00,-1.75){\begin{math}a b c\end{math}};
                \node(5)at(5.00,-3.50){\begin{math}b c\end{math}};
                \draw(0)--(1);
                \draw(2)--(3);
                \draw(3)--(1);
                \draw(4)--(5);
                \draw(5)--(3);
                \draw(6)--(5);
                \node(r)at(1.00,1.31){};
                \draw(r)--(1);
            \end{scope}\end{tikzpicture}};
        \node(F)at(0,-8.5){
            \begin{tikzpicture}
            [baseline=(current bounding box.center),xscale=.3,yscale=.32]
            \begin{scope}
                \node(0)at(0.00,-1.75){};
                \node(2)at(2.00,-3.50){};
                \node(4)at(4.00,-5.25){};
                \node(6)at(6.00,-5.25){};
                \node(1)at(1.00,0.00){\begin{math}a b c\end{math}};
                \node(3)at(3.00,-1.75){\begin{math}a b c\end{math}};
                \node(5)at(5.00,-3.50){\begin{math}a b c\end{math}};
                \draw(0)--(1);
                \draw(2)--(3);
                \draw(3)--(1);
                \draw(4)--(5);
                \draw(5)--(3);
                \draw(6)--(5);
                \node(r)at(1.00,1.31){};
                \draw(r)--(1);
            \end{scope}\end{tikzpicture}};
        \draw[EdgeRew](A)--(B);
        \draw[EdgeRew](A)--(C);
        \draw[EdgeRew](B)--(D);
        \draw[EdgeRew](C)--(E);
        \draw[EdgeRew](D)--(F);
        \draw[EdgeRew](E)--(F);
    \end{tikzpicture}}\,,
    \end{equation}
    \end{minipage}
    \quad
    \begin{minipage}[c]{.29\linewidth}
    \begin{equation}
    \scalebox{.5}{\begin{tikzpicture}
        [baseline=(current bounding box.center),xscale=.8,yscale=1.1]
        \node(A)at(0,0){
            \begin{tikzpicture}
            [baseline=(current bounding box.center),xscale=.3,yscale=.33]
            \begin{scope}
                \node(0)at(0.00,-1.75){};
                \node(2)at(2.00,-5.25){};
                \node(4)at(4.00,-5.25){};
                \node(6)at(6.00,-3.50){};
                \node(1)at(1.00,0.00){\begin{math}a\end{math}};
                \node(3)at(3.00,-3.50){\begin{math}b\end{math}};
                \node(5)at(5.00,-1.75){\begin{math}c\end{math}};
                \draw(0)--(1);
                \draw(2)--(3);
                \draw(3)--(5);
                \draw(4)--(3);
                \draw(5)--(1);
                \draw(6)--(5);
                \node(r)at(1.00,1.31){};
                \draw(r)--(1);
            \end{scope}\end{tikzpicture}};
        \node(B)at(-3,-2.5){
            \begin{tikzpicture}
            [baseline=(current bounding box.center),xscale=.3,yscale=.32]
            \begin{scope}
                \node(0)at(0.00,-1.75){};
                \node(2)at(2.00,-3.50){};
                \node(4)at(4.00,-5.25){};
                \node(6)at(6.00,-5.25){};
                \node(1)at(1.00,0.00){\begin{math}a\end{math}};
                \node(3)at(3.00,-1.75){\begin{math}b c\end{math}};
                \node(5)at(5.00,-3.50){\begin{math}b c\end{math}};
                \draw(0)--(1);
                \draw(2)--(3);
                \draw(3)--(1);
                \draw(4)--(5);
                \draw(5)--(3);
                \draw(6)--(5);
                \node(r)at(1.00,1.31){};
                \draw(r)--(1);
            \end{scope}\end{tikzpicture}};
        \node(C)at(3,-2.5){
            \begin{tikzpicture}
            [baseline=(current bounding box.center),xscale=.3,yscale=.33]
            \begin{scope}
                \node(0)at(0.00,-1.75){};
                \node(2)at(2.00,-5.25){};
                \node(4)at(4.00,-5.25){};
                \node(6)at(6.00,-3.50){};
                \node(1)at(1.00,0.00){\begin{math}a c\end{math}};
                \node(3)at(3.00,-3.50){\begin{math}b\end{math}};
                \node(5)at(5.00,-1.75){\begin{math}a c\end{math}};
                \draw(0)--(1);
                \draw(2)--(3);
                \draw(3)--(5);
                \draw(4)--(3);
                \draw(5)--(1);
                \draw(6)--(5);
                \node(r)at(1.00,1.31){};
                \draw(r)--(1);
            \end{scope}\end{tikzpicture}};
        \node(D)at(-3,-5.75){
            \begin{tikzpicture}
            [baseline=(current bounding box.center),xscale=.3,yscale=.32]
            \begin{scope}
                \node(0)at(0.00,-1.75){};
                \node(2)at(2.00,-3.50){};
                \node(4)at(4.00,-5.25){};
                \node(6)at(6.00,-5.25){};
                \node(1)at(1.00,0.00){\begin{math}a b c\end{math}};
                \node(3)at(3.00,-1.75){\begin{math}a b c\end{math}};
                \node(5)at(5.00,-3.50){\begin{math}b c\end{math}};
                \draw(0)--(1);
                \draw(2)--(3);
                \draw(3)--(1);
                \draw(4)--(5);
                \draw(5)--(3);
                \draw(6)--(5);
                \node(r)at(1.00,1.31){};
                \draw(r)--(1);
        \end{scope}\end{tikzpicture}};
        \node(E)at(3,-5.75){
            \begin{tikzpicture}
            [baseline=(current bounding box.center),xscale=.3,yscale=.32]
            \begin{scope}
                \node(0)at(0.00,-1.75){};
                \node(2)at(2.00,-3.50){};
                \node(4)at(4.00,-5.25){};
                \node(6)at(6.00,-5.25){};
                \node(1)at(1.00,0.00){\begin{math}a c\end{math}};
                \node(3)at(3.00,-1.75){\begin{math}a b c\end{math}};
                \node(5)at(5.00,-3.50){\begin{math}a b c\end{math}};
                \draw(0)--(1);
                \draw(2)--(3);
                \draw(3)--(1);
                \draw(4)--(5);
                \draw(5)--(3);
                \draw(6)--(5);
                \node(r)at(1.00,1.31){};
                \draw(r)--(1);
            \end{scope}\end{tikzpicture}};
        \node(F)at(0,-8.5){
            \begin{tikzpicture}
            [baseline=(current bounding box.center),xscale=.3,yscale=.32]
            \begin{scope}
                \node(0)at(0.00,-1.75){};
                \node(2)at(2.00,-3.50){};
                \node(4)at(4.00,-5.25){};
                \node(6)at(6.00,-5.25){};
                \node(1)at(1.00,0.00){\begin{math}a b c\end{math}};
                \node(3)at(3.00,-1.75){\begin{math}a b c\end{math}};
                \node(5)at(5.00,-3.50){\begin{math}a b c\end{math}};
                \draw(0)--(1);
                \draw(2)--(3);
                \draw(3)--(1);
                \draw(4)--(5);
                \draw(5)--(3);
                \draw(6)--(5);
                \node(r)at(1.00,1.31){};
                \draw(r)--(1);
            \end{scope}\end{tikzpicture}};
        \draw[EdgeRew](A)--(B);
        \draw[EdgeRew](A)--(C);
        \draw[EdgeRew](B)--(D);
        \draw[EdgeRew](C)--(E);
        \draw[EdgeRew](D)--(F);
        \draw[EdgeRew](E)--(F);
    \end{tikzpicture}}\,,
    \end{equation}
    \end{minipage}
    \quad
    \begin{minipage}[c]{.29\linewidth}
    \begin{equation}
    \scalebox{.5}{\begin{tikzpicture}
    [baseline=(current bounding box.center),xscale=.8,yscale=1.1]
        \node(A)at(0,0){
            \begin{tikzpicture}
            [baseline=(current bounding box.center),xscale=.3,yscale=.32]
            \begin{scope}
                \node(0)at(0.00,-1.75){};
                \node(2)at(2.00,-3.50){};
                \node(4)at(4.00,-5.25){};
                \node(6)at(6.00,-5.25){};
                \node(1)at(1.00,0.00){\begin{math}a\end{math}};
                \node(3)at(3.00,-1.75){\begin{math}b\end{math}};
                \node(5)at(5.00,-3.50){\begin{math}c\end{math}};
                \draw(0)--(1);
                \draw(2)--(3);
                \draw(3)--(1);
                \draw(4)--(5);
                \draw(5)--(3);
                \draw(6)--(5);
                \node(r)at(1.00,1.31){};
                \draw(r)--(1);
            \end{scope}\end{tikzpicture}};
        \node(B)at(-3,-2.5){
            \begin{tikzpicture}
            [baseline=(current bounding box.center),xscale=.3,yscale=.32]
            \begin{scope}
                \node(0)at(0.00,-1.75){};
                \node(2)at(2.00,-3.50){};
                \node(4)at(4.00,-5.25){};
                \node(6)at(6.00,-5.25){};
                \node(1)at(1.00,0.00){\begin{math}a b\end{math}};
                \node(3)at(3.00,-1.75){\begin{math}a b\end{math}};
                \node(5)at(5.00,-3.50){\begin{math}c\end{math}};
                \draw(0)--(1);
                \draw(2)--(3);
                \draw(3)--(1);
                \draw(4)--(5);
                \draw(5)--(3);
                \draw(6)--(5);
                \node(r)at(1.00,1.31){};
                \draw(r)--(1);
            \end{scope}\end{tikzpicture}};
        \node(C)at(3,-2.5){
            \begin{tikzpicture}
            [baseline=(current bounding box.center),xscale=.3,yscale=.32]
            \begin{scope}
                \node(0)at(0.00,-1.75){};
                \node(2)at(2.00,-3.50){};
                \node(4)at(4.00,-5.25){};
                \node(6)at(6.00,-5.25){};
                \node(1)at(1.00,0.00){\begin{math}a\end{math}};
                \node(3)at(3.00,-1.75){\begin{math}b c\end{math}};
                \node(5)at(5.00,-3.50){\begin{math}b c\end{math}};
                \draw(0)--(1);
                \draw(2)--(3);
                \draw(3)--(1);
                \draw(4)--(5);
                \draw(5)--(3);
                \draw(6)--(5);
                \node(r)at(1.00,1.31){};
                \draw(r)--(1);
            \end{scope}\end{tikzpicture}};
        \node(D)at(-3,-5.75){
            \begin{tikzpicture}
            [baseline=(current bounding box.center),xscale=.3,yscale=.32]
            \begin{scope}
                \node(0)at(0.00,-1.75){};
                \node(2)at(2.00,-3.50){};
                \node(4)at(4.00,-5.25){};
                \node(6)at(6.00,-5.25){};
                \node(1)at(1.00,0.00){\begin{math}a b\end{math}};
                \node(3)at(3.00,-1.75){\begin{math}a b c\end{math}};
                \node(5)at(5.00,-3.50){\begin{math}a b c\end{math}};
                \draw(0)--(1);
                \draw(2)--(3);
                \draw(3)--(1);
                \draw(4)--(5);
                \draw(5)--(3);
                \draw(6)--(5);
                \node(r)at(1.00,1.31){};
                \draw(r)--(1);
        \end{scope}\end{tikzpicture}};
        \node(E)at(3,-5.75){
            \begin{tikzpicture}
            [baseline=(current bounding box.center),xscale=.3,yscale=.32]
            \begin{scope}
                \node(0)at(0.00,-1.75){};
                \node(2)at(2.00,-3.50){};
                \node(4)at(4.00,-5.25){};
                \node(6)at(6.00,-5.25){};
                \node(1)at(1.00,0.00){\begin{math}a b c\end{math}};
                \node(3)at(3.00,-1.75){\begin{math}a b c\end{math}};
                \node(5)at(5.00,-3.50){\begin{math}b c\end{math}};
                \draw(0)--(1);
                \draw(2)--(3);
                \draw(3)--(1);
                \draw(4)--(5);
                \draw(5)--(3);
                \draw(6)--(5);
                \node(r)at(1.00,1.31){};
                \draw(r)--(1);
            \end{scope}\end{tikzpicture}};
        \node(F)at(0,-8.5){
            \begin{tikzpicture}
            [baseline=(current bounding box.center),xscale=.3,yscale=.32]
            \begin{scope}
                \node(0)at(0.00,-1.75){};
                \node(2)at(2.00,-3.50){};
                \node(4)at(4.00,-5.25){};
                \node(6)at(6.00,-5.25){};
                \node(1)at(1.00,0.00){\begin{math}a b c\end{math}};
                \node(3)at(3.00,-1.75){\begin{math}a b c\end{math}};
                \node(5)at(5.00,-3.50){\begin{math}a b c\end{math}};
                \draw(0)--(1);
                \draw(2)--(3);
                \draw(3)--(1);
                \draw(4)--(5);
                \draw(5)--(3);
                \draw(6)--(5);
                \node(r)at(1.00,1.31){};
                \draw(r)--(1);
            \end{scope}\end{tikzpicture}};
        \draw[EdgeRew](A)--(B);
        \draw[EdgeRew](A)--(C);
        \draw[EdgeRew](B)--(D);
        \draw[EdgeRew](C)--(E);
        \draw[EdgeRew](D)--(F);
        \draw[EdgeRew](E)--(F);
    \end{tikzpicture}}\,.
    \end{equation}
    \end{minipage}
    \end{subequations}

    Therefore, this shows that all critical pairs of $\Rew_\Pca'$ of
    trees consisting in three internal nodes or less are joinable. Since
    $\Rew_\Pca'$ is of degree $2$ and is terminating
    (Lemma~\ref{lem:terminating_rewrite_rule}), by
    Lemma~\ref{lem:confluent_few_nodes}, this implies that $\Rew_\Pca'$
    is confluent.
\end{proof}
\medskip

In Lemma~\ref{lem:confluent_rewrite_rule}, the condition on $\Pca$ to be
a forest poset is a necessary condition. Indeed, by setting
\begin{equation}
    \Pca :=
    \begin{tikzpicture}
        [baseline=(current bounding box.center),xscale=.4,yscale=.4]
        \node[Vertex](1)at(-.75,0){\begin{math}1\end{math}};
        \node[Vertex](2)at(.75,0){\begin{math}2\end{math}};
        \node[Vertex](3)at(0,-1){\begin{math}3\end{math}};
        \draw[Edge](1)--(3);
        \draw[Edge](2)--(3);
    \end{tikzpicture}\,,
\end{equation}
the critical tree
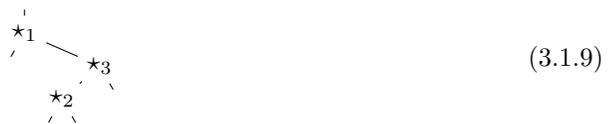
\begin{equation}
    \begin{tikzpicture}
    [baseline=(current bounding box.center),xscale=.25,yscale=.25]
        \node(0)at(0.00,-1.75){};
        \node(2)at(2.00,-5.25){};
        \node(4)at(4.00,-5.25){};
        \node(6)at(6.00,-3.50){};
        \node(1)at(1.00,0.00){\begin{math}\Op_1\end{math}};
        \node(3)at(3.00,-3.50){\begin{math}\Op_2\end{math}};
        \node(5)at(5.00,-1.75){\begin{math}\Op_3\end{math}};
        \draw(0)--(1);
        \draw(2)--(3);
        \draw(3)--(5);
        \draw(4)--(3);
        \draw(5)--(1);
        \draw(6)--(5);
        \node(r)at(1.00,1.75){};
        \draw(r)--(1);
    \end{tikzpicture}
\end{equation}
of $\Rew_\Pca'$ admits the critical pair consisting in the two trees
\begin{equation}
    \begin{tikzpicture}
    [baseline=(current bounding box.center),xscale=.25,yscale=.25]
        \node(0)at(0.00,-1.75){};
        \node(2)at(2.00,-5.25){};
        \node(4)at(4.00,-5.25){};
        \node(6)at(6.00,-3.50){};
        \node(1)at(1.00,0.00){\begin{math}\Op_1\end{math}};
        \node(3)at(3.00,-3.50){\begin{math}\Op_2\end{math}};
        \node(5)at(5.00,-1.75){\begin{math}\Op_1\end{math}};
        \draw(0)--(1);
        \draw(2)--(3);
        \draw(3)--(5);
        \draw(4)--(3);
        \draw(5)--(1);
        \draw(6)--(5);
        \node(r)at(1.00,1.75){};
        \draw(r)--(1);
    \end{tikzpicture}\,,\quad
    \begin{tikzpicture}
    [baseline=(current bounding box.center),xscale=.25,yscale=.28]
        \node(0)at(0.00,-1.75){};
        \node(2)at(2.00,-3.50){};
        \node(4)at(4.00,-5.25){};
        \node(6)at(6.00,-5.25){};
        \node(1)at(1.00,0.00){\begin{math}\Op_1\end{math}};
        \node(3)at(3.00,-1.75){\begin{math}\Op_2\end{math}};
        \node(5)at(5.00,-3.50){\begin{math}\Op_2\end{math}};
        \draw(0)--(1);
        \draw(2)--(3);
        \draw(3)--(1);
        \draw(4)--(5);
        \draw(5)--(3);
        \draw(6)--(5);
        \node(r)at(1.00,1.5){};
        \draw(r)--(1);
    \end{tikzpicture}\,.
\end{equation}
Since these two trees are normal forms of $\Rew_\Pca'$, this critical
pair is not joinable, hence showing that $\Rew_\Pca'$ is not confluent.
\medskip

\subsubsection{Koszulity}

\begin{Theorem} \label{thm:koszulity}
    Let $\Pca$ be a forest poset. Then, the operad $\As(\Pca)$ is Koszul
    and the set $\NormalF(\Pca)$ of syntax trees of
    $\Free(\FreeGen_\Pca^\Op)$ forms a Poincaré-Birkhoff-Witt basis
    of~$\As(\Pca)$.
\end{Theorem}
\begin{proof}
    By Lemma~\ref{lem:rewrite_rule_space_relations}, the rewrite rule
    $\Rew_\Pca'$, defined by~\eqref{equ:rewrite_1}
    and~\eqref{equ:rewrite_2}, is an orientation of the space of
    relations $\FreeRel_\Pca^\Op$ of $\As(\Pca)$. Moreover, by
    Lemmas~\ref{lem:terminating_rewrite_rule}
    and~\ref{lem:confluent_rewrite_rule}, this rewrite rule is convergent
    when $\Pca$ is a forest poset. Therefore, by
    Lemma~\ref{lem:koszulity_criterion_pbw}, $\As(\Pca)$ is Koszul.
    Finally, by Lemma~\ref{lem:normal_forms}, $\NormalF(\Pca)$ is the
    set of all normal forms of $\Rew_\Pca'$ and, by definition, forms a
    Poincaré-Birkhoff-Witt basis of~$\As(\Pca)$.
\end{proof}
\medskip

\subsection{Dimensions and realization}
The Koszulity, and more specifically the existence of a
Poincaré-Birkhoff-Witt basis $\NormalF(\Pca)$ highlighted by
Theorem~\ref{thm:koszulity} for $\As(\Pca)$ when $\Pca$ is a forest
poset, lead to a combinatorial realization of $\As(\Pca)$. Before
describing this realization, we shall provide a functional equation for
the Hilbert series of $\As(\Pca)$.
\medskip

\subsubsection{Dimensions}

\begin{Proposition} \label{prop:hilbert_series}
    Let $\Pca$ be a forest poset. Then, the Hilbert series
    $\Hilbert_\Pca(t)$ of $\As(\Pca)$ satisfies
    \begin{equation} \label{equ:hilbert_series}
        \Hilbert_\Pca(t) = t + \sum_{a \in \Pca} \Hilbert_\Pca^a(t),
    \end{equation}
    where for all $a \in \Pca$, the $\Hilbert_\Pca^a(t)$ satisfy
    \begin{equation} \label{equ:hilbert_series_h_a}
        \Hilbert_\Pca^a(t) =
        \left(t + \bar \Hilbert_\Pca^a(t)\right)
        \left(t + \bar \Hilbert_\Pca^a(t) + \Hilbert_\Pca^a(t)\right),
    \end{equation}
    and for all $a \in \Pca$, the $\bar \Hilbert_\Pca^a(t)$ satisfy
    \begin{equation} \label{equ:hilbert_series_h_a_incomp}
        \bar \Hilbert_\Pca^a(t) =
        \sum_{\substack{b \in \Pca \\ a \not \Ord_\Pca b \\
                b \not \Ord_\Pca a}}
        \Hilbert_\Pca^b(t).
    \end{equation}
\end{Proposition}
\begin{proof}
    By Theorem~\ref{thm:koszulity}, $\As(\Pca)$ admits the set
    $\NormalF(\Pca)$ as a Poincaré-Birkhoff-Witt basis, whose definition
    appears in the statement of Lemma~\ref{lem:normal_forms}. Then,
    by~\cite{Hof10}, for any $n \geq 1$, the dimension of $\As(\Pca)(n)$
    is equal to the cardinality of $\NormalF(\Pca)(n)$.
    \smallskip

    To enumerate the trees of $\NormalF(\Pca)$, let us consider their
    description provided by Lemma~\ref{lem:normal_forms} and their
    recursive general form provided by~\eqref{equ:normal_forms}. In this
    manner, let $\Tfr$ be a tree of $\NormalF(\Pca)$ such that its root
    is labeled by $\Op_a$, $a \in \Pca$. Then, $\Tfr$ has as a left
    subtree an element of $\NormalF(\Pca)$ equal to the leaf or with a
    root labeled by $\Op_b$, $b \in \Pca$, such that $a$ and $b$ are
    incomparable in $\Pca$. Moreover, $\Tfr$ has as a right subtree an
    element of $\NormalF(\Pca)$ equal to the leaf, or with a root
    labeled by $\Op_c$, $c \in \Pca$, such that $a$ and $c$ are
    incomparable in $\Pca$, or with a root labeled by $\Op_a$. This
    explains~\eqref{equ:hilbert_series_h_a} since $\Hilbert_\Pca^a(t)$
    is the generating series of the trees of $\NormalF(\Pca)$ such that
    their roots are labeled by $\Op_a$, $a \in \Pca$, wherein
    $\bar \Hilbert_\Pca^a(t)$ is the series of trees of $\NormalF(\Pca)$
    such that their roots are labeled by $\Op_b$, $b \in \Pca$, where
    $a$ and $b$ are incomparable in $\Pca$. This also
    explains~\eqref{equ:hilbert_series_h_a_incomp}. Finally, since any
    tree of $\NormalF(\Pca)$ is the leaf or has a root labeled by a
    $\Op_a$, $a \in \Pca$, the generating series of $\NormalF(\Pca)$
    satisfies~\eqref{equ:hilbert_series}.
\end{proof}
\medskip

For instance, let us use Proposition~\ref{prop:hilbert_series} for the
operad $\As(\Pca)$ when $\Pca$ is the total order on the set $[\ell]$,
$\ell \geq 0$. This operad is the multiassociative operad~\cite{Gir15b},
whose definition is recalled in Section~\ref{subsubsec:first_examples}.
By~\eqref{equ:hilbert_series_h_a_incomp}, we have
\begin{equation}
    \bar \Hilbert_\Pca^a(t) = 0,
    \qquad a \in [\ell],
\end{equation}
and hence, by~\eqref{equ:hilbert_series_h_a},
\begin{equation}
    \Hilbert_\Pca^a(t) = \frac{t^2}{1 - t},
    \qquad a \in [\ell].
\end{equation}
Then, by~\eqref{equ:hilbert_series}, the Hilbert series of $\As(\Pca)$
satisfies
\begin{equation}
    \Hilbert_\Pca(t) = t + \frac{\ell t^2}{1 - t},
    \qquad \ell \geq 0.
\end{equation}
\medskip

Let us use Proposition~\ref{prop:hilbert_series} for the operad
$\As(\Pca)$ when $\Pca$ is the trivial poset on the set $[\ell]$,
$\ell \geq 0$. This operad is the dual multiassociative
operad~\cite{Gir15b}, whose definition is recalled in
Section~\ref{subsubsec:first_examples}.
By~\eqref{equ:hilbert_series_h_a_incomp}, one has
\begin{equation}
    \bar \Hilbert_\Pca^a(t) =
    \sum_{\substack{b \in [\ell] \\ b \ne a}} \bar \Hilbert_\Pca^b(t),
    \qquad a \in [\ell],
\end{equation}
implying, by~\eqref{equ:hilbert_series}, that
\begin{equation}
    \bar \Hilbert_\Pca^a(t) = \Hilbert_\Pca(t) - t - \Hilbert_\Pca^a(t),
    \qquad a \in [\ell].
\end{equation}
Now, by~\eqref{equ:hilbert_series_h_a}, we obtain
\begin{equation}
    \Hilbert_\Pca^a(t) =
    \frac{\Hilbert_\Pca(t)^2}{1 + \Hilbert_\Pca(t)},
    \qquad a \in [\ell].
\end{equation}
Therefore, by~\eqref{equ:hilbert_series}, the
Hilbert series of $\As(\Pca)$ satisfies the quadratic functional
equation
\begin{equation}
    t + (t - 1)\Hilbert_\Pca(t) + (\ell - 1) \Hilbert_\Pca(t)^2 = 0,
    \qquad \ell \geq 0,
\end{equation}
and expresses as
\begin{equation}
    \Hilbert_\Pca(t) =
    \frac{1 - t - \sqrt{1 + (2 - 4 \ell) t + t^2}}
    {2 (\ell - 1)},
    \qquad \ell = 0 \mbox{ or } \ell \geq 2.
\end{equation}
The dimensions of the first homogeneous components of $\As(\Pca)$ are
\begin{equation}
    1, 2, 6, 22, 90, 394, 1806, 8558, 41586, 206098,
    \qquad \ell = 2,
\end{equation}
\begin{equation}
    1, 3, 15, 93, 645, 4791, 37275, 299865, 2474025, 20819307,
    \qquad \ell = 3,
\end{equation}
\begin{equation}
    1, 4, 28, 244, 2380, 24868, 272188, 3080596, 35758828, 423373636,
    \qquad \ell = 4,
\end{equation}
\begin{equation}
    1, 5, 45, 505, 6345, 85405, 1204245, 17558705, 262577745, 4005148405,
    \qquad \ell = 5.
\end{equation}
The first one is Sequence~\Sloane{A006318}, the second one is
Sequence~\Sloane{A103210}, and the fourth one is Sequence~\Sloane{A103211}
and the last one is Sequence~\Sloane{A133305} of~\cite{Slo}.
\medskip

Finally, let us use Proposition~\ref{prop:hilbert_series} for the operad
$\As(\Pca)$ when $\Pca$ is the forest poset
\begin{equation}
    \Pca :=
    \begin{tikzpicture}
        [baseline=(current bounding box.center),xscale=.5,yscale=.5]
        \node[Vertex](1)at(0,0){\begin{math}1\end{math}};
        \node[Vertex](2)at(0,-1){\begin{math}2\end{math}};
        \node[Vertex](3)at(1,0){\begin{math}3\end{math}};
        \node[Vertex](4)at(1,-1){\begin{math}4\end{math}};
        \draw[Edge](1)--(2);
        \draw[Edge](3)--(4);
    \end{tikzpicture}\,.
\end{equation}
By~\eqref{equ:hilbert_series_h_a_incomp}, one has
\begin{subequations}
\begin{equation}
    \bar \Hilbert_\Pca^1(t) = \bar \Hilbert_\Pca^2(t) =
    \Hilbert_\Pca^3(t) + \Hilbert_\Pca^4(t),
\end{equation}
\begin{equation}
    \bar \Hilbert_\Pca^3(t) = \bar \Hilbert_\Pca^4(t) =
    \Hilbert_\Pca^1(t) + \Hilbert_\Pca^2(t),
\end{equation}
\end{subequations}
and, by~\eqref{equ:hilbert_series_h_a} and straightforward computations,
we obtain that
\begin{equation}
    \Hilbert_\Pca^1(t) = \Hilbert_\Pca^2(t) =
    \Hilbert_\Pca^3(t) = \Hilbert_\Pca^4(t),
\end{equation}
so that the Hilbert series of $\As(\Pca)$ satisfies the quadratic
functional equation
\begin{equation}
    \frac{1}{2}t + \frac{1}{4}t^2 +
    \left(t - \frac{1}{2}\right)\Hilbert_\Pca(t) +
    \frac{3}{4}\Hilbert_\Pca(t)^2 = 0.
\end{equation}
This Hilbert series expresses as
\begin{equation}
    \Hilbert_\Pca(t) =
    \frac{1 - 2t - \sqrt{1 - 10t + t^2}}{3},
\end{equation}
and the dimensions of the first homogeneous components of $\As(\Pca)$
are
\begin{equation}
    1, 4, 20, 124, 860, 6388, 49700, 399820, 3298700, 27759076.
\end{equation}
Terms of this sequence are the ones of Sequence~\Sloane{A107841}
of~\cite{Slo} multiplied by $2$.
\medskip

\subsubsection{Realization} \label{subsubsec:realization_as_poset_forest}
Let us describe a combinatorial realization of $\As(\Pca)$ when $\Pca$
is a forest poset in terms of Schröder trees with a certain labeling
and through an algorithm to compute their partial composition. Recall
that a {\em Schröder tree}~\cite{Sta11} is a planar rooted tree wherein
all internal nodes have two or more children.
\medskip

If $\Pca$ is a poset (not necessarily a forest poset just now), a
{\em $\Pca$-Schröder tree} is a Schröder tree where internal nodes are
labeled on $\Pca$. For any element $a$ of $\Pca$ and any $n \geq 2$, we
denote by $\Corolla_a^n$ the $\Pca$-Schröder tree consisting in a single
internal node labeled by $a$ attached to $n$ leaves. We call these trees
{\em $\Pca$-corollas}. A {\em $\Pca$-alternating Schröder tree} is a
$\Pca$-Schröder tree $\Tfr$ such that for any internal node $y$ of $\Tfr$
having a father $x$, the labels of $x$ and $y$ are incomparable in
$\Pca$. We denote by $\SetASchr(\Pca)$ the set of $\Pca$-alternating
Schröder trees and by $\SetASchr(\Pca)(n)$, $n \geq 1$, the set
$\SetASchr(\Pca)$ restricted to trees with exactly $n$ leaves. Any tree
$\Tfr$ of $\SetASchr(\Pca)$ different from the leaf is of the recursive
unique general form
\begin{equation} \label{equ:description_aschr}
    \Tfr =
    \begin{tikzpicture}
        [baseline=(current bounding box.center),xscale=.6,yscale=.3]
        \node(0)at(0.00,-2.00){\footnotesize \begin{math}\Sfr_1\end{math}};
        \node(3)at(2.00,-2.00){\footnotesize \begin{math}\Sfr_\ell\end{math}};
        \node[Node](1)at(1.00,0.00){\begin{math}a\end{math}};
        \draw[Edge](0)--(1);
        \draw[Edge](3)--(1);
        \node(r)at(1.00,1.25){};
        \draw[Edge](r)--(1);
        \node[below of=1,node distance=.6cm]{\footnotesize\begin{math}\dots\end{math}};
    \end{tikzpicture}\,,
\end{equation}
where $a \in \Pca$ and for any $i \in [\ell]$, $\Sfr_i$ is a tree of
$\SetASchr(\Pca)$ such that $\Sfr_i$ is a leaf or its root is labeled by
a $b \in \Pca$ and $a$ and $b$ and incomparable in~$\Pca$.
\medskip

Relying on the description of the elements of $\NormalF(\Pca)$ provided
by Lemma~\ref{lem:normal_forms} and on their recursive general form
provided by~\eqref{equ:normal_forms}, let us consider the map
\begin{equation}
    \ASchrMap_\Pca : \NormalF(\Pca)(n) \to \SetASchr(\Pca)(n),
    \qquad n \geq 1,
\end{equation}
defined recursively by sending the leaf to the leaf and, for any tree
$\Tfr$ of $\NormalF(\Pca)$ different from the leaf, by
\begin{equation} \label{equ:recursive_description_aschr_map}
    \ASchrMap_\Pca(\Tfr) =
    \ASchrMap_\Pca
    \left(
    \begin{tikzpicture}
        [baseline=(current bounding box.center),xscale=.4,yscale=.35]
        \node(0)at(0.00,-1.67){\begin{math}\Sfr_1\end{math}};
        \node(2)at(2.00,-3.33){\begin{math}\Sfr_{\ell - 1}\end{math}};
        \node(4)at(4.00,-3.33){\begin{math}\Sfr_\ell\end{math}};
        \node(1)at(1.00,0.00){\begin{math}\Op_a\end{math}};
        \node(3)at(3.00,-1.67){\begin{math}\Op_a\end{math}};
        \draw(0)--(1);
        \draw(2)--(3);
        \draw[densely dashed](3)--(1);
        \draw(4)--(3);
        \node(r)at(1.00,1.25){};
        \draw(r)--(1);
    \end{tikzpicture}
    \right) :=
    \begin{tikzpicture}
        [baseline=(current bounding box.center),xscale=.9,yscale=.3]
        \node(0)at(0.00,-2.00)
            {\footnotesize \begin{math}\ASchrMap_\Pca(\Sfr_1)\end{math}};
        \node(3)at(2.00,-2.00)
            {\footnotesize \begin{math}\ASchrMap_\Pca(\Sfr_\ell)\end{math}};
        \node[Node](1)at(1.00,0.00){\begin{math}a\end{math}};
        \draw[Edge](0)--(1);
        \draw[Edge](3)--(1);
        \node(r)at(1.00,1.25){};
        \draw[Edge](r)--(1);
        \node[below of=1,node distance=.6cm]{\footnotesize\begin{math}\dots\end{math}};
    \end{tikzpicture}\,,
\end{equation}
where $a \in \Pca$ and, in the syntax tree
of~\eqref{equ:recursive_description_aschr_map}, the dashed edge denotes
a right comb tree wherein internal nodes are labeled by $\Op_a$ and for
any $i \in [\ell]$, $\Sfr_i$ is a tree of $\NormalF(\Pca)$ such that
$\Sfr_i$ is the leaf or its root is labeled by $\Op_b$, $b \in \Pca$,
and $a$ and $b$ are incomparable in $\Pca$. It is immediate that
$\ASchrMap_\Pca(\Tfr)$ is a $\Pca$-alternating Schröder tree, so that
$\ASchrMap_\Pca$ is a well-defined map.
\medskip

\begin{Lemma} \label{lem:bijection_pbw_basis_aschr}
    Let $\Pca$ be a poset. Then, for any $n \geq 1$, the map
    $\ASchrMap_\Pca$ is a bijection between the set of syntax trees of
    $\NormalF(\Pca)(n)$ with $n$ leaves and the set $\SetASchr(\Pca)(n)$
    of $\Pca$-alternating Schröder trees with $n$ leaves.
\end{Lemma}
\begin{proof}
    The existence of the map
    \begin{math}
        \ASchrMap_\Pca^{-1} : \SetASchr(\Pca)(n) \to \NormalF(\Pca)(n),
    \end{math}
    $n \geq 1$, being the inverse of $\ASchrMap_\Pca$, follows by
    structural induction on the trees of $\SetASchr(\Pca)$ and
    $\NormalF(\Pca)$, relying on their respective recursive descriptions
    provided by~\eqref{equ:description_aschr} and~\eqref{equ:normal_forms}.
    The statement of the lemma follows.
\end{proof}
\medskip

In order to define a partial composition for $\Pca$-alternating Schröder
trees, we introduce the following rewrite rule. When $\Pca$ is a forest
poset, consider the rewrite rule $\RewS_\Pca'$ on $\Pca$-Schröder trees
(not necessarily $\Pca$-alternating Schröder trees) satisfying
\begin{equation} \label{equ:rewrite_schroder}
    \begin{tikzpicture}
        [baseline=(current bounding box.center),xscale=.35,yscale=.25]
        \node[Leaf](0)at(0.00,-2.00){};
        \node[Leaf](2)at(1.00,-4.00){};
        \node[Leaf](4)at(3.00,-4.00){};
        \node[Leaf](5)at(4.00,-2.00){};
        \node[Node](1)at(2.00,0.00){\begin{math}b\end{math}};
        \node[Node](3)at(2.00,-2.00){\begin{math}a\end{math}};
        \draw[Edge](0)--(1);
        \draw[Edge](2)--(3);
        \draw[Edge](3)--(1);
        \draw[Edge](4)--(3);
        \draw[Edge](5)--(1);
        \node(r)at(2.00,1.50){};
        \draw[Edge](r)--(1);
        \node[right of=0,node distance=.32cm]
            {\footnotesize\begin{math}\dots\end{math}};
        \node[left of=5,node distance=.32cm]
            {\footnotesize\begin{math}\dots\end{math}};
        \node[right of=2,node distance=.4cm]
            {\footnotesize\begin{math}\dots\end{math}};
    \end{tikzpicture}
    \enspace \RewS_\Pca' \enspace
    \begin{tikzpicture}
        [baseline=(current bounding box.center),xscale=.45,yscale=.5]
        \node[Leaf](0)at(0.00,-1.50){};
        \node[Leaf](2)at(2.00,-1.50){};
        \node[Node](1)at(1.00,0.00){\begin{math}a \Min_\Pca b\end{math}};
        \draw[Edge](0)--(1);
        \draw[Edge](2)--(1);
        \node(r)at(1.00,1.5){};
        \draw[Edge](r)--(1);
        \node[below of=1,node distance=.8cm]
            {\footnotesize\begin{math}\dots\end{math}};
    \end{tikzpicture},
    \qquad
    a, b \in \Pca \mbox{ and }
    (a \Ord_\Pca b \mbox{ or } b \Ord_\Pca a).
\end{equation}
Equation~\eqref{equ:example_schroder_rewriting_steps} shows examples
of steps of rewritings by $\RewS_\Pca'$ for the poset $\Pca$ defined
in~\eqref{equ:example_forest_poset}.
\medskip

\begin{Lemma} \label{lem:rewrite_rule_schroder_trees}
    Let $\Pca$ be a poset. Then, the rewrite rule $\RewS_\Pca'$ on
    $\Pca$-Schröder trees is terminating and the set of normal forms of
    $\RewS_\Pca'$ is the set of $\Pca$-alternating Schröder trees.
    Moreover, when~$\Pca$ is a forest poset, $\RewS_\Pca'$ is confluent.
\end{Lemma}
\begin{proof}
    Assume first that $\Pca$ is only a poset. By definition of
    $\RewS_\Pca'$, if $\Sfr$ and $\Tfr$ are two $\Pca$-Schröder trees
    such that $\Sfr \RewS_\Pca \Tfr$, $\Tfr$ has less internal nodes
    than $\Sfr$. This implies that $\RewS_\Pca'$ is terminating.
    \smallskip

    Since $\RewS_\Pca'$ is terminating, its normal forms are by
    definition the $\Pca$-Schröder trees that cannot be rewritten by
    $\RewS_\Pca'$. These trees are exactly the ones avoiding the
    patterns appearing as left member in~\eqref{equ:rewrite_schroder}.
    Hence, if $\Tfr$ is a normal form of $\RewS_\Pca'$ and $x$ and $y$
    are two internal nodes of $\Tfr$ such that $y$ is a child of $x$,
    the labels of $x$ and $y$ are incomparable in $\Pca$. Therefore,
    $\Tfr$ is a $\Pca$-alternating Schröder tree.
    \smallskip

    Let us now prove that $\RewS_\Pca'$ is confluent when $\Pca$ is
    a forest poset. The $\Pca$-Schröder trees
    
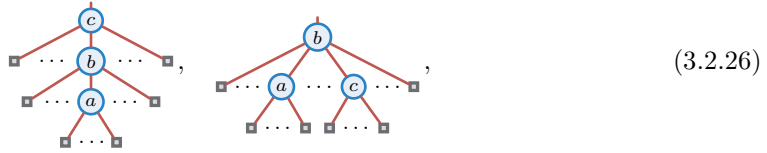
\begin{equation} \label{equ:rewrite_rule_schroder_trees_critical_trees}
        \begin{tikzpicture}
            [baseline=(current bounding box.center),xscale=.34,yscale=.24]
            \node[Leaf](0)at(0.00,-2.25){};
            \node[Leaf](2)at(0.50,-4.50){};
            \node[Leaf](4)at(2.00,-6.75){};
            \node[Leaf](6)at(4.00,-6.75){};
            \node[Leaf](7)at(5.50,-4.50){};
            \node[Leaf](8)at(6.00,-2.25){};
            \node[Node](1)at(3.00,0.00){\begin{math}c\end{math}};
            \node[Node](3)at(3.00,-2.25){\begin{math}b\end{math}};
            \node[Node](5)at(3.00,-4.50){\begin{math}a\end{math}};
            \draw[Edge](0)--(1);
            \draw[Edge](2)--(3);
            \draw[Edge](3)--(1);
            \draw[Edge](4)--(5);
            \draw[Edge](5)--(3);
            \draw[Edge](6)--(5);
            \draw[Edge](7)--(3);
            \draw[Edge](8)--(1);
            \node(r)at(3.00,1.5){};
            \draw[Edge](r)--(1);
            \node[right of=0,node distance=.5cm]
                {\footnotesize\begin{math}\dots\end{math}};
            \node[left of=8,node distance=.5cm]
                {\footnotesize\begin{math}\dots\end{math}};
            \node[right of=2,node distance=.4cm]
                {\footnotesize\begin{math}\dots\end{math}};
            \node[left of=7,node distance=.4cm]
                {\footnotesize\begin{math}\dots\end{math}};
            \node[right of=4,node distance=.35cm]
                {\footnotesize\begin{math}\dots\end{math}};
        \end{tikzpicture}\,,\quad
        \begin{tikzpicture}
            [baseline=(current bounding box.center),xscale=.32,yscale=.22]
            \node[Leaf](0)at(0.00,-3.00){};
            \node[Leaf](1)at(1.25,-5.50){};
            \node[Leaf](3)at(3.50,-5.50){};
            \node[Leaf](5)at(4.50,-5.50){};
            \node[Leaf](7)at(6.75,-5.50){};
            \node[Leaf](8)at(8.00,-3.00){};
            \node[Node](2)at(2.50,-3.00){\begin{math}a\end{math}};
            \node[Node](4)at(4.00,0.00){\begin{math}b\end{math}};
            \node[Node](6)at(5.50,-3.00){\begin{math}c\end{math}};
            \draw[Edge](0)--(4);
            \draw[Edge](1)--(2);
            \draw[Edge](2)--(4);
            \draw[Edge](3)--(2);
            \draw[Edge](5)--(6);
            \draw[Edge](6)--(4);
            \draw[Edge](7)--(6);
            \draw[Edge](8)--(4);
            \node(r)at(4.00,1.75){};
            \draw[Edge](r)--(4);
            \node[right of=0,node distance=.35cm]{\footnotesize\begin{math}\dots\end{math}};
            \node[left of=8,node distance=.35cm]{\footnotesize\begin{math}\dots\end{math}};
            \node[right of=0,node distance=1.3cm]{\footnotesize\begin{math}\dots\end{math}};
            \node[right of=1,node distance=.35cm]{\footnotesize\begin{math}\dots\end{math}};
            \node[right of=5,node distance=.35cm]{\footnotesize\begin{math}\dots\end{math}};
        \end{tikzpicture}\,,
    \end{equation}
    are critical trees of $\RewS_\Pca'$ where $a$ and $b$ are comparable
    in $\Pca$, and $b$ and $c$ are comparable in $\Pca$. Moreover, all
    critical trees of $\RewS_\Pca'$ consisting in three or less internal
    nodes are one of the two trees
    of~\eqref{equ:rewrite_rule_schroder_trees_critical_trees}.
    \smallskip

    During the rest of this proof, to gain clarity in our drawings of
    trees, we will represent labels of internal nodes by words
    $a_1 \dots a_k$ of elements of $\Pca$ to specify the label
    $a_1 \Min_\Pca \dots \Min_\Pca a_k$. Because given a tree $\Tfr$
    of~\eqref{equ:rewrite_rule_schroder_trees_critical_trees} there are
    exactly two ways to rewrite $\Tfr$ in one step by $\RewS_\Pca'$,
    each critical tree of $\RewS_\Pca'$ gives rise to one critical pair.
    By denoting
    by~$\raisebox{.2em}{\tikz{\draw[EdgeRew](0,0)--(.4,0);}}$ the
    relation $\RewS_\Pca$, the fact that $\Pca$ is a forest poset and
    Lemma~\ref{lem:forest_poset_max_three_elements} imply that we have
    the two graphs of rewriting by~$\RewS_\Pca'$

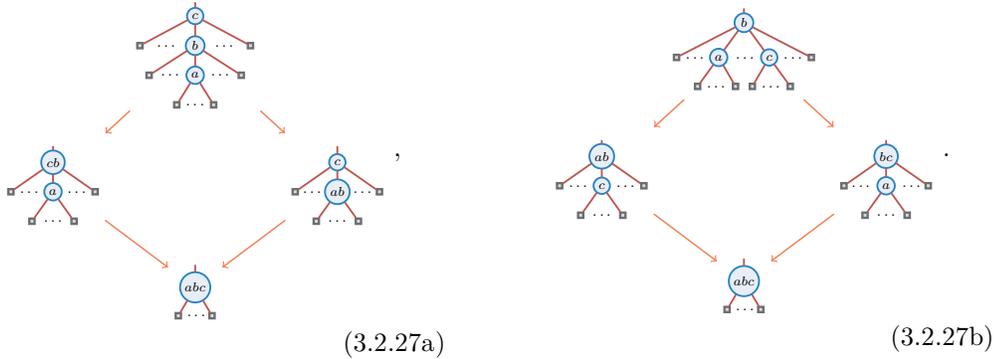
    \begin{subequations}
    \begin{minipage}[c]{.45\linewidth}
    \begin{equation}
    \scalebox{.7}{\begin{tikzpicture}
    [baseline=(current bounding box.center),scale=.9]
        \node(A)at(0,0){
            \begin{tikzpicture}
            [baseline=(current bounding box.center),xscale=.35,yscale=.25]
            \begin{scope}
            \node[Leaf](0)at(0.00,-2.25){};
            \node[Leaf](2)at(0.50,-4.50){};
            \node[Leaf](4)at(2.00,-6.75){};
            \node[Leaf](6)at(4.00,-6.75){};
            \node[Leaf](7)at(5.50,-4.50){};
            \node[Leaf](8)at(6.00,-2.25){};
            \node[Node](1)at(3.00,0.00){\begin{math}c\end{math}};
            \node[Node](3)at(3.00,-2.25){\begin{math}b\end{math}};
            \node[Node](5)at(3.00,-4.50){\begin{math}a\end{math}};
            \draw[Edge](0)--(1);
            \draw[Edge](2)--(3);
            \draw[Edge](3)--(1);
            \draw[Edge](4)--(5);
            \draw[Edge](5)--(3);
            \draw[Edge](6)--(5);
            \draw[Edge](7)--(3);
            \draw[Edge](8)--(1);
            \node(r)at(3.00,1.5){};
            \draw[Edge](r)--(1);
            \node[right of=0,node distance=.5cm]
                {\footnotesize\begin{math}\dots\end{math}};
            \node[left of=8,node distance=.5cm]
                {\footnotesize\begin{math}\dots\end{math}};
            \node[right of=2,node distance=.4cm]
                {\footnotesize\begin{math}\dots\end{math}};
            \node[left of=7,node distance=.4cm]
                {\footnotesize\begin{math}\dots\end{math}};
            \node[right of=4,node distance=.35cm]
                {\footnotesize\begin{math}\dots\end{math}};
            \end{scope}\end{tikzpicture}};
        \node(B)at(-3,-2.75){
            \begin{tikzpicture}
            [baseline=(current bounding box.center),xscale=.4,yscale=.28]
            \begin{scope}
                \node[Leaf](0)at(0.00,-2.00){};
                \node[Leaf](2)at(1.00,-4.00){};
                \node[Leaf](4)at(3.00,-4.00){};
                \node[Leaf](5)at(4.00,-2.00){};
                \node[Node](1)at(2.00,0.00){\begin{math}c b\end{math}};
                \node[Node](3)at(2.00,-2.00){\begin{math}a\end{math}};
                \draw[Edge](0)--(1);
                \draw[Edge](2)--(3);
                \draw[Edge](3)--(1);
                \draw[Edge](4)--(3);
                \draw[Edge](5)--(1);
                \node(r)at(2.00,1.50){};
                \draw[Edge](r)--(1);
                \node[right of=0,node distance=.35cm]
                    {\footnotesize\begin{math}\dots\end{math}};
                \node[left of=5,node distance=.35cm]
                    {\footnotesize\begin{math}\dots\end{math}};
                \node[right of=2,node distance=.4cm]
                    {\footnotesize\begin{math}\dots\end{math}};
            \end{scope}\end{tikzpicture}};
        \node(C)at(3,-2.75){
            \begin{tikzpicture}
            [baseline=(current bounding box.center),xscale=.4,yscale=.28]
            \begin{scope}
                \node[Leaf](0)at(0.00,-2.00){};
                \node[Leaf](2)at(1.00,-4.00){};
                \node[Leaf](4)at(3.00,-4.00){};
                \node[Leaf](5)at(4.00,-2.00){};
                \node[Node](1)at(2.00,0.00){\begin{math}c\end{math}};
                \node[Node](3)at(2.00,-2.00){\begin{math}a b\end{math}};
                \draw[Edge](0)--(1);
                \draw[Edge](2)--(3);
                \draw[Edge](3)--(1);
                \draw[Edge](4)--(3);
                \draw[Edge](5)--(1);
                \node(r)at(2.00,1.50){};
                \draw[Edge](r)--(1);
                \node[right of=0,node distance=.32cm]
                    {\footnotesize\begin{math}\dots\end{math}};
                \node[left of=5,node distance=.32cm]
                    {\footnotesize\begin{math}\dots\end{math}};
                \node[right of=2,node distance=.4cm]
                    {\footnotesize\begin{math}\dots\end{math}};
            \end{scope}\end{tikzpicture}};
        \node(D)at(0,-5){
            \begin{tikzpicture}
            [baseline=(current bounding box.center),xscale=.32,yscale=.36]
            \begin{scope}
                \node[Leaf](0)at(0.00,-1.50){};
                \node[Leaf](2)at(2.00,-1.50){};
                \node[Node](1)at(1.00,0.00){\begin{math}a b c\end{math}};
                \draw[Edge](0)--(1);
                \draw[Edge](2)--(1);
                \node(r)at(1.00,1.5){};
                \draw[Edge](r)--(1);
                \node[right of=0,node distance=.34cm]
                    {\footnotesize\begin{math}\dots\end{math}};
            \end{scope}\end{tikzpicture}};
        \draw[EdgeRew](A)--(B);
        \draw[EdgeRew](A)--(C);
        \draw[EdgeRew](B)--(D);
        \draw[EdgeRew](C)--(D);
    \end{tikzpicture}}\,,
    \end{equation}
    \end{minipage}
    \qquad
    \begin{minipage}[c]{.45\linewidth}
    \begin{equation}
    \scalebox{.7}{\begin{tikzpicture}
    [baseline=(current bounding box.center),scale=.9]
        \node(A)at(0,0){
            \begin{tikzpicture}
            [baseline=(current bounding box.center),xscale=.32,yscale=.22]
            \begin{scope}
                \node[Leaf](0)at(0.00,-3.00){};
                \node[Leaf](1)at(1.25,-5.50){};
                \node[Leaf](3)at(3.50,-5.50){};
                \node[Leaf](5)at(4.50,-5.50){};
                \node[Leaf](7)at(6.75,-5.50){};
                \node[Leaf](8)at(8.00,-3.00){};
                \node[Node](2)at(2.50,-3.00){\begin{math}a\end{math}};
                \node[Node](4)at(4.00,0.00){\begin{math}b\end{math}};
                \node[Node](6)at(5.50,-3.00){\begin{math}c\end{math}};
                \draw[Edge](0)--(4);
                \draw[Edge](1)--(2);
                \draw[Edge](2)--(4);
                \draw[Edge](3)--(2);
                \draw[Edge](5)--(6);
                \draw[Edge](6)--(4);
                \draw[Edge](7)--(6);
                \draw[Edge](8)--(4);
                \node(r)at(4.00,1.75){};
                \draw[Edge](r)--(4);
                \node[right of=0,node distance=.35cm]
                    {\footnotesize\begin{math}\dots\end{math}};
                \node[left of=8,node distance=.35cm]
                    {\footnotesize\begin{math}\dots\end{math}};
                \node[right of=0,node distance=1.3cm]
                    {\footnotesize\begin{math}\dots\end{math}};
                \node[right of=1,node distance=.35cm]
                    {\footnotesize\begin{math}\dots\end{math}};
                \node[right of=5,node distance=.35cm]
                    {\footnotesize\begin{math}\dots\end{math}};
            \end{scope}\end{tikzpicture}};
        \node(B)at(-3,-2.75){
            \begin{tikzpicture}
            [baseline=(current bounding box.center),xscale=.4,yscale=.28]
            \begin{scope}
                \node[Leaf](0)at(0.00,-2.00){};
                \node[Leaf](2)at(1.00,-4.00){};
                \node[Leaf](4)at(3.00,-4.00){};
                \node[Leaf](5)at(4.00,-2.00){};
                \node[Node](1)at(2.00,0.00){\begin{math}a b\end{math}};
                \node[Node](3)at(2.00,-2.00){\begin{math}c\end{math}};
                \draw[Edge](0)--(1);
                \draw[Edge](2)--(3);
                \draw[Edge](3)--(1);
                \draw[Edge](4)--(3);
                \draw[Edge](5)--(1);
                \node(r)at(2.00,1.50){};
                \draw[Edge](r)--(1);
                \node[right of=0,node distance=.35cm]
                    {\footnotesize\begin{math}\dots\end{math}};
                \node[left of=5,node distance=.35cm]
                    {\footnotesize\begin{math}\dots\end{math}};
                \node[right of=2,node distance=.4cm]
                    {\footnotesize\begin{math}\dots\end{math}};
            \end{scope}\end{tikzpicture}};
        \node(C)at(3,-2.75){
            \begin{tikzpicture}
            [baseline=(current bounding box.center),xscale=.4,yscale=.28]
            \begin{scope}
                \node[Leaf](0)at(0.00,-2.00){};
                \node[Leaf](2)at(1.00,-4.00){};
                \node[Leaf](4)at(3.00,-4.00){};
                \node[Leaf](5)at(4.00,-2.00){};
                \node[Node](1)at(2.00,0.00){\begin{math}b c\end{math}};
                \node[Node](3)at(2.00,-2.00){\begin{math}a\end{math}};
                \draw[Edge](0)--(1);
                \draw[Edge](2)--(3);
                \draw[Edge](3)--(1);
                \draw[Edge](4)--(3);
                \draw[Edge](5)--(1);
                \node(r)at(2.00,1.50){};
                \draw[Edge](r)--(1);
                \node[right of=0,node distance=.32cm]
                    {\footnotesize\begin{math}\dots\end{math}};
                \node[left of=5,node distance=.32cm]
                    {\footnotesize\begin{math}\dots\end{math}};
                \node[right of=2,node distance=.4cm]
                    {\footnotesize\begin{math}\dots\end{math}};
            \end{scope}\end{tikzpicture}};
        \node(D)at(0,-5){
            \begin{tikzpicture}
            [baseline=(current bounding box.center),xscale=.32,yscale=.36]
            \begin{scope}
                \node[Leaf](0)at(0.00,-1.50){};
                \node[Leaf](2)at(2.00,-1.50){};
                \node[Node](1)at(1.00,0.00){\begin{math}a b c\end{math}};
                \draw[Edge](0)--(1);
                \draw[Edge](2)--(1);
                \node(r)at(1.00,1.5){};
                \draw[Edge](r)--(1);
                \node[right of=0,node distance=.34cm]
                    {\footnotesize\begin{math}\dots\end{math}};
            \end{scope}\end{tikzpicture}};
        \draw[EdgeRew](A)--(B);
        \draw[EdgeRew](A)--(C);
        \draw[EdgeRew](B)--(D);
        \draw[EdgeRew](C)--(D);
    \end{tikzpicture}}\,.
    \end{equation}
    \end{minipage}
    \end{subequations}

    Therefore, this shows that all critical pairs of $\RewS_\Pca'$ of
    trees consisting in three internal nodes of less are joinable.
    Since $\RewS_\Pca'$ is of degree $2$ and is terminating, by
    Lemma~\ref{lem:confluent_few_nodes}, this implies that $\RewS_\Pca'$
    is confluent.
\end{proof}
\medskip

In Lemma~\ref{lem:rewrite_rule_schroder_trees}, the condition on $\Pca$
to be a forest poset is a necessary condition for the confluence of
$\RewS_\Pca'$. Indeed, by setting
\begin{equation}
    \Pca :=
    \begin{tikzpicture}
        [baseline=(current bounding box.center),xscale=.4,yscale=.4]
        \node[Vertex](1)at(-.75,0){\begin{math}1\end{math}};
        \node[Vertex](2)at(.75,0){\begin{math}2\end{math}};
        \node[Vertex](3)at(0,-1){\begin{math}3\end{math}};
        \draw[Edge](1)--(3);
        \draw[Edge](2)--(3);
    \end{tikzpicture}\,,
\end{equation}
the critical tree
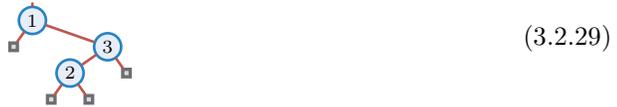
\begin{equation}
    \begin{tikzpicture}
    [baseline=(current bounding box.center),xscale=.25,yscale=.2]
        \node[Leaf](0)at(0.00,-1.75){};
        \node[Leaf](2)at(2.00,-5.25){};
        \node[Leaf](4)at(4.00,-5.25){};
        \node[Leaf](6)at(6.00,-3.50){};
        \node[Node](1)at(1.00,0.00){\begin{math}1\end{math}};
        \node[Node](3)at(3.00,-3.50){\begin{math}2\end{math}};
        \node[Node](5)at(5.00,-1.75){\begin{math}3\end{math}};
        \draw[Edge](0)--(1);
        \draw[Edge](2)--(3);
        \draw[Edge](3)--(5);
        \draw[Edge](4)--(3);
        \draw[Edge](5)--(1);
        \draw[Edge](6)--(5);
        \node(r)at(1.00,1.75){};
        \draw[Edge](r)--(1);
    \end{tikzpicture}
\end{equation}
of $\RewS_\Pca'$ admits the critical pair consisting in the two trees
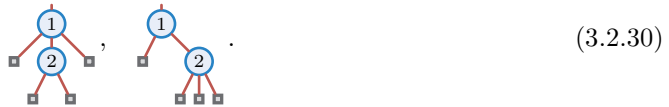
\begin{equation}
    \begin{tikzpicture}
    [baseline=(current bounding box.center),xscale=.25,yscale=.25]
        \node[Leaf](0)at(0.00,-2.00){};
        \node[Leaf](2)at(1.00,-4.00){};
        \node[Leaf](4)at(3.00,-4.00){};
        \node[Leaf](5)at(4.00,-2.00){};
        \node[Node](1)at(2.00,0.00){\begin{math}1\end{math}};
        \node[Node](3)at(2.00,-2.00){\begin{math}2\end{math}};
        \draw[Edge](0)--(1);
        \draw[Edge](2)--(3);
        \draw[Edge](3)--(1);
        \draw[Edge](4)--(3);
        \draw[Edge](5)--(1);
        \node(r)at(2.00,1.50){};
        \draw[Edge](r)--(1);
    \end{tikzpicture}\,,\quad
    \begin{tikzpicture}
    [baseline=(current bounding box.center),xscale=.25,yscale=.25]
        \node[Leaf](0)at(0.00,-2.00){};
        \node[Leaf](2)at(2.00,-4.00){};
        \node[Leaf](4)at(3.00,-4.00){};
        \node[Leaf](5)at(4.00,-4.00){};
        \node[Node](1)at(1.00,0.00){\begin{math}1\end{math}};
        \node[Node](3)at(3.00,-2.00){\begin{math}2\end{math}};
        \draw[Edge](0)--(1);
        \draw[Edge](2)--(3);
        \draw[Edge](3)--(1);
        \draw[Edge](4)--(3);
        \draw[Edge](5)--(3);
        \node(r)at(1.00,1.50){};
        \draw[Edge](r)--(1);
    \end{tikzpicture}\,.
\end{equation}
Since these two trees are normal forms of $\RewS_\Pca'$, this critical
pair is not joinable and hence, $\RewS_\Pca'$ is not confluent.
\medskip

We define the partial composition $\Sfr \circ_i \Tfr$ of two
$\Pca$-alternating Schröder trees $\Sfr$ and $\Tfr$ as the
$\Pca$-alternating Schröder tree being the normal form by $\RewS_\Pca'$
of the $\Pca$-Schröder tree obtained by grafting the root $\Tfr$ on the
$i$th leaf of $\Sfr$. We denote by $\ASchr(\Pca)$ the linear span of the
set of the $\Pca$-alternating Schröder trees endowed with the partial
composition described above and extended by linearity. Consider for
instance the forest poset
\begin{equation} \label{equ:example_forest_poset}
    \Pca :=
    \begin{tikzpicture}
        [baseline=(current bounding box.center),xscale=.45,yscale=.42]
        \node[Vertex](1)at(0,0){\begin{math}1\end{math}};
        \node[Vertex](2)at(-.75,-1){\begin{math}2\end{math}};
        \node[Vertex](3)at(.75,-1){\begin{math}3\end{math}};
        \node[Vertex](4)at(2,0){\begin{math}4\end{math}};
        \node[Vertex](5)at(2,-1){\begin{math}5\end{math}};
        \node[Vertex](6)at(2,-2){\begin{math}6\end{math}};
        \draw[Edge](1)--(2);
        \draw[Edge](1)--(3);
        \draw[Edge](4)--(5);
        \draw[Edge](5)--(6);
    \end{tikzpicture}\,.
\end{equation}
Then, we have in $\ASchr(\Pca)$ the partial composition
\begin{equation}
    \begin{tikzpicture}
    [baseline=(current bounding box.center),xscale=.25,yscale=.22]
        \node[Leaf](0)at(0.00,-3.33){};
        \node[Leaf](2)at(2.00,-3.33){};
        \node[Leaf](4)at(4.00,-1.67){};
        \node[Node](1)at(1.00,-1.67){\begin{math}4\end{math}};
        \node[Node](3)at(3.00,0.00){\begin{math}6\end{math}};
        \draw[Edge](0)--(1);
        \draw[Edge](1)--(3);
        \draw[Edge](2)--(1);
        \draw[Edge](4)--(3);
        \node(r)at(3.00,1.75){};
        \draw[Edge](r)--(3);
    \end{tikzpicture}
    \enspace \circ_3 \enspace
    \begin{tikzpicture}
        [baseline=(current bounding box.center),xscale=.25,yscale=.22]
        \node[Leaf](0)at(0.00,-1.67){};
        \node[Leaf](2)at(2.00,-3.33){};
        \node[Leaf](4)at(4.00,-3.33){};
        \node[Node,Mark2](1)at(1.00,0.00){\begin{math}2\end{math}};
        \node[Node,Mark2](3)at(3.00,-1.67){\begin{math}6\end{math}};
        \draw[Edge](0)--(1);
        \draw[Edge](2)--(3);
        \draw[Edge](3)--(1);
        \draw[Edge](4)--(3);
        \node(r)at(1.00,1.75){};
        \draw[Edge](r)--(1);
    \end{tikzpicture}
    \enspace = \enspace
    \begin{tikzpicture}
    [baseline=(current bounding box.center),xscale=.25,yscale=.18]
        \node[Leaf](0)at(0.00,-4.50){};
        \node[Leaf](2)at(2.00,-4.50){};
        \node[Leaf](4)at(4.00,-4.50){};
        \node[Leaf](6)at(6.00,-6.75){};
        \node[Leaf](8)at(8.00,-6.75){};
        \node[Node](1)at(1.00,-2.25){\begin{math}4\end{math}};
        \node[Node](3)at(3.00,0.00){\begin{math}6\end{math}};
        \node[Node,Mark2](5)at(5.00,-2.25){\begin{math}2\end{math}};
        \node[Node,Mark2](7)at(7.00,-4.50){\begin{math}6\end{math}};
        \draw[Edge](0)--(1);
        \draw[Edge](1)--(3);
        \draw[Edge](2)--(1);
        \draw[Edge](4)--(5);
        \draw[Edge](5)--(3);
        \draw[Edge](6)--(7);
        \draw[Edge](7)--(5);
        \draw[Edge](8)--(7);
        \node(r)at(3.00,2){};
        \draw[Edge](r)--(3);
    \end{tikzpicture}\,,
\end{equation}
and also
\begin{equation} \label{equ:example_schroder_partial_composition}
    \begin{tikzpicture}
        [baseline=(current bounding box.center),xscale=.25,yscale=.22]
        \node[Leaf](0)at(0.00,-1.67){};
        \node[Leaf](2)at(2.00,-3.33){};
        \node[Leaf](4)at(4.00,-3.33){};
        \node[Node](1)at(1.00,0.00){\begin{math}1\end{math}};
        \node[Node](3)at(3.00,-1.67){\begin{math}4\end{math}};
        \draw[Edge](0)--(1);
        \draw[Edge](2)--(3);
        \draw[Edge](3)--(1);
        \draw[Edge](4)--(3);
        \node(r)at(1.00,1.75){};
        \draw[Edge](r)--(1);
    \end{tikzpicture}
    \enspace \circ_1 \enspace
    \begin{tikzpicture}
        [baseline=(current bounding box.center),xscale=.25,yscale=.18]
        \node[Leaf](0)at(0.00,-4.67){};
        \node[Leaf](2)at(2.00,-4.67){};
        \node[Leaf](4)at(4.00,-4.67){};
        \node[Leaf](6)at(6.00,-4.67){};
        \node[Node,Mark2](1)at(1.00,-2.33){\begin{math}3\end{math}};
        \node[Node,Mark2](3)at(3.00,0.00){\begin{math}2\end{math}};
        \node[Node,Mark2](5)at(5.00,-2.33){\begin{math}3\end{math}};
        \draw[Edge](0)--(1);
        \draw[Edge](1)--(3);
        \draw[Edge](2)--(1);
        \draw[Edge](4)--(5);
        \draw[Edge](5)--(3);
        \draw[Edge](6)--(5);
        \node(r)at(3.00,2){};
        \draw[Edge](r)--(3);
    \end{tikzpicture}
    \enspace = \enspace
    \begin{tikzpicture}
    [baseline=(current bounding box.center),xscale=.25,yscale=.18]
        \node[Leaf](0)at(0.00,-2.67){};
        \node[Leaf](1)at(1.00,-2.67){};
        \node[Leaf](3)at(2.00,-2.67){};
        \node[Leaf](4)at(3.00,-2.67){};
        \node[Leaf](5)at(4.00,-5.33){};
        \node[Leaf](7)at(6.00,-5.33){};
        \node[Node](2)at(2.00,0.00){\begin{math}1\end{math}};
        \node[Node](6)at(5.00,-2.67){\begin{math}4\end{math}};
        \draw[Edge](0)--(2);
        \draw[Edge](1)--(2);
        \draw[Edge](3)--(2);
        \draw[Edge](4)--(2);
        \draw[Edge](5)--(6);
        \draw[Edge](6)--(2);
        \draw[Edge](7)--(6);
        \node(r)at(2.00,2.00){};
        \draw[Edge](r)--(2);
    \end{tikzpicture}\,,
\end{equation}
since
\begin{equation} \label{equ:example_schroder_rewriting_steps}
    \begin{tikzpicture}
    [baseline=(current bounding box.center),xscale=.25,yscale=.15]
        \node[Leaf](0)at(0.00,-8.25){};
        \node[Leaf](10)at(10.00,-5.50){};
        \node[Leaf](2)at(2.00,-8.25){};
        \node[Leaf](4)at(4.00,-8.25){};
        \node[Leaf](6)at(6.00,-8.25){};
        \node[Leaf](8)at(8.00,-5.50){};
        \node[Node](1)at(1.00,-5.50){\begin{math}3\end{math}};
        \node[Node](3)at(3.00,-2.75){\begin{math}2\end{math}};
        \node[Node](5)at(5.00,-5.50){\begin{math}3\end{math}};
        \node[Node](7)at(7.00,0.00){\begin{math}1\end{math}};
        \node[Node](9)at(9.00,-2.75){\begin{math}4\end{math}};
        \draw[Edge](0)--(1);
        \draw[Edge](1)--(3);
        \draw[Edge](10)--(9);
        \draw[Edge](2)--(1);
        \draw[Edge](3)--(7);
        \draw[Edge](4)--(5);
        \draw[Edge](5)--(3);
        \draw[Edge](6)--(5);
        \draw[Edge](8)--(9);
        \draw[Edge](9)--(7);
        \node(r)at(7.00,2.25){};
        \draw[Edge](r)--(7);
    \end{tikzpicture}
    \enspace \RewS_\Pca \enspace
    \begin{tikzpicture}
    [baseline=(current bounding box.center),xscale=.25,yscale=.14]
        \node[Leaf](0)at(0.00,-6.67){};
        \node[Leaf](2)at(2.00,-6.67){};
        \node[Leaf](4)at(3.00,-6.67){};
        \node[Leaf](6)at(5.00,-6.67){};
        \node[Leaf](7)at(6.00,-6.67){};
        \node[Leaf](9)at(8.00,-6.67){};
        \node[Node](1)at(1.00,-3.33){\begin{math}3\end{math}};
        \node[Node](3)at(4.00,0.00){\begin{math}1\end{math}};
        \node[Node](5)at(4.00,-3.33){\begin{math}3\end{math}};
        \node[Node](8)at(7.00,-3.33){\begin{math}4\end{math}};
        \draw[Edge](0)--(1);
        \draw[Edge](1)--(3);
        \draw[Edge](2)--(1);
        \draw[Edge](4)--(5);
        \draw[Edge](5)--(3);
        \draw[Edge](6)--(5);
        \draw[Edge](7)--(8);
        \draw[Edge](8)--(3);
        \draw[Edge](9)--(8);
        \node(r)at(4.00,2.50){};
        \draw[Edge](r)--(3);
    \end{tikzpicture}
    \enspace \RewS_\Pca \enspace
    \begin{tikzpicture}
    [baseline=(current bounding box.center),xscale=.25,yscale=.14]
        \node[Leaf](0)at(0.00,-6.00){};
        \node[Leaf](2)at(2.00,-6.00){};
        \node[Leaf](3)at(3.00,-3.00){};
        \node[Leaf](5)at(5.00,-3.00){};
        \node[Leaf](6)at(6.00,-6.00){};
        \node[Leaf](8)at(8.00,-6.00){};
        \node[Node](1)at(1.00,-3.00){\begin{math}3\end{math}};
        \node[Node](4)at(4.00,0.00){\begin{math}1\end{math}};
        \node[Node](7)at(7.00,-3.00){\begin{math}4\end{math}};
        \draw[Edge](0)--(1);
        \draw[Edge](1)--(4);
        \draw[Edge](2)--(1);
        \draw[Edge](3)--(4);
        \draw[Edge](5)--(4);
        \draw[Edge](6)--(7);
        \draw[Edge](7)--(4);
        \draw[Edge](8)--(7);
        \node(r)at(4.00,2.5){};
        \draw[Edge](r)--(4);
    \end{tikzpicture}
    \enspace \RewS_\Pca \enspace
    \begin{tikzpicture}
    [baseline=(current bounding box.center),xscale=.25,yscale=.16]
        \node[Leaf](0)at(0.00,-2.67){};
        \node[Leaf](1)at(1.00,-2.67){};
        \node[Leaf](3)at(2.00,-2.67){};
        \node[Leaf](4)at(3.00,-2.67){};
        \node[Leaf](5)at(4.00,-5.33){};
        \node[Leaf](7)at(6.00,-5.33){};
        \node[Node](2)at(2.00,0.00){\begin{math}1\end{math}};
        \node[Node](6)at(5.00,-2.67){\begin{math}4\end{math}};
        \draw[Edge](0)--(2);
        \draw[Edge](1)--(2);
        \draw[Edge](3)--(2);
        \draw[Edge](4)--(2);
        \draw[Edge](5)--(6);
        \draw[Edge](6)--(2);
        \draw[Edge](7)--(6);
        \node(r)at(2.00,2.25){};
        \draw[Edge](r)--(2);
    \end{tikzpicture}
\end{equation}
is a sequence of rewritings steps by $\Rew_\Pca'$, where the leftmost
tree of~\eqref{equ:example_schroder_rewriting_steps} is obtained by
grafting the root of the second tree
of~\eqref{equ:example_schroder_partial_composition} onto the first leaf
of the first tree of~\eqref{equ:example_schroder_partial_composition}.
\medskip

\begin{Proposition} \label{prop:operad_aschr}
    Let $\Pca$ be a forest poset. Then, $\ASchr(\Pca)$ is an operad
    graded by the number of the leaves of the trees. Moreover, as an
    operad, $\ASchr(\Pca)$ is generated by the set of $\Pca$-corollas of
    arity two.
\end{Proposition}
\begin{proof}
    The fact that $\ASchr(\Pca)$ is an operad, that is it satisfies
    Relations~\eqref{equ:operad_axiom_1}, \eqref{equ:operad_axiom_2},
    and~\eqref{equ:operad_axiom_3}, is a direct consequence of the fact
    that, by Lemma~\ref{lem:rewrite_rule_schroder_trees}, the rewrite
    rule $\RewS_\Pca'$ intervening in the computation of the partial
    compositions of two $\Pca$-alternating Schröder trees is convergent.
    This operad is graded by the number of the leaves of the trees by
    definition of its partial composition.
    \smallskip

    Finally, a straightforward structural induction on $\Pca$-alternating
    Schröder trees, relying on their recursive general form provided
    by~\eqref{equ:description_aschr}, shows that any $\Pca$-alternating
    Schröder tree can be expressed by partial compositions involving
    only $\Pca$-corollas of arity two. Whence the second part of the
    statement of the proposition.
\end{proof}
\medskip

\begin{Theorem} \label{thm:realization}
    Let $\Pca$ be a forest poset. Then, the operads $\As(\Pca)$ and
    $\ASchr(\Pca)$ are isomorphic.
\end{Theorem}
\begin{proof}
    First, by Proposition~\ref{prop:operad_aschr}, $\ASchr(\Pca)$ is an
    operad wherein for any $n \geq 1$, its graded component of arity $n$
    has bases indexed by $\Pca$-alternating Schröder trees with $n$
    leaves. By Lemma~\ref{lem:bijection_pbw_basis_aschr}, these trees
    are in bijection with the elements of the Poincaré-Birkhoff-Witt
    basis $\NormalF(\Pca)$ of $\As(\Pca)$ provided by
    Theorem~\ref{thm:koszulity}. By~\cite{Hof10}, this shows that
    $\ASchr(\Pca)$ and $\As(\Pca)$ are isomorphic as graded vector spaces.
    \smallskip

    The generators of $\ASchr(\Pca)$, that are by
    Proposition~\ref{prop:operad_aschr} $\Pca$-corollas of arity two,
    satisfy at least the nontrivial relations
    \begin{subequations}
    \begin{equation} \label{equ:realization_relation_1}
        \Corolla_a^2 \circ_1 \Corolla_b^2
        -
        \Corolla_{a \Min_\Pca b}^2 \circ_2 \Corolla_{a \Min_\Pca b}^2 = 0,
        \qquad a, b \in \Pca \mbox{ and }
        (a \Ord_\Pca b \mbox{ or } b \Ord_\Pca a),
    \end{equation}
    \begin{equation} \label{equ:realization_relation_2}
        \Corolla_{a \Min_\Pca b}^2 \circ_1 \Corolla_{a \Min_\Pca b}^2
        -
        \Corolla_a^2 \circ_2 \Corolla_b^2 = 0,
        \qquad a, b \in \Pca \mbox{ and }
        (a \Ord_\Pca b \mbox{ or } b \Ord_\Pca a),
    \end{equation}
    \end{subequations}
    obtained by a direct computation in $\ASchr(\Pca)$. By using the
    same reasoning as the one used to establish
    Proposition~\ref{prop:dimension_relations}, we obtain that there are
    as many elements of the form~\eqref{equ:realization_relation_1}
    or~\eqref{equ:realization_relation_2} as generating relations for
    the space of relations  $\FreeRel_\Pca^\Op$ of $\As(\Pca)$
    (see~\eqref{equ:relation_1} and~\eqref{equ:relation_2}). Therefore,
    as $\ASchr(\Pca)$ and $\As(\Pca)$ are isomorphic as graded vector
    spaces, it cannot be more nontrivial relations in $\ASchr(\Pca)$
    than Relations~\eqref{equ:realization_relation_1}
    and~\eqref{equ:realization_relation_2}.
    \smallskip

    Finally, by identifying all symbols $\Corolla_a^2$, $a \in \Pca$,
    with $\Op_a$, we observe that $\As(\Pca)$ and $\ASchr(\Pca)$ admit
    the same presentation. This implies that $\As(\Pca)$ and
    $\ASchr(\Pca)$ are isomorphic operads.
\end{proof}
\medskip

As announced, Theorem~\ref{thm:realization} provides a combinatorial
realization $\ASchr(\Pca)$ of $\As(\Pca)$ when $\Pca$ is a forest poset.
\medskip

\subsubsection{Free forest poset associative algebras over one generator}%
\label{subsubsec:free_as_forest_poset_algebras}
The realization of $\As(\Pca)$, when $\Pca$ is a forest poset, provided
by Theorem~\ref{thm:realization} in terms of $\Pca$-alternating Schröder
trees leads to the following description. The free $\Pca$-associative
algebra over one generator, where $\Pca$ is a forest poset, has
$\ASchr(\Pca)$ as underlying vector space and is endowed with linear
operations
\begin{equation}
    \Op_a : \ASchr(\Pca) \otimes \ASchr(\Pca) \to \ASchr(\Pca),
    \qquad a \in \Pca,
\end{equation}
satisfying for all $\Pca$-alternating Schröder trees $\Sfr$ and $\Tfr$,
\begin{equation}
    \Sfr \Op_a \Tfr
    = \left(\Corolla_a^2 \circ_2 \Tfr\right) \circ_1 \Sfr.
\end{equation}
In an alternative way, $\Sfr \Op_a \Tfr$ is the $\Pca$-alternating
Schröder obtained by considering the normal form by $\RewS_\Pca'$ of the
tree obtained by grafting $\Sfr$ and $\Tfr$ respectively as left and
right child of a binary corolla labeled by $a$.
\medskip

Let us provide examples of computations in the free $\Pca$-associative
algebra over one generator where $\Pca$ is the forest poset
\begin{equation}
    \Pca :=
    \begin{tikzpicture}
        [baseline=(current bounding box.center),xscale=.45,yscale=.42]
        \node[Vertex](1)at(0,0){\begin{math}1\end{math}};
        \node[Vertex](2)at(-.75,-1){\begin{math}2\end{math}};
        \node[Vertex](3)at(.75,-1){\begin{math}3\end{math}};
        \node[Vertex](4)at(.75,-2){\begin{math}4\end{math}};
        \node[Vertex](5)at(2,0){\begin{math}5\end{math}};
        \draw[Edge](1)--(2);
        \draw[Edge](1)--(3);
        \draw[Edge](3)--(4);
    \end{tikzpicture}\,.
\end{equation}
We have
\begin{subequations}
\begin{equation}
    \begin{tikzpicture}
        [baseline=(current bounding box.center),xscale=.25,yscale=.18]
        \node[Leaf](0)at(0.00,-2.00){};
        \node[Leaf](2)at(1.00,-2.00){};
        \node[Leaf](3)at(2.00,-4.00){};
        \node[Leaf](5)at(4.00,-4.00){};
        \node[Node](1)at(1.00,0.00){\begin{math}2\end{math}};
        \node[Node](4)at(3.00,-2.00){\begin{math}4\end{math}};
        \draw[Edge](0)--(1);
        \draw[Edge](2)--(1);
        \draw[Edge](3)--(4);
        \draw[Edge](4)--(1);
        \draw[Edge](5)--(4);
        \node(r)at(1.00,2.25){};
        \draw[Edge](r)--(1);
    \end{tikzpicture}
    \enspace \Op_1 \enspace
    \begin{tikzpicture}
        [baseline=(current bounding box.center),xscale=.25,yscale=.17]
        \node[Leaf](0)at(0.00,-4.67){};
        \node[Leaf](2)at(2.00,-4.67){};
        \node[Leaf](4)at(4.00,-4.67){};
        \node[Leaf](6)at(6.00,-4.67){};
        \node[Node,Mark2](1)at(1.00,-2.33){\begin{math}2\end{math}};
        \node[Node,Mark2](3)at(3.00,0.00){\begin{math}3\end{math}};
        \node[Node,Mark2](5)at(5.00,-2.33){\begin{math}5\end{math}};
        \draw[Edge](0)--(1);
        \draw[Edge](1)--(3);
        \draw[Edge](2)--(1);
        \draw[Edge](4)--(5);
        \draw[Edge](5)--(3);
        \draw[Edge](6)--(5);
        \node(r)at(3.00,2.25){};
        \draw[Edge](r)--(3);
    \end{tikzpicture}
    \enspace = \enspace
    \begin{tikzpicture}
        [baseline=(current bounding box.center),xscale=.18,yscale=.14]
        \node[Leaf](0)at(0.00,-3.33){};
        \node[Leaf](1)at(1.00,-3.33){};
        \node[Leaf](2)at(2.00,-3.33){};
        \node[Leaf](4)at(3.00,-3.33){};
        \node[Leaf](5)at(4.00,-3.33){};
        \node[Leaf](6)at(5.00,-3.33){};
        \node[Leaf](7)at(6.00,-6.67){};
        \node[Leaf](9)at(8.00,-6.67){};
        \node[Node,Mark1](3)at(3.00,0.00){\begin{math}1\end{math}};
        \node[Node,Mark2](8)at(7.00,-3.33){\begin{math}5\end{math}};
        \draw[Edge](0)--(3);
        \draw[Edge](1)--(3);
        \draw[Edge](2)--(3);
        \draw[Edge](4)--(3);
        \draw[Edge](5)--(3);
        \draw[Edge](6)--(3);
        \draw[Edge](7)--(8);
        \draw[Edge](8)--(3);
        \draw[Edge](9)--(8);
        \node(r)at(3.00,2.50){};
        \draw[Edge](r)--(3);
    \end{tikzpicture}\,,
\end{equation}
\begin{equation}
    \begin{tikzpicture}
        [baseline=(current bounding box.center),xscale=.25,yscale=.18]
        \node[Leaf](0)at(0.00,-2.00){};
        \node[Leaf](2)at(1.00,-2.00){};
        \node[Leaf](3)at(2.00,-4.00){};
        \node[Leaf](5)at(4.00,-4.00){};
        \node[Node](1)at(1.00,0.00){\begin{math}2\end{math}};
        \node[Node](4)at(3.00,-2.00){\begin{math}4\end{math}};
        \draw[Edge](0)--(1);
        \draw[Edge](2)--(1);
        \draw[Edge](3)--(4);
        \draw[Edge](4)--(1);
        \draw[Edge](5)--(4);
        \node(r)at(1.00,2.25){};
        \draw[Edge](r)--(1);
    \end{tikzpicture}
    \enspace \Op_2 \enspace
    \begin{tikzpicture}
        [baseline=(current bounding box.center),xscale=.25,yscale=.17]
        \node[Leaf](0)at(0.00,-4.67){};
        \node[Leaf](2)at(2.00,-4.67){};
        \node[Leaf](4)at(4.00,-4.67){};
        \node[Leaf](6)at(6.00,-4.67){};
        \node[Node,Mark2](1)at(1.00,-2.33){\begin{math}2\end{math}};
        \node[Node,Mark2](3)at(3.00,0.00){\begin{math}3\end{math}};
        \node[Node,Mark2](5)at(5.00,-2.33){\begin{math}5\end{math}};
        \draw[Edge](0)--(1);
        \draw[Edge](1)--(3);
        \draw[Edge](2)--(1);
        \draw[Edge](4)--(5);
        \draw[Edge](5)--(3);
        \draw[Edge](6)--(5);
        \node(r)at(3.00,2.25){};
        \draw[Edge](r)--(3);
    \end{tikzpicture}
    \enspace = \enspace
    \begin{tikzpicture}
        [baseline=(current bounding box.center),xscale=.22,yscale=.13]
        \node[Leaf](0)at(0.00,-3.25){};
        \node[Leaf](1)at(1.00,-3.25){};
        \node[Leaf](10)at(10.00,-9.75){};
        \node[Leaf](12)at(12.00,-9.75){};
        \node[Leaf](3)at(3.00,-6.50){};
        \node[Leaf](5)at(5.00,-6.50){};
        \node[Leaf](6)at(6.00,-9.75){};
        \node[Leaf](8)at(8.00,-9.75){};
        \node[Node,Mark2](11)at(11.00,-6.50){\begin{math}5\end{math}};
        \node[Node,Mark1](2)at(2.00,0.00){\begin{math}2\end{math}};
        \node[Node](4)at(4.00,-3.25){\begin{math}4\end{math}};
        \node[Node,Mark2](7)at(7.00,-6.50){\begin{math}2\end{math}};
        \node[Node,Mark2](9)at(9.00,-3.25){\begin{math}3\end{math}};
        \draw[Edge](0)--(2);
        \draw[Edge](1)--(2);
        \draw[Edge](10)--(11);
        \draw[Edge](11)--(9);
        \draw[Edge](12)--(11);
        \draw[Edge](3)--(4);
        \draw[Edge](4)--(2);
        \draw[Edge](5)--(4);
        \draw[Edge](6)--(7);
        \draw[Edge](7)--(9);
        \draw[Edge](8)--(7);
        \draw[Edge](9)--(2);
        \node(r)at(2.00,2.75){};
        \draw[Edge](r)--(2);
    \end{tikzpicture}\,,
\end{equation}
\begin{equation}
    \begin{tikzpicture}
        [baseline=(current bounding box.center),xscale=.25,yscale=.18]
        \node[Leaf](0)at(0.00,-2.00){};
        \node[Leaf](2)at(1.00,-2.00){};
        \node[Leaf](3)at(2.00,-4.00){};
        \node[Leaf](5)at(4.00,-4.00){};
        \node[Node](1)at(1.00,0.00){\begin{math}2\end{math}};
        \node[Node](4)at(3.00,-2.00){\begin{math}4\end{math}};
        \draw[Edge](0)--(1);
        \draw[Edge](2)--(1);
        \draw[Edge](3)--(4);
        \draw[Edge](4)--(1);
        \draw[Edge](5)--(4);
        \node(r)at(1.00,2.25){};
        \draw[Edge](r)--(1);
    \end{tikzpicture}
    \enspace \Op_3 \enspace
    \begin{tikzpicture}
        [baseline=(current bounding box.center),xscale=.25,yscale=.17]
        \node[Leaf](0)at(0.00,-4.67){};
        \node[Leaf](2)at(2.00,-4.67){};
        \node[Leaf](4)at(4.00,-4.67){};
        \node[Leaf](6)at(6.00,-4.67){};
        \node[Node,Mark2](1)at(1.00,-2.33){\begin{math}2\end{math}};
        \node[Node,Mark2](3)at(3.00,0.00){\begin{math}3\end{math}};
        \node[Node,Mark2](5)at(5.00,-2.33){\begin{math}5\end{math}};
        \draw[Edge](0)--(1);
        \draw[Edge](1)--(3);
        \draw[Edge](2)--(1);
        \draw[Edge](4)--(5);
        \draw[Edge](5)--(3);
        \draw[Edge](6)--(5);
        \node(r)at(3.00,2.25){};
        \draw[Edge](r)--(3);
    \end{tikzpicture}
    \enspace = \enspace
    \begin{tikzpicture}
        [baseline=(current bounding box.center),xscale=.25,yscale=.14]
        \node[Leaf](0)at(0.00,-6.50){};
        \node[Leaf](10)at(8.00,-6.50){};
        \node[Leaf](12)at(10.00,-6.50){};
        \node[Leaf](2)at(1.00,-6.50){};
        \node[Leaf](3)at(2.00,-9.75){};
        \node[Leaf](5)at(4.00,-9.75){};
        \node[Leaf](7)at(5.00,-6.50){};
        \node[Leaf](9)at(7.00,-6.50){};
        \node[Node](1)at(1.00,-3.25){\begin{math}2\end{math}};
        \node[Node,Mark2](11)at(9.00,-3.25){\begin{math}5\end{math}};
        \node[Node](4)at(3.00,-6.50){\begin{math}4\end{math}};
        \node[Node,Mark1](6)at(6.00,0.00){\begin{math}3\end{math}};
        \node[Node,Mark2](8)at(6.00,-3.25){\begin{math}2\end{math}};
        \draw[Edge](0)--(1);
        \draw[Edge](1)--(6);
        \draw[Edge](10)--(11);
        \draw[Edge](11)--(6);
        \draw[Edge](12)--(11);
        \draw[Edge](2)--(1);
        \draw[Edge](3)--(4);
        \draw[Edge](4)--(1);
        \draw[Edge](5)--(4);
        \draw[Edge](7)--(8);
        \draw[Edge](8)--(6);
        \draw[Edge](9)--(8);
        \node(r)at(6.00,2.44){};
        \draw[Edge](r)--(6);
    \end{tikzpicture}\,,
\end{equation}
\begin{equation}
    \begin{tikzpicture}
        [baseline=(current bounding box.center),xscale=.25,yscale=.18]
        \node[Leaf](0)at(0.00,-2.00){};
        \node[Leaf](2)at(1.00,-2.00){};
        \node[Leaf](3)at(2.00,-4.00){};
        \node[Leaf](5)at(4.00,-4.00){};
        \node[Node](1)at(1.00,0.00){\begin{math}2\end{math}};
        \node[Node](4)at(3.00,-2.00){\begin{math}4\end{math}};
        \draw[Edge](0)--(1);
        \draw[Edge](2)--(1);
        \draw[Edge](3)--(4);
        \draw[Edge](4)--(1);
        \draw[Edge](5)--(4);
        \node(r)at(1.00,2.25){};
        \draw[Edge](r)--(1);
    \end{tikzpicture}
    \enspace \Op_4 \enspace
    \begin{tikzpicture}
        [baseline=(current bounding box.center),xscale=.25,yscale=.17]
        \node[Leaf](0)at(0.00,-4.67){};
        \node[Leaf](2)at(2.00,-4.67){};
        \node[Leaf](4)at(4.00,-4.67){};
        \node[Leaf](6)at(6.00,-4.67){};
        \node[Node,Mark2](1)at(1.00,-2.33){\begin{math}2\end{math}};
        \node[Node,Mark2](3)at(3.00,0.00){\begin{math}3\end{math}};
        \node[Node,Mark2](5)at(5.00,-2.33){\begin{math}5\end{math}};
        \draw[Edge](0)--(1);
        \draw[Edge](1)--(3);
        \draw[Edge](2)--(1);
        \draw[Edge](4)--(5);
        \draw[Edge](5)--(3);
        \draw[Edge](6)--(5);
        \node(r)at(3.00,2.25){};
        \draw[Edge](r)--(3);
    \end{tikzpicture}
    \enspace = \enspace
    \begin{tikzpicture}
        [baseline=(current bounding box.center),xscale=.25,yscale=.14]
        \node[Leaf](0)at(0.00,-6.50){};
        \node[Leaf](10)at(8.00,-6.50){};
        \node[Leaf](12)at(10.00,-6.50){};
        \node[Leaf](2)at(1.00,-6.50){};
        \node[Leaf](3)at(2.00,-9.75){};
        \node[Leaf](5)at(4.00,-9.75){};
        \node[Leaf](7)at(5.00,-6.50){};
        \node[Leaf](9)at(7.00,-6.50){};
        \node[Node](1)at(1.00,-3.25){\begin{math}2\end{math}};
        \node[Node,Mark2](11)at(9.00,-3.25){\begin{math}5\end{math}};
        \node[Node](4)at(3.00,-6.50){\begin{math}4\end{math}};
        \node[Node,Mark1](6)at(6.00,0.00){\begin{math}3\end{math}};
        \node[Node,Mark2](8)at(6.00,-3.25){\begin{math}2\end{math}};
        \draw[Edge](0)--(1);
        \draw[Edge](1)--(6);
        \draw[Edge](10)--(11);
        \draw[Edge](11)--(6);
        \draw[Edge](12)--(11);
        \draw[Edge](2)--(1);
        \draw[Edge](3)--(4);
        \draw[Edge](4)--(1);
        \draw[Edge](5)--(4);
        \draw[Edge](7)--(8);
        \draw[Edge](8)--(6);
        \draw[Edge](9)--(8);
        \node(r)at(6.00,2.44){};
        \draw[Edge](r)--(6);
    \end{tikzpicture}\,,
\end{equation}
\begin{equation}
    \begin{tikzpicture}
        [baseline=(current bounding box.center),xscale=.25,yscale=.18]
        \node[Leaf](0)at(0.00,-2.00){};
        \node[Leaf](2)at(1.00,-2.00){};
        \node[Leaf](3)at(2.00,-4.00){};
        \node[Leaf](5)at(4.00,-4.00){};
        \node[Node](1)at(1.00,0.00){\begin{math}2\end{math}};
        \node[Node](4)at(3.00,-2.00){\begin{math}4\end{math}};
        \draw[Edge](0)--(1);
        \draw[Edge](2)--(1);
        \draw[Edge](3)--(4);
        \draw[Edge](4)--(1);
        \draw[Edge](5)--(4);
        \node(r)at(1.00,2.25){};
        \draw[Edge](r)--(1);
    \end{tikzpicture}
    \enspace \Op_5 \enspace
    \begin{tikzpicture}
        [baseline=(current bounding box.center),xscale=.25,yscale=.17]
        \node[Leaf](0)at(0.00,-4.67){};
        \node[Leaf](2)at(2.00,-4.67){};
        \node[Leaf](4)at(4.00,-4.67){};
        \node[Leaf](6)at(6.00,-4.67){};
        \node[Node,Mark2](1)at(1.00,-2.33){\begin{math}2\end{math}};
        \node[Node,Mark2](3)at(3.00,0.00){\begin{math}3\end{math}};
        \node[Node,Mark2](5)at(5.00,-2.33){\begin{math}5\end{math}};
        \draw[Edge](0)--(1);
        \draw[Edge](1)--(3);
        \draw[Edge](2)--(1);
        \draw[Edge](4)--(5);
        \draw[Edge](5)--(3);
        \draw[Edge](6)--(5);
        \node(r)at(3.00,2.25){};
        \draw[Edge](r)--(3);
    \end{tikzpicture}
    \enspace = \enspace
    \begin{tikzpicture}
        [baseline=(current bounding box.center),xscale=.25,yscale=.11]
        \node[Leaf](0)at(0.00,-7.00){};
        \node[Leaf](11)at(10.00,-10.50){};
        \node[Leaf](13)at(12.00,-10.50){};
        \node[Leaf](2)at(1.00,-7.00){};
        \node[Leaf](3)at(2.00,-10.50){};
        \node[Leaf](5)at(4.00,-10.50){};
        \node[Leaf](7)at(6.00,-10.50){};
        \node[Leaf](9)at(8.00,-10.50){};
        \node[Node](1)at(1.00,-3.50){\begin{math}2\end{math}};
        \node[Node,Mark2](10)at(9.00,-3.50){\begin{math}3\end{math}};
        \node[Node,Mark2](12)at(11.00,-7.00){\begin{math}5\end{math}};
        \node[Node](4)at(3.00,-7.00){\begin{math}4\end{math}};
        \node[Node,Mark1](6)at(5.00,0.00){\begin{math}5\end{math}};
        \node[Node,Mark2](8)at(7.00,-7.00){\begin{math}2\end{math}};
        \draw[Edge](0)--(1);
        \draw[Edge](1)--(6);
        \draw[Edge](10)--(6);
        \draw[Edge](11)--(12);
        \draw[Edge](12)--(10);
        \draw[Edge](13)--(12);
        \draw[Edge](2)--(1);
        \draw[Edge](3)--(4);
        \draw[Edge](4)--(1);
        \draw[Edge](5)--(4);
        \draw[Edge](7)--(8);
        \draw[Edge](8)--(10);
        \draw[Edge](9)--(8);
        \node(r)at(5.00,3.25){};
        \draw[Edge](r)--(6);
    \end{tikzpicture}\,.
\end{equation}
\end{subequations}
\medskip

\subsection{Koszul dual}
We now establish a first presentation for the Koszul dual $\As(\Pca)^!$
of $\As(\Pca)$ where $\Pca$ is a poset (and not necessarily a forest
poset) and provide moreover a second presentation of $\As(\Pca)^!$ when
$\Pca$ is a forest poset. This second presentation of $\As(\Pca)^!$ is
simpler than the first one and it shall be considered in the next section.
\medskip

\subsubsection{Presentation by generators and relations}

\begin{Proposition} \label{prop:presentation_koszul_dual}
    Let $\Pca$ be a poset. Then, the Koszul dual $\As(\Pca)^!$ of
    $\As(\Pca)$ admits the following presentation. It is generated by
    \begin{equation}
        \FreeGen_\Pca^{\OpDual} := \FreeGen_\Pca^{\OpDual}(2)
        := \{\OpDual_a : a \in \Pca\},
    \end{equation}
    and its space of relations $\FreeRel_\Pca^{\OpDual}$ is generated by
    \begin{subequations}
    \begin{equation} \label{equ:relation_dual_1}
        \OpDual_a \circ_1 \OpDual_a - \OpDual_a \circ_2 \OpDual_a +
        \sum_{\substack{b \in \Pca \\ a \OrdStrict_\Pca b}}
        (\OpDual_b \circ_1 \OpDual_a + \OpDual_a \circ_1 \OpDual_b
        - \OpDual_b \circ_2 \OpDual_a - \OpDual_a \circ_2 \OpDual_b),
        \qquad a \in \Pca,
    \end{equation}
    \begin{equation} \label{equ:relation_dual_2}
        \OpDual_c \circ_1 \OpDual_d,
        \qquad c, d \in \Pca
        \mbox{ and } c \not \Ord_\Pca d
        \mbox{ and } d \not \Ord_\Pca c,
    \end{equation}
    \begin{equation} \label{equ:relation_dual_3}
        \OpDual_c \circ_2 \OpDual_d,
        \qquad c, d \in \Pca
        \mbox{ and } c \not \Ord_\Pca d
        \mbox{ and } d \not \Ord_\Pca c.
    \end{equation}
    \end{subequations}
\end{Proposition}
\begin{proof}
    Let
    \begin{equation}
        x :=
        \sum_{\Tfr \in \Free(\FreeGen_\Pca^\Op)(3)}
        \lambda_\Tfr \, \Tfr,
    \end{equation}
    be an element of $\FreeRel_\Pca^\Op$, where the $\lambda_\Tfr$ are
    elements of $\K$. By definition of the Koszul duality of operads,
    $\langle r, x \rangle = 0$ for all $r \in \FreeRel_\Pca^\Op$,
    where $\langle -, - \rangle$ is the scalar product defined
    in~\eqref{equ:scalar_product_koszul}. Then,
    Relations~\eqref{equ:relation_1} and~\eqref{equ:relation_2} imply
    the relations
    \begin{subequations}
    \begin{equation}
        \lambda_{\Op_a \circ_1 \Op_b} +
        \lambda_{\Op_{a \Min_\Pca b} \circ_2 \Op_{a \Min_\Pca b}} = 0,
    \end{equation}
    \begin{equation}
        \lambda_{\Op_{a \Min_\Pca b} \circ_1 \Op_{a \Min_\Pca b}} +
        \lambda_{\Op_a \circ_2 \Op_b} = 0,
    \end{equation}
    \end{subequations}
    between the $\lambda_\Tfr$, for all $a, b \in \Pca$ such that
    $a \Ord_\Pca b$ or $b \Ord_\Pca a$. This implies that $x$ is of the
    form
    \begin{equation}\begin{split}
        x = \lambda_1
            \sum_{a \in \Pca} (\Op_a \circ_1 \Op_a - \Op_a \circ_2 \Op_a)
            + \lambda_1
            \sum_{\substack{a, b \in \Pca \\ a \OrdStrict_\Pca b}}
            (\Op_b \circ_1 \Op_a + \Op_a \circ_1 \Op_b
            - \Op_b \circ_2 \Op_a - \Op_a \circ_2 \Op_b) \\
            + \lambda_2
            \sum_{\substack{c, d \in \Pca \\ c \not \Ord_\Pca d \\
                d \not \Ord_\Pca c}} (\Op_c \circ_1 \Op_d)
            + \lambda_3
            \sum_{\substack{c, d \in \Pca \\ c \not \Ord_\Pca d \\
                d \not \Ord_\Pca c}} (\Op_c \circ_2 \Op_d),
    \end{split}\end{equation}
    where $\lambda_1$, $\lambda_2$, and $\lambda_3$ are elements of $\K$.
    Therefore, by identifying $\OpDual_a$ with $\Op_a$ for all
    $a \in \Pca$, we obtain that $\FreeRel_\Pca^{\OpDual}$ is generated
    by the elements~\eqref{equ:relation_dual_1},
    \eqref{equ:relation_dual_2}, and~\eqref{equ:relation_dual_3}.
\end{proof}
\medskip

\begin{Proposition} \label{prop:dimension_relations_dual}
    Let $\Pca$ be a poset. Then, the dimension of the space
    $\FreeRel_\Pca^{\OpDual}$ of relations of $\As(\Pca)^!$ satisfies
    \begin{equation} \label{equ:dimension_relations_koszul_dual}
        \dim \FreeRel_\Pca^{\OpDual} =
        2\, (\# \Pca)^2 + 3\, \# \Pca - 4\, \NbInterv(\Pca).
    \end{equation}
\end{Proposition}
\begin{proof}
    To compute the dimension of the space of relations
    $\FreeRel_\Pca^{\OpDual}$ of $\As(\Pca)^!$, we consider the
    presentation of $\As(\Pca)^!$ exhibited by
    Proposition~\ref{prop:presentation_koszul_dual}. Consider the space
    $\FreeRel_1$ generated by the family consisting in the elements
    of~\eqref{equ:relation_dual_1}. Since this family is linearly
    independent and each of its elements is totally specified by an $a$
    of $\Pca$, $\FreeRel_1$ is of dimension $\# \Pca$. Consider now the
    space $\FreeRel_2$ generated by the family consisting in the elements
    of~\eqref{equ:relation_dual_2}. This family is linearly independent
    and each of its elements is totally specified by two incomparable
    elements $c$ and $d$ of $\Pca$. Since the number of pairs of
    comparable elements of $\Pca$ is $2\, \NbInterv(\Pca) - \# \Pca$, we
    obtain
    \begin{equation}
        \dim \FreeRel_2 = (\# \Pca)^2 + \# \Pca - 2\, \NbInterv(\Pca).
    \end{equation}
    For the same reason, the dimension of the space $\FreeRel_3$
    generated by the elements of~\eqref{equ:relation_dual_3} satisfies
    $\dim \FreeRel_3 = \dim \FreeRel_2$. Therefore, since
    \begin{equation}
        \FreeRel_\Pca^{\OpDual} =
        \FreeRel_1 \oplus \FreeRel_2 \oplus \FreeRel_3,
    \end{equation}
    we obtain the stated formula~\eqref{equ:dimension_relations_koszul_dual}
    by summing the dimensions of $\FreeRel_1$, $\FreeRel_2$,
    and~$\FreeRel_3$.
\end{proof}
\medskip

Observe that, by Propositions~\ref{prop:dimension_relations}
and~\ref{prop:dimension_relations_dual}, we have
\begin{equation}\begin{split}
    \dim \FreeRel_\Pca^\Op + \dim \FreeRel_\Pca^{\OpDual} & =
    4 \, \NbInterv(\Pca) - 3 \, \# \Pca +
    2\, (\# \Pca)^2 + 3\, \# \Pca - 4\, \NbInterv(\Pca) \\
    & = 2\, (\# \Pca)^2 \\
    & = \dim \Free(\FreeGen_\Pca^\Op)(3),
\end{split}\end{equation}
as expected by Koszul duality.
\medskip

\subsubsection{Alternative presentation}%
\label{subsubsect:alternative_presentation_koszul_dual}
For any element $a$ of a poset $\Pca$ (not necessarily a forest poset
just now), let $\OpDualB_a$ be the element of
$\Free\left(\FreeGen_\Pca^{\OpDual}\right)(2)$ defined by
\begin{equation} \label{equ:basis_change_dual}
    \OpDualB_a :=
    \sum_{\substack{b \in \Pca \\ a \Ord_\Pca b}}
    \OpDual_b.
\end{equation}
We denote by $\FreeGen_\Pca^{\OpDualB}$ the set of all $\OpDualB_a$,
$a \in \Pca$. By triangularity, the family $\FreeGen_{\Pca}^{\OpDualB}$
forms a basis of $\Free(\FreeGen_\Pca^{\OpDual})(2)$ and hence,
generates $\As(\Pca)^!$. Consider for instance the poset
\begin{equation}
    \Pca :=
    \begin{tikzpicture}
        [baseline=(current bounding box.center),xscale=.4,yscale=.42]
        \node[Vertex](1)at(0,0){\begin{math}1\end{math}};
        \node[Vertex](2)at(2,0){\begin{math}2\end{math}};
        \node[Vertex](3)at(1,-1){\begin{math}3\end{math}};
        \node[Vertex](4)at(3,-1){\begin{math}4\end{math}};
        \node[Vertex](5)at(2,-2){\begin{math}5\end{math}};
        \draw[Edge](1)--(3);
        \draw[Edge](2)--(3);
        \draw[Edge](2)--(4);
        \draw[Edge](4)--(5);
    \end{tikzpicture}\,.
\end{equation}
The elements of $\FreeGen_\Pca^{\OpDualB}$ then express as
\begin{subequations}
\begin{equation}
    \OpDualB_1 = \OpDual_1 + \OpDual_3,
\end{equation}
\begin{equation}
    \OpDualB_2 = \OpDual_2 + \OpDual_3 + \OpDual_4 + \OpDual_5,
\end{equation}
\begin{equation}
    \OpDualB_3 = \OpDual_3,
\end{equation}
\begin{equation}
    \OpDualB_4 = \OpDual_4 + \OpDual_5,
\end{equation}
\begin{equation}
    \OpDualB_5 = \OpDual_5.
\end{equation}
\end{subequations}
\medskip

\begin{Proposition} \label{prop:alternative_presentation_koszul_dual}
    Let $\Pca$ be a forest poset. Then, the operad $\As(\Pca)^!$ admits
    the following presentation. It is generated by
    $\FreeGen_{\Pca}^{\OpDualB}$ and its space of relations
    $\FreeRel_{\Pca}^{\OpDualB}$ is generated by
    \begin{subequations}
    \begin{equation} \label{equ:alternative_relation_dual_1}
        \OpDualB_a \circ_1 \OpDualB_a - \OpDualB_a \circ_2 \OpDualB_a,
        \qquad a \in \Pca,
    \end{equation}
    \begin{equation} \label{equ:alternative_relation_dual_2}
        \OpDualB_c \circ_1 \OpDualB_d,
        \qquad c, d \in \Pca
        \mbox{ and } c \not \Ord_\Pca d
        \mbox{ and } d \not \Ord_\Pca c,
    \end{equation}
    \begin{equation} \label{equ:alternative_relation_dual_3}
        \OpDualB_c \circ_2 \OpDualB_d,
        \qquad c, d \in \Pca
        \mbox{ and } c \not \Ord_\Pca d
        \mbox{ and } d \not \Ord_\Pca c.
    \end{equation}
    \end{subequations}
\end{Proposition}
\begin{proof}
    Let us show that $\FreeRel_{\Pca}^{\OpDualB}$ is equal to the space
    of relations $\FreeRel_\Pca^{\OpDual}$ of $\As(\Pca)^!$ defined in
    the statement of Proposition~\ref{prop:presentation_koszul_dual}. By
    this last proposition, for any
    $x \in \Free(\FreeGen_\Pca^{\OpDual})(3)$, $x$ is in
    $\FreeRel_\Pca^{\OpDual}$ if and only if $\pi(x) = 0$ where
    $\pi : \Free(\FreeGen_\Pca^{\OpDual}) \to \As(\Pca)^!$ is the
    canonical projection.
    \smallskip

    Let us compute the image by $\pi$ of the
    elements~\eqref{equ:alternative_relation_dual_1},
    \eqref{equ:alternative_relation_dual_2},
    and~\eqref{equ:alternative_relation_dual_3} generating
    $\FreeRel_\Pca^{\OpDualB}$, by expanding these over the elements
    $\OpDual_a$, $a \in \Pca$, by using~\eqref{equ:basis_change_dual}.
    We first have, for all $a \in \Pca$,
    \begin{equation}\begin{split}
    \label{equ:alternative_presentation_koszul_dual_computation_1}
        \pi\left(\OpDualB_a \circ_1 \OpDualB_a
                - \OpDualB_a \circ_2 \OpDualB_a\right)
            & = \sum_{\substack{b, b' \in \Pca \\
                a \Ord_\Pca b \\ a \Ord_\Pca b'}}
            \pi\left(\OpDual_b \circ_1 \OpDual_{b'}\right) -
            \pi\left(\OpDual_b \circ_2 \OpDual_{b'}\right) \\
            & = \sum_{\substack{b \in \Pca \\
                a \Ord_\Pca b}}
            \pi\left(\OpDual_b \circ_1 \OpDual_b -
            \OpDual_b \circ_2 \OpDual_b\right) \\
            & \qquad +
            \sum_{\substack{b \in \Pca \\
                a \Ord_\Pca b}} \;
            \sum_{\substack{b' \in \Pca \\
                b \OrdStrict_\Pca b'}}
            \pi\left(\OpDual_b \circ_1 \OpDual_{b'}
                + \OpDual_{b'} \circ_1 \OpDual_b
                - \OpDual_b \circ_2 \OpDual_{b'}
                - \OpDual_{b'} \circ_2 \OpDual_b\right) \\
            & = 0.
    \end{split}\end{equation}
    Indeed, the second equality
    of~\eqref{equ:alternative_presentation_koszul_dual_computation_1}
    comes from the fact that $\pi(\OpDual_b \circ_i \OpDual_{b'}) = 0$
    for all $i = 1$ or $i = 2$ whenever $b$ and $b'$ are incomparable
    elements of $\Pca$. The last equality
    of~\eqref{equ:alternative_presentation_koszul_dual_computation_1}
    is a consequence of the fact that~\eqref{equ:relation_dual_1} is
    in $\FreeRel_\Pca^{\OpDual}$. Moreover, for all incomparable
    elements $c$ and $d$ of $\Pca$, we have
    \begin{equation}\begin{split}
        \pi\left(\OpDualB_c \circ_1 \OpDualB_d\right)
            & = \sum_{\substack{c', d' \in \Pca \\
                c \Ord_\Pca c' \\ d \Ord_\Pca d'}}
            \pi\left(\OpDual_{c'} \circ_1 \OpDual_{d'}\right) = 0.
    \end{split}\end{equation}
    Indeed, since $\Pca$ is a forest poset, for all $c', d' \in \Pca$
    such that $c \Ord_\Pca c'$ and $d \Ord_\Pca d'$, $c'$ and $d'$ are
    incomparable in $\Pca$. We have shown that $\FreeRel_\Pca^{\OpDualB}$
    is a subspace of $\FreeRel_\Pca^{\OpDual}$.
    \smallskip

    Finally, one can observe that all the
    elements~\eqref{equ:alternative_relation_dual_1},
    \eqref{equ:alternative_relation_dual_2},
    and~\eqref{equ:alternative_relation_dual_3} are linearly independent,
    and, by using the same arguments as the ones used in the proof of
    Proposition~\ref{prop:dimension_relations_dual}, we obtain
    \begin{equation}
        \dim \FreeRel_\Pca^{\OpDualB} =
        2\, (\# \Pca)^2 + 3\, \# \Pca - 4\, \NbInterv(\Pca).
    \end{equation}
    This shows that $\FreeRel_\Pca^{\OpDualB}$ and
    $\FreeRel_\Pca^{\OpDual}$ and have the same dimension. The statement
    of the proposition follows.
\end{proof}
\medskip

By considering the presentation of $\As(\Pca)^!$ furnished by
Proposition~\ref{prop:alternative_presentation_koszul_dual} when $\Pca$
is a forest poset, we obtain by Koszul duality a new presentation
$(\FreeGen_\Pca^{\OpB}, \FreeRel_\Pca^{\OpB})$ for $\As(\Pca)$ where the
set of generators $\FreeGen_\Pca^{\OpB}$ is defined by
\begin{equation}
    \FreeGen_\Pca^{\OpB} := \FreeGen_\Pca^{\OpB}(2)
    := \{\OpB_a : a \in \Pca\},
\end{equation}
and the space of relations $\FreeRel_\Pca^{\OpB}$ is generated by
\begin{subequations}
\begin{equation}
    \OpB_a \circ_1 \OpB_a - \OpB_a \circ_2 \OpB_a,
    \qquad a \in \Pca,
\end{equation}
\begin{equation}
    \OpB_a \circ_1 \OpB_b,
    \qquad a, b \in \Pca
    \mbox{ and (} a \OrdStrict_\Pca b
    \mbox{ or } b \OrdStrict_\Pca a \mbox{)},
\end{equation}
\begin{equation}
    \OpB_a \circ_2 \OpB_b,
    \qquad a, b \in \Pca
    \mbox{ and (} a \OrdStrict_\Pca b
    \mbox{ or } b \OrdStrict_\Pca a \mbox{)}.
\end{equation}
\end{subequations}
\medskip

\subsubsection{Example}
To end this section, let us give a complete example of the spaces of
relations $\FreeRel_\Pca^\Op$, $\FreeRel_\Pca^{\OpDual}$,
$\FreeRel_\Pca^{\OpDualB}$, and $\FreeRel_\Pca^{\OpB}$ of the operads
$\As(\Pca)$ and $\As(\Pca)^!$ where $\Pca$ is the forest poset
\begin{equation}
    \Pca :=
    \begin{tikzpicture}
        [baseline=(current bounding box.center),xscale=.35,yscale=.42]
        \node[Vertex](1)at(0,0){\begin{math}1\end{math}};
        \node[Vertex](2)at(-1,-1){\begin{math}2\end{math}};
        \node[Vertex](3)at(1,-1){\begin{math}3\end{math}};
        \draw[Edge](1)--(2);
        \draw[Edge](1)--(3);
    \end{tikzpicture}\,.
\end{equation}
\medskip

First, by definition of $\As$ and by Lemma~\ref{lem:relations}, the
space of relations  $\FreeRel_\Pca^\Op$ of $\As(\Pca)$ contains
\begin{subequations}
\begin{equation}\begin{split}
    \Op_1 \circ_1 \Op_1
    = \Op_1 \circ_1 \Op_2
    & = \Op_2 \circ_1 \Op_1
    = \Op_1 \circ_1 \Op_3
    = \Op_3 \circ_1 \Op_1 \\
    & = \Op_3 \circ_2 \Op_1
    = \Op_1 \circ_2 \Op_3
    = \Op_2 \circ_2 \Op_1
    = \Op_1 \circ_2 \Op_2
    = \Op_1 \circ_2 \Op_1,
\end{split}
\end{equation}
\begin{equation}
    \Op_2 \circ_1 \Op_2 = \Op_2 \circ_2 \Op_2,
\end{equation}
\begin{equation}
    \Op_3 \circ_1 \Op_3 = \Op_3 \circ_2 \Op_3.
\end{equation}
\end{subequations}
\medskip

By Proposition~\ref{prop:presentation_koszul_dual}, the space of
relations  $\FreeRel_\Pca^{\OpDual}$ of $\As(\Pca)^!$ contains
\begin{subequations}
\begin{equation}\begin{split}
    \OpDual_1 \circ_1 \OpDual_1
    + \OpDual_1 \circ_1 \OpDual_2
    & + \OpDual_2 \circ_1 \OpDual_1
    + \OpDual_1 \circ_1 \OpDual_3
    + \OpDual_3 \circ_1 \OpDual_1 \\
    & = \OpDual_3 \circ_2 \OpDual_1
    + \OpDual_1 \circ_2 \OpDual_3
    + \OpDual_2 \circ_2 \OpDual_1
    + \OpDual_1 \circ_2 \OpDual_2
    + \OpDual_1 \circ_2 \OpDual_1,
\end{split}
\end{equation}
\begin{equation}
    \OpDual_2 \circ_1 \OpDual_2 = \OpDual_2 \circ_2 \OpDual_2,
\end{equation}
\begin{equation}
    \OpDual_3 \circ_1 \OpDual_3 = \OpDual_3 \circ_2 \OpDual_3,
\end{equation}
\begin{equation}
    \OpDual_2 \circ_1 \OpDual_3
    = \OpDual_3 \circ_1 \OpDual_2
    = \OpDual_3 \circ_2 \OpDual_2
    = \OpDual_2 \circ_2 \OpDual_3 = 0.
\end{equation}
\end{subequations}
\medskip

By Proposition~\ref{prop:alternative_presentation_koszul_dual}, the
space of relations  $\FreeRel_\Pca^{\OpDualB}$ of $\As(\Pca)^!$ contains
\begin{subequations}
\begin{equation}
    \OpDualB_1 \circ_1 \OpDualB_1 = \OpDualB_1 \circ_2 \OpDualB_1,
\end{equation}
\begin{equation}
    \OpDualB_2 \circ_1 \OpDualB_2 = \OpDualB_2 \circ_2 \OpDualB_2,
\end{equation}
\begin{equation}
    \OpDualB_3 \circ_1 \OpDualB_3 = \OpDualB_3 \circ_2 \OpDualB_3,
\end{equation}
\begin{equation}
    \OpDualB_2 \circ_1 \OpDualB_3
    = \OpDualB_3 \circ_1 \OpDualB_2
    = \OpDualB_3 \circ_2 \OpDualB_2
    = \OpDualB_2 \circ_2 \OpDualB_3 = 0.
\end{equation}
\end{subequations}
\medskip

Finally, by the observation established at the end of
Section~\ref{subsubsect:alternative_presentation_koszul_dual}, the space
of relations  $\FreeRel_\Pca^{\OpB}$ of $\As(\Pca)$ contains
\begin{subequations}
\begin{equation}
    \OpB_1 \circ_1 \OpB_1 = \OpB_1 \circ_2 \OpB_1,
\end{equation}
\begin{equation}
    \OpB_2 \circ_1 \OpB_2 = \OpB_2 \circ_2 \OpB_2,
\end{equation}
\begin{equation}
    \OpB_3 \circ_1 \OpB_3 = \OpB_3 \circ_2 \OpB_3,
\end{equation}
\begin{equation}\begin{split}
    \OpB_1 \circ_1 \OpB_2
    = \OpB_2 \circ_1 \OpB_1
    & = \OpB_1 \circ_1 \OpB_3
    = \OpB_3 \circ_1 \OpB_1 \\
    & = \OpB_3 \circ_2 \OpB_1
    = \OpB_1 \circ_2 \OpB_3
    = \OpB_2 \circ_2 \OpB_1
    = \OpB_1 \circ_2 \OpB_2 = 0.
\end{split}\end{equation}
\end{subequations}
\medskip

\section{Thin forest posets and Koszul duality}%
\label{sec:thin_forest_posets}
As we have seen in Section~\ref{sec:forest_posets}, certain properties
satisfied by the poset $\Pca$ imply properties for the operad $\As(\Pca)$.
In this section, we show that when $\Pca$ is a forest poset with an extra
condition, the Koszul dual $\As(\Pca)^!$ of $\As(\Pca)$ can be
constructed via the construction $\As$.
\medskip

\subsection{Thin forest posets}
A subclass of the class of forest posets, whose elements are called thin
forest posets, is described here. We also define an involution on these
posets linked, as we shall see later, to Koszul duality of the concerned
operads.
\medskip

\subsubsection{Description}
A {\em thin forest poset} is a forest poset avoiding the pattern
$\LinePattern\; \LinePattern$. In other words, a thin forest poset is a
poset so that the nonplanar rooted tree $\Tfr$ obtained by adding a
(new) root to the Hasse diagram of $\Pca$ has the following property.
Any node $x$ of $\Tfr$ has at most one child $y$ such that the bottom
subtree of $\Tfr$ rooted at $y$ has two nodes or more. For instance, any
forest poset admitting the Hasse diagram
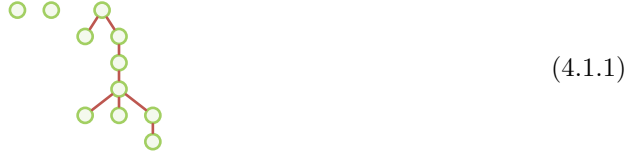
\begin{equation} \label{equ:example_thin_forest_poset}
    \begin{tikzpicture}
        [baseline=(current bounding box.center),xscale=.45,yscale=.35]
        \node[Vertex](1)at(3,0){};
        \node[Vertex](2)at(4,0){};
        \node[Vertex](3)at(5.5,0){};
        \node[Vertex](4)at(5,-1){};
        \node[Vertex](5)at(6,-1){};
        \node[Vertex](6)at(6,-2){};
        \node[Vertex](7)at(6,-3){};
        \node[Vertex](8)at(5,-4){};
        \node[Vertex](9)at(6,-4){};
        \node[Vertex](A)at(7,-4){};
        \node[Vertex](B)at(7,-5){};
        \draw[Edge](3)--(4);
        \draw[Edge](3)--(5);
        \draw[Edge](5)--(6);
        \draw[Edge](6)--(7);
        \draw[Edge](7)--(8);
        \draw[Edge](7)--(9);
        \draw[Edge](7)--(A);
        \draw[Edge](A)--(B);
    \end{tikzpicture}
\end{equation}
is a thin forest poset while any forest poset admitting the Hasse diagram
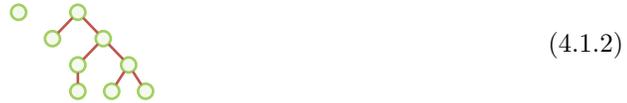
\begin{equation}
    \begin{tikzpicture}
        [baseline=(current bounding box.center),xscale=.45,yscale=.35]
        \node[Vertex](1)at(0,0){};
        \node[Vertex](2)at(1.75,0){};
        \node[Vertex](3)at(1,-1){};
        \node[Vertex](4)at(2.5,-1){};
        \node[Vertex](5)at(1.75,-2){};
        \node[Vertex](6)at(1.75,-3){};
        \node[Vertex](7)at(3.25,-2){};
        \node[Vertex](8)at(2.75,-3){};
        \node[Vertex](9)at(3.75,-3){};
        \draw[Edge](2)--(3);
        \draw[Edge](2)--(4);
        \draw[Edge](4)--(5);
        \draw[Edge](4)--(7);
        \draw[Edge](5)--(6);
        \draw[Edge](7)--(8);
        \draw[Edge](7)--(9);
    \end{tikzpicture}
\end{equation}
is not.
\medskip

The {\em standard labeling} of a thin forest poset $\Pca$ consists in
labeling the vertices of the Hasse diagram of $\Pca$ from $1$ to
$\# \Pca$ in the order they appear in a depth first traversal, by always
visiting in a same sibling the node with the biggest subtree as last.
For instance, the standard labeling of a poset
admitting~\eqref{equ:example_thin_forest_poset} as Hasse diagram
produces the Hasse diagram
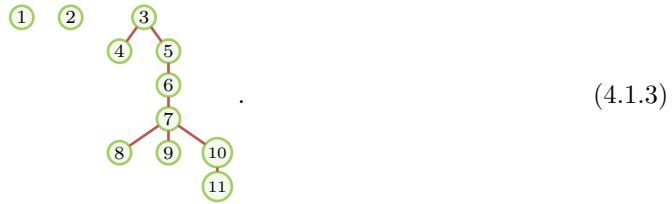
\begin{equation}
    \begin{tikzpicture}
        [baseline=(current bounding box.center),xscale=.65,yscale=.45]
        \node[Vertex](1)at(3,0){\begin{math}1\end{math}};
        \node[Vertex](2)at(4,0){\begin{math}2\end{math}};
        \node[Vertex](3)at(5.5,0){\begin{math}3\end{math}};
        \node[Vertex](4)at(5,-1){\begin{math}4\end{math}};
        \node[Vertex](5)at(6,-1){\begin{math}5\end{math}};
        \node[Vertex](6)at(6,-2){\begin{math}6\end{math}};
        \node[Vertex](7)at(6,-3){\begin{math}7\end{math}};
        \node[Vertex](8)at(5,-4){\begin{math}8\end{math}};
        \node[Vertex](9)at(6,-4){\begin{math}9\end{math}};
        \node[Vertex](A)at(7,-4){\tiny \begin{math}10\end{math}};
        \node[Vertex](B)at(7,-5){\tiny \begin{math}11\end{math}};
        \draw[Edge](3)--(4);
        \draw[Edge](3)--(5);
        \draw[Edge](5)--(6);
        \draw[Edge](6)--(7);
        \draw[Edge](7)--(8);
        \draw[Edge](7)--(9);
        \draw[Edge](7)--(A);
        \draw[Edge](A)--(B);
    \end{tikzpicture}\,.
\end{equation}
In what follows, we shall consider only standardly labeled thin forest
posets and we shall identify any element $x$ of a thin forest posets
$\Pca$ as the label of $x$ in the standard labeling of $\Pca$. Moreover,
we shall see thin forest posets as forests of nonplanar rooted trees,
obtained by considering Hasse diagrams of these posets.
\medskip

Thin forest posets admit the following recursive description. If $\Pca$
is a thin forest poset, then $\Pca$ is the empty forest $\PosetEmpty$,
or it is the forest
\begin{equation}
    \raisebox{.25em}{\OnePoset} \; \Pca'
\end{equation}
consisting in the tree of one node $\raisebox{.25em}{\OnePoset}$
(labeled by $1$) and a thin forest poset $\Pca'$, or it is the forest
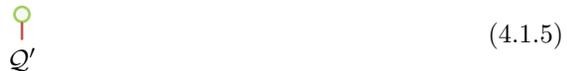
\begin{equation}
    \begin{tikzpicture}[baseline=(current bounding box.center),yscale=.6]
        \node[Vertex](1)at(0,0){};
        \node(2)at(0,-1){\begin{math}\Pca'\end{math}};
        \draw[Edge](1)--(2);
    \end{tikzpicture}
\end{equation}
consisting in one root (labeled by $1$) attached to the roots of the
trees of the thin forest poset~$\Pca'$. Therefore, there are $2^{n - 1}$
thin forest posets of size $n \geq 1$.
\medskip

\subsubsection{Duality}
Given a thin forest poset $\Pca$, the {\em dual} of $\Pca$ is the poset
$\Pca^\perp$ such that, for all $a, b \in \Pca^\perp$,
$a \Ord_{\Pca^\perp} b$ if and only if $a = b$ or $a$ and $b$ are
incomparable in $\Pca$ and $a < b$. For instance, consider the poset
\begin{equation}
    \Pca :=
    \begin{tikzpicture}
        [baseline=(current bounding box.center),xscale=.45,yscale=.42]
        \node[Vertex](1)at(0,2){\begin{math}1\end{math}};
        \node[Vertex](2)at(2,2){\begin{math}2\end{math}};
        \node[Vertex](3)at(1,1){\begin{math}3\end{math}};
        \node[Vertex](4)at(2,1){\begin{math}4\end{math}};
        \node[Vertex](5)at(3,1){\begin{math}5\end{math}};
        \node[Vertex](6)at(3,0){\begin{math}6\end{math}};
        \draw[Edge](2)--(3);
        \draw[Edge](2)--(4);
        \draw[Edge](2)--(5);
        \draw[Edge](5)--(6);
    \end{tikzpicture}\,.
\end{equation}
Since $1 \not \Ord_\Pca 2$, $1 \not \Ord_\Pca 3$, $1 \not \Ord_\Pca 4$,
$1 \not \Ord_\Pca 5$, $1 \not \Ord_\Pca 6$, $3 \not \Ord_\Pca 4$,
$3 \not \Ord_\Pca 5$, $3 \not \Ord_\Pca 6$, $4 \not \Ord_\Pca 5$,
$4 \not \Ord_\Pca 6$, in the dual $\Pca^\perp$ of $\Pca$ we have
$1 \Ord_{\Pca^\perp} 2$, $1 \Ord_{\Pca^\perp} 3$, $1 \Ord_{\Pca^\perp} 4$,
$1 \Ord_{\Pca^\perp} 5$, $1 \Ord_{\Pca^\perp} 6$, $3 \Ord_{\Pca^\perp} 4$,
$3 \Ord_{\Pca^\perp} 5$, $3 \Ord_{\Pca^\perp} 6$, $4 \Ord_{\Pca^\perp} 5$,
$4 \Ord_{\Pca^\perp} 6$ and hence,
\begin{equation}
    \Pca^\perp =
    \begin{tikzpicture}
        [baseline=(current bounding box.center),xscale=.45,yscale=.42]
        \node[Vertex](1)at(.25,3){\begin{math}1\end{math}};
        \node[Vertex](2)at(-.5,2){\begin{math}2\end{math}};
        \node[Vertex](3)at(1,2){\begin{math}3\end{math}};
        \node[Vertex](4)at(1,1){\begin{math}4\end{math}};
        \node[Vertex](5)at(0.5,0){\begin{math}5\end{math}};
        \node[Vertex](6)at(1.5,0){\begin{math}6\end{math}};
        \draw[Edge](1)--(2);
        \draw[Edge](1)--(3);
        \draw[Edge](3)--(4);
        \draw[Edge](4)--(5);
        \draw[Edge](4)--(6);
    \end{tikzpicture}\,.
\end{equation}
Observe that this operation $^\perp$ is an involution on thin forest
posets.
\medskip

We now state two lemmas about thin forest posets and the
operation~$^\perp$.
\medskip

\begin{Lemma} \label{lem:recursive_description_dual_poset}
    Let $\Pca$ be a thin forest poset. The dual $\Pca^\perp$ of $\Pca$
    admits the following recursive expression:
    \begin{enumerate}[label={\it (\roman*)}]
        \item if $\Pca$ is the empty forest $\PosetEmpty$, then
        \begin{equation}
            \PosetEmpty^\perp = \PosetEmpty;
        \end{equation}
        \item if $\Pca$ is of the form
        $\Pca = \raisebox{.25em}{\OnePoset} \; \Pca'$ where $\Pca'$ is
        a thin forest poset, then
        \begin{equation}
            \left(\raisebox{.25em}{\OnePoset} \; \Pca'\right)^\perp =
            \begin{tikzpicture}
                [baseline=(current bounding box.center),yscale=.6]
                \node[Vertex](1)at(0,0){};
                \node(2)at(0,-1){\begin{math}{\Pca'}^\perp\end{math}};
                \draw[Edge](1)--(2);
            \end{tikzpicture};
        \end{equation}
        \item otherwise, $\Pca$ is of the form
        \begin{math}
            \Pca =
            \begin{tikzpicture}
                [baseline=(current bounding box.center),yscale=.4]
                \node[Vertex](1)at(0,0){};
                \node(2)at(0,-1){\footnotesize\begin{math}\Pca'\end{math}};
                \draw[Edge](1)--(2);
            \end{tikzpicture}
        \end{math}
        where $\Pca'$ is a thin forest poset, and then
        \begin{equation}
            \left(
            \begin{tikzpicture}
                [baseline=(current bounding box.center),yscale=.6]
                \node[Vertex](1)at(0,0){};
                \node(2)at(0,-1){\begin{math}\Pca'\end{math}};
                \draw[Edge](1)--(2);
            \end{tikzpicture}
            \right)^\perp
            =
            \raisebox{.25em}{\OnePoset} \; {\Pca'}^\perp.
        \end{equation}
    \end{enumerate}
\end{Lemma}
\begin{proof}
    We denote by $^\vee$ the operation on thin forest posets defined
    in the statement of the lemma. To prove that $^\vee = ^\perp$, we
    proceed by structural induction on $\Pca$. If $\Pca = \PosetEmpty$,
    we have
    \begin{math}
        \Pca^\perp = \PosetEmpty^\perp =
        \PosetEmpty = \PosetEmpty^\vee
        = \Pca^\vee.
    \end{math}
    Else, if $\Pca$ is of the form
    $\Pca = \raisebox{.25em}{\OnePoset} \; \Pca'$ where $\Pca'$ is
    a thin forest poset, by definition of $^\perp$, since
    $1 \Ord_{\Pca} a$ implies $a = 1$,
    we have $1 \Ord_{\Pca^\perp} a$ for all $a \in \Pca^\perp$. Now,
    by definition of $^\vee$, we have
    \begin{equation}
        \Pca^\vee =
        \left(\raisebox{.25em}{\OnePoset} \; \Pca'\right)^\vee =
        \begin{tikzpicture}[baseline=(current bounding box.center),yscale=.6]
            \node[Vertex](1)at(0,0){};
            \node(2)at(0,-1){\begin{math}{\Pca'}^\vee\end{math}};
            \draw[Edge](1)--(2);
        \end{tikzpicture}.
    \end{equation}
    Then, since $1 \Ord_{\Pca^\vee} a$ for all $a \in \Pca^\vee$
    and, by induction hypothesis, ${\Pca'}^\perp = {\Pca'}^\vee$, we
    obtain that $\Pca^\perp = \Pca^\vee$.
    Otherwise, $\Pca$ is of the form
    \begin{math}
        \Pca =
        \begin{tikzpicture}[baseline=(current bounding box.center),yscale=.4]
            \node[Vertex](1)at(0,0){};
            \node(2)at(0,-1){\footnotesize\begin{math}\Pca'\end{math}};
            \draw[Edge](1)--(2);
        \end{tikzpicture}
    \end{math}
    where $\Pca'$ is a thin forest poset. By definition of $\perp$,
    since $1 \Ord_\Pca a$ for all $a \in \Pca$, we have
    $1 \Ord_{\Pca^\perp} a$ implies $a = 1$. Now, by definition of
    $\vee$, we have
    \begin{equation}
        \Pca^\vee =
        \left(
        \begin{tikzpicture}
            [baseline=(current bounding box.center),yscale=.6]
            \node[Vertex](1)at(0,0){};
            \node(2)at(0,-1){\begin{math}\Pca'\end{math}};
            \draw[Edge](1)--(2);
        \end{tikzpicture}
        \right)^\vee
        =
        \raisebox{.25em}{\OnePoset} \; {\Pca'}^\vee.
    \end{equation}
    Then, since $1 \Ord_{\Pca^\vee} a$ implies $a = 1$ and, by
    induction hypothesis, ${\Pca'}^\perp = {\Pca'}^\vee$, we obtain
    that $\Pca^\perp = \Pca^\vee$.
\end{proof}
\medskip

\begin{Lemma} \label{lem:relation_intervals_poset_dual}
    Let $\Pca$ be a thin forest poset. Then, the number of intervals of
    $\Pca$ and the number of intervals of its dual are related by
    \begin{equation}
        \NbInterv(\Pca) + \NbInterv\left(\Pca^\perp\right) =
        \frac{(\# \Pca)^2 + 3\, \# \Pca}{2}.
    \end{equation}
\end{Lemma}
\begin{proof}
    We proceed by structural induction on $\Pca$. If $\Pca = \PosetEmpty$,
    then $\Pca^\perp = \Pca$ and
    $\NbInterv(\Pca) + \NbInterv\left(\Pca^\perp\right) = 0$, so that
    the statement of the lemma is satisfied. Else, if $\Pca$ is of the
    form $\Pca = \raisebox{.25em}{\OnePoset} \; \Pca'$ where $\Pca'$ is
    a thin forest poset, by
    Lemma~\ref{lem:recursive_description_dual_poset},
    \begin{equation}
        \left(\raisebox{.25em}{\OnePoset} \; \Pca'\right)^\perp =
        \begin{tikzpicture}[baseline=(current bounding box.center),yscale=.6]
            \node[Vertex](1)at(0,0){};
            \node(2)at(0,-1){\begin{math}{\Pca'}^\perp\end{math}};
            \draw[Edge](1)--(2);
        \end{tikzpicture}\,.
    \end{equation}
    Now, we have
    \begin{equation}
        \NbInterv(\Pca) =
        \NbInterv\left(\raisebox{.25em}{\OnePoset} \; \Pca'\right) =
        1 + \NbInterv(\Pca')
    \end{equation}
    and
    \begin{equation}
        \NbInterv(\Pca^\perp) =
        \NbInterv\left(
        \begin{tikzpicture}
                [baseline=(current bounding box.center),yscale=.6]
            \node[Vertex](1)at(0,0){};
            \node(2)at(0,-1){\begin{math}{\Pca'}^\perp\end{math}};
            \draw[Edge](1)--(2);
        \end{tikzpicture}
        \right) =
        1 + \NbInterv\left({\Pca'}^\perp\right) + \# {\Pca'}^\perp.
    \end{equation}
    By induction hypothesis, we obtain that
    $\Pca$ satisfies the statement of the lemma. Otherwise, $\Pca$ is of
    the form
    \begin{math}
        \Pca =
        \begin{tikzpicture}
                [baseline=(current bounding box.center),yscale=.4]
            \node[Vertex](1)at(0,0){};
            \node(2)at(0,-1){\footnotesize\begin{math}\Pca'\end{math}};
            \draw[Edge](1)--(2);
        \end{tikzpicture}
    \end{math}
    where $\Pca'$ is a thin forest poset. By
    Lemma~\ref{lem:recursive_description_dual_poset},
    \begin{equation}
        \Pca^\perp =
        \left(
        \begin{tikzpicture}
                [baseline=(current bounding box.center),yscale=.6]
            \node[Vertex](1)at(0,0){};
            \node(2)at(0,-1){\begin{math}\Pca'\end{math}};
            \draw[Edge](1)--(2);
        \end{tikzpicture}
        \right)^\perp
        =
        \raisebox{.25em}{\OnePoset} \; {\Pca'}^\perp,
    \end{equation}
    so that
    \begin{equation}
        \NbInterv(\Pca) =
        \NbInterv\left(
        \begin{tikzpicture}
                [baseline=(current bounding box.center),yscale=.6]
            \node[Vertex](1)at(0,0){};
            \node(2)at(0,-1){\begin{math}{\Pca'}\end{math}};
            \draw[Edge](1)--(2);
        \end{tikzpicture}
        \right) =
        1 + \NbInterv\left({\Pca'}\right) + \# {\Pca'}
    \end{equation}
    and
    \begin{equation}
        \NbInterv(\Pca^\perp) =
        \NbInterv\left(\raisebox{.25em}{\OnePoset} \; {\Pca'}^\perp\right)
        = 1 + \NbInterv\left({\Pca'}^\perp\right).
    \end{equation}
    By using the same arguments as those used for the previous case, the
    statement of the lemma is established.
\end{proof}
\medskip

\subsection{Koszul duality and poset duality}
By defining here an alternative basis for $\As(\Pca)$ when $\Pca$ is a
thin forest poset, we show that the construction $\As$ is closed under
Koszul duality on thin forest posets. More precisely, we show that
$\As(\Pca)^!$ and $\As(\Pca^\perp)$ are two isomorphic operads.
\medskip

\subsubsection{Alternative basis}
Let $\Pca$ be a thin forest poset. For any element $b$ of $\Pca$, let
$\OpDualA_b$ be the element of $\Free(\FreeGen_\Pca^{\OpDualB})(2)$
defined by
\begin{equation} \label{equ:basis_change_dual_thin_forest_poset}
    \OpDualA_b :=
    \sum_{\substack{a \in \Pca^\perp \\ a \Ord_{\Pca^\perp} b}}
    \OpDualB_a.
\end{equation}
We denote by $\FreeGen_\Pca^{\OpDualA}$ the set of all $\OpDualA_b$,
$b \in \Pca$. By triangularity, the family $\FreeGen_\Pca^{\OpDualA}$
forms a basis of $\Free(\FreeGen_\Pca^{\OpDualB})(2)$ and hence,
generates $\As(\Pca)^!$. Consider for instance the thin forest poset
\begin{equation} \label{equ:thin_forest_poset_example}
    \Pca :=
    \begin{tikzpicture}
        [baseline=(current bounding box.center),xscale=.45,yscale=.42]
        \node[Vertex](1)at(0,0){\begin{math}1\end{math}};
        \node[Vertex](2)at(1,0){\begin{math}2\end{math}};
        \node[Vertex](3)at(2.75,0){\begin{math}3\end{math}};
        \node[Vertex](4)at(2.25,-1){\begin{math}4\end{math}};
        \node[Vertex](5)at(3.25,-1){\begin{math}5\end{math}};
        \node[Vertex](6)at(3.25,-2){\begin{math}6\end{math}};
        \draw[Edge](3)--(4);
        \draw[Edge](3)--(5);
        \draw[Edge](5)--(6);
    \end{tikzpicture}\,.
\end{equation}
The dual poset of $\Pca$ is
\begin{equation} \label{equ:thin_forest_poset_dual_example}
    \Pca^\perp =
    \begin{tikzpicture}
        [baseline=(current bounding box.center),xscale=.45,yscale=.42]
        \node[Vertex](1)at(0,0){\begin{math}1\end{math}};
        \node[Vertex](2)at(0,-1){\begin{math}2\end{math}};
        \node[Vertex](3)at(-1,-2){\begin{math}3\end{math}};
        \node[Vertex](4)at(1,-2){\begin{math}4\end{math}};
        \node[Vertex](5)at(.5,-3){\begin{math}5\end{math}};
        \node[Vertex](6)at(1.5,-3){\begin{math}6\end{math}};
        \draw[Edge](1)--(2);
        \draw[Edge](2)--(3);
        \draw[Edge](2)--(4);
        \draw[Edge](4)--(5);
        \draw[Edge](4)--(6);
    \end{tikzpicture}
\end{equation}
and hence, the elements of $\FreeGen_\Pca^{\OpDualA}$ express as
\begin{subequations}
\begin{equation}
    \OpDualA_1 = \OpDualB_1,
\end{equation}
\begin{equation}
    \OpDualA_2 = \OpDualB_1 + \OpDualB_2,
\end{equation}
\begin{equation}
    \OpDualA_3 = \OpDualB_1 + \OpDualB_2 + \OpDualB_3,
\end{equation}
\begin{equation}
    \OpDualA_4 = \OpDualB_1 + \OpDualB_2 + \OpDualB_4,
\end{equation}
\begin{equation}
    \OpDualA_5 = \OpDualB_1 + \OpDualB_2 + \OpDualB_4 + \OpDualB_5,
\end{equation}
\begin{equation}
    \OpDualA_6 = \OpDualB_1 + \OpDualB_2 + \OpDualB_4 + \OpDualB_6.
\end{equation}
\end{subequations}
\medskip

\begin{Lemma} \label{lem:dimensions_relations_thin_forest_poset}
    Let $\Pca$ be a thin forest poset. Then, the dimension of the space
    $\FreeRel_{\Pca^\perp}^\Op$ of relations of $\As(\Pca^\perp)$ and
    the dimension of the space $\FreeRel_\Pca^{\OpDual}$ of relations of
    $\As(\Pca)^!$ are related by
    \begin{equation} \label{equ:dimensions_relations_thin_forest_poset}
        \dim \FreeRel_{\Pca^\perp}^\Op =
        4 \, \NbInterv\left(\Pca^\perp\right) - 3 \, \# \Pca =
        \dim \FreeRel_\Pca^{\OpDual}.
    \end{equation}
\end{Lemma}
\begin{proof}
    First, by Proposition~\ref{prop:dimension_relations}, the first
    equality of~\eqref{equ:dimensions_relations_thin_forest_poset} is
    established.
    \smallskip

    Moreover, by Proposition~\ref{prop:dimension_relations_dual} and by
    Lemma~\ref{lem:relation_intervals_poset_dual}, we have
    \begin{equation}\begin{split}
        \dim \FreeRel_\Pca^{\OpDual} & =
            2\, (\# \Pca)^2 + 3\, \# \Pca - 4\, \NbInterv(\Pca) \\
        & = 2\, (\# \Pca)^2 + 3\, \# \Pca -
            4 \frac{(\# \Pca)^2 + 3\, \# \Pca}{2}
            + 4\, \NbInterv\left(\Pca^\perp\right) \\
        & = 4\, \NbInterv\left(\Pca^\perp\right) - 3 \, \# \Pca,
    \end{split}\end{equation}
    establishing the second equality
    of~\eqref{equ:dimensions_relations_thin_forest_poset} and implying
    the statement of the lemma.
\end{proof}
\medskip

\subsubsection{Isomorphism}

\begin{Theorem} \label{thm:isomorphism_thin_forests_dual_operad}
    Let $\Pca$ be a thin forest poset. Then, the map
    $\phi : \As(\Pca^\perp) \to \As(\Pca)^!$ defined for any
    $a \in \Pca^\perp$ by $\phi(\Op_a) := \OpDualA_a$ extends in a
    unique way to an isomorphism of operads.
\end{Theorem}
\begin{proof}
    Let us denote by $\FreeRel_\Pca^{\OpDualA}$ the space of relations
    of $\As(\Pca)^!$, expressed on the generating family
    $\FreeGen_\Pca^{\OpDualA}$. This space is the same as the space
    $\FreeRel_\Pca^{\OpDualB}$, described by
    Proposition~\ref{prop:alternative_presentation_koszul_dual}. Let us
    exhibit a generating family of $\FreeRel_\Pca^{\OpDualA}$ as a vector
    space. For this, let
    $\pi : \Free(\FreeGen_\Pca^{\OpDualB}) \to \As(\Pca)^!$ be the
    canonical projection and let $a$ and $b$ be two elements of
    $\Pca^\perp$ such that $a \Ord_{\Pca^\perp} b$. We have, by
    using~\eqref{equ:basis_change_dual_thin_forest_poset},
    \begin{equation}\begin{split}
    \label{equ:isomorphism_thin_forests_dual_operad_computation_1}
        \pi\left(\OpDualA_a \circ_1 \OpDualA_b
                - \OpDualA_a \circ_2 \OpDualA_a\right) & =
        \sum_{\substack{a', b' \in \Pca^\perp \\
            a' \Ord_{\Pca^\perp} a \\
            b' \Ord_{\Pca^\perp} b}}
        \pi\left(\OpDualB_{a'} \circ_1 \OpDualB_{b'}\right)
        -
        \sum_{\substack{a', a'' \in \Pca^\perp \\
            a' \Ord_{\Pca^\perp} a \\
            a'' \Ord_{\Pca^\perp} a}}
        \pi\left(\OpDualB_{a'} \circ_2 \OpDualB_{a''}\right) \\
        & =
        \sum_{\substack{a' \in \Pca^\perp \\
            a' \Ord_{\Pca^\perp} a}}
        \pi\left(\OpDualB_{a'} \circ_1 \OpDualB_{a'}\right)
        -
        \sum_{\substack{a' \in \Pca^\perp \\
            a' \Ord_{\Pca^\perp} a}}
        \pi\left(\OpDualB_{a'} \circ_2 \OpDualB_{a'}\right) \\
        & = 0.
    \end{split}\end{equation}
    Indeed the second equality
    of~\eqref{equ:isomorphism_thin_forests_dual_operad_computation_1}
    comes, by
    Proposition~\ref{prop:alternative_presentation_koszul_dual}, from the
    presence of the elements~\eqref{equ:alternative_relation_dual_2}
    and~\eqref{equ:alternative_relation_dual_3} in
    $\FreeRel_\Pca^{\OpDualB}$, together with the fact that for all
    comparable elements $a'$ and $b'$ in $\Pca$, the fact that
    $a' \Ord_{\Pca^\perp} a$, $b' \Ord_{\Pca^\perp} b$, and
    $a \Ord_{\Pca^\perp} b$ implies that $a' = b'$. Besides, the last
    equality
    of~\eqref{equ:isomorphism_thin_forests_dual_operad_computation_1}
    comes, by
    Proposition~\ref{prop:alternative_presentation_koszul_dual}, from the
    presence of the elements~\eqref{equ:alternative_relation_dual_1} in
    $\FreeRel_\Pca^{\OpDualB}$. Similar arguments show that
    \begin{subequations}
    \begin{equation}
        \pi\left(\OpDualA_b \circ_1 \OpDualA_a
            - \OpDualA_a \circ_2 \OpDualA_a\right) = 0,
    \end{equation}
    \begin{equation}
        \pi\left(\OpDualA_a \circ_1 \OpDualA_a
            - \OpDualA_b \circ_2 \OpDualA_a\right) = 0,
    \end{equation}
    \begin{equation}
        \pi\left(\OpDualA_a \circ_1 \OpDualA_a
            - \OpDualA_a \circ_2 \OpDualA_b\right) = 0.
    \end{equation}
    \end{subequations}
    \smallskip

    We then have shown that the elements
    \begin{subequations}
    \begin{equation}
        \label{equ:isomorphism_thin_forests_dual_operad_generator_1}
        \OpDualA_a \circ_1 \OpDualA_b
        -
        \OpDualA_{a \Min_{\Pca^\perp} b} \circ_2
                \OpDualA_{a \Min_{\Pca^\perp} b},
        \qquad a, b \in {\Pca^\perp} \mbox{ and }
            (a \Ord_{\Pca^\perp} b \mbox{ or } b \Ord_{\Pca^\perp} a),
    \end{equation}
    \begin{equation}
        \label{equ:isomorphism_thin_forests_dual_operad_generator_2}
        \OpDualA_{a \Min_{\Pca^\perp} b} \circ_1
                \OpDualA_{a \Min_{\Pca^\perp} b}
        -
        \OpDualA_a \circ_2 \OpDualA_b,
        \qquad  a, b \in {\Pca^\perp} \mbox{ and }
            (a \Ord_{\Pca^\perp} b \mbox{ or } b \Ord_{\Pca^\perp} a).
    \end{equation}
    \end{subequations}
    are in $\FreeRel_\Pca^{\OpDualB}$. It is immediate that the family
    consisting in the
    elements~\eqref{equ:isomorphism_thin_forests_dual_operad_generator_1}
    and~\eqref{equ:isomorphism_thin_forests_dual_operad_generator_2} is
    free. We denote by $\FreeRel$ the vector space generated by this
    family. By using the same arguments as the ones used in the proof
    of Proposition~\ref{prop:dimension_relations}, we obtain that the
    dimension of $\FreeRel$ is
    \begin{equation}
        \dim \FreeRel =
        4 \, \NbInterv\left(\Pca^\perp\right) - 3 \, \# \Pca^\perp.
    \end{equation}
    Now, by Lemma~\ref{lem:dimensions_relations_thin_forest_poset},
    we deduce that
    \begin{math}
        \dim \FreeRel =
        \dim \FreeRel_\Pca^{\OpDual} =
        \dim \FreeRel_\Pca^{\OpDualB},
    \end{math}
    implying that $\FreeRel$ and $\FreeRel_\Pca^{\OpDualB}$ are equal.
    \smallskip

    Therefore, the family of the $\FreeGen_\Pca^{\OpDualA}$ generating
    $\As(\Pca)^!$ is submitted to the same relations as the family of
    the $\FreeGen_\Pca^\Op$ generating $\As(\Pca^\perp)$
    (compare~\eqref{equ:isomorphism_thin_forests_dual_operad_generator_1}
    with~\eqref{equ:relation_1}
    and~\eqref{equ:isomorphism_thin_forests_dual_operad_generator_2}
    with~\eqref{equ:relation_2}). Whence the statement of the theorem.
\end{proof}
\medskip

The isomorphism $\phi$ between $\As(\Pca^\perp)$ and $\As(\Pca)^!$
provided by Theorem~\ref{thm:isomorphism_thin_forests_dual_operad}
expresses from the generating family $\FreeGen_{\Pca^\perp}^\Op$ of
$\As(\Pca^\perp)$ to the generating family $\FreeGen_\Pca^{\OpDual}$ of
$\As(\Pca)^!$, for any $b \in \Pca^\perp$, as
\begin{equation}
    \phi(\Op_b) =
    \sum_{\substack{a \in \Pca^\perp \\ a \Ord_{\Pca^\perp} b}} \;
    \sum_{\substack{c \in \Pca \\ a \Ord_\Pca c}} \OpDual_c.
\end{equation}
\medskip

For instance, by considering the pair of thin forest posets in duality
\begin{equation}
    \left(\Pca, \Pca^\perp\right) =
    \left(
    \begin{tikzpicture}
        [baseline=(current bounding box.center),xscale=.45,yscale=.42]
        \node[Vertex](1)at(0,0){\begin{math}1\end{math}};
        \node[Vertex](2)at(1,0){\begin{math}2\end{math}};
        \node[Vertex](3)at(2.75,0){\begin{math}3\end{math}};
        \node[Vertex](4)at(2.25,-1){\begin{math}4\end{math}};
        \node[Vertex](5)at(3.25,-1){\begin{math}5\end{math}};
        \node[Vertex](6)at(3.25,-2){\begin{math}6\end{math}};
        \draw[Edge](3)--(4);
        \draw[Edge](3)--(5);
        \draw[Edge](5)--(6);
    \end{tikzpicture}\,, \enspace
    \begin{tikzpicture}
        [baseline=(current bounding box.center),xscale=.45,yscale=.42]
        \node[Vertex](1)at(0,0){\begin{math}1\end{math}};
        \node[Vertex](2)at(0,-1){\begin{math}2\end{math}};
        \node[Vertex](3)at(-1,-2){\begin{math}3\end{math}};
        \node[Vertex](4)at(1,-2){\begin{math}4\end{math}};
        \node[Vertex](5)at(.5,-3){\begin{math}5\end{math}};
        \node[Vertex](6)at(1.5,-3){\begin{math}6\end{math}};
        \draw[Edge](1)--(2);
        \draw[Edge](2)--(3);
        \draw[Edge](2)--(4);
        \draw[Edge](4)--(5);
        \draw[Edge](4)--(6);
    \end{tikzpicture}\,\right),
\end{equation}
the map $\phi : \As(\Pca^\perp) \to \As(\Pca)^!$ defined in the
statement of Theorem~\ref{thm:isomorphism_thin_forests_dual_operad}
satisfies
\begin{subequations}
\begin{equation}
    \phi(\Op_1) = \OpDual_1,
\end{equation}
\begin{equation}
    \phi(\Op_2) = \OpDual_1 + \OpDual_2,
\end{equation}
\begin{equation}
    \phi(\Op_3) = \OpDual_1 + \OpDual_2 + \OpDual_3 + \OpDual_4
        + \OpDual_5 + \OpDual_6,
\end{equation}
\begin{equation}
    \phi(\Op_4) = \OpDual_1 + \OpDual_2 + \OpDual_4,
\end{equation}
\begin{equation}
    \phi(\Op_5) = \OpDual_1 + \OpDual_2 + \OpDual_4 + \OpDual_5
        + \OpDual_6,
\end{equation}
\begin{equation}
    \phi(\Op_6) = \OpDual_1 + \OpDual_2 + \OpDual_4 + \OpDual_6.
\end{equation}
\end{subequations}
\medskip

Moreover, by considering the opposite pair of thin posets forests in
duality
\begin{equation}
    \left(\Pca, \Pca^\perp\right) =
    \left(
    \begin{tikzpicture}
        [baseline=(current bounding box.center),xscale=.45,yscale=.42]
        \node[Vertex](1)at(0,0){\begin{math}1\end{math}};
        \node[Vertex](2)at(0,-1){\begin{math}2\end{math}};
        \node[Vertex](3)at(-1,-2){\begin{math}3\end{math}};
        \node[Vertex](4)at(1,-2){\begin{math}4\end{math}};
        \node[Vertex](5)at(.5,-3){\begin{math}5\end{math}};
        \node[Vertex](6)at(1.5,-3){\begin{math}6\end{math}};
        \draw[Edge](1)--(2);
        \draw[Edge](2)--(3);
        \draw[Edge](2)--(4);
        \draw[Edge](4)--(5);
        \draw[Edge](4)--(6);
    \end{tikzpicture}\,, \enspace
    \begin{tikzpicture}
        [baseline=(current bounding box.center),xscale=.45,yscale=.42]
        \node[Vertex](1)at(0,0){\begin{math}1\end{math}};
        \node[Vertex](2)at(1,0){\begin{math}2\end{math}};
        \node[Vertex](3)at(2.75,0){\begin{math}3\end{math}};
        \node[Vertex](4)at(2.25,-1){\begin{math}4\end{math}};
        \node[Vertex](5)at(3.25,-1){\begin{math}5\end{math}};
        \node[Vertex](6)at(3.25,-2){\begin{math}6\end{math}};
        \draw[Edge](3)--(4);
        \draw[Edge](3)--(5);
        \draw[Edge](5)--(6);
    \end{tikzpicture}\,\right),
\end{equation}
the map $\phi : \As(\Pca^\perp) \to \As(\Pca)^!$ defined in the
statement of Theorem~\ref{thm:isomorphism_thin_forests_dual_operad}
satisfies
\begin{subequations}
\begin{equation}
    \phi(\Op_1) = \OpDual_1 + \OpDual_2 + \OpDual_3 + \OpDual_4
        + \OpDual_5 + \OpDual_6,
\end{equation}
\begin{equation}
    \phi(\Op_2) = \OpDual_2 + \OpDual_3 + \OpDual_4
        + \OpDual_5 + \OpDual_6,
\end{equation}
\begin{equation}
    \phi(\Op_3) = \OpDual_3,
\end{equation}
\begin{equation}
    \phi(\Op_4) = \OpDual_3 + \OpDual_4 + \OpDual_5 + \OpDual_6,
\end{equation}
\begin{equation}
    \phi(\Op_5) = \OpDual_3 + \OpDual_5,
\end{equation}
\begin{equation}
    \phi(\Op_6) = \OpDual_3 + \OpDual_5 + \OpDual_6.
\end{equation}
\end{subequations}
\medskip

Notice also that since the dual of the total order $\Pca$ on a set of
$\ell \geq 0$ elements is the trivial order $\Pca^\perp$ on the same set,
by Theorem~\ref{thm:isomorphism_thin_forests_dual_operad}, $\As(\Pca)$
is the Koszul dual of $\As(\Pca^\perp)$. This is coherent with the
results of~\cite{Gir15b} about the multiassociative operad (equal to
$\As(\Pca)$) and the dual multiassociative operad (equal
to~$\As(\Pca^\perp)$).
\medskip

\section*{Conclusion and remaining questions}
Through this work, we have presented a functorial construction $\As$
from posets to operads establishing a link between the two underlying
categories. The operads obtained through this construction generalize
the (dual) multiassociative operads introduced in~\cite{Gir15b}. As we
have seen, some combinatorial properties of the starting posets $\Pca$
imply properties on the obtained operads $\As(\Pca)$ as, among others,
basicity and Koszulity.
\medskip

This work raises several questions. We have presented two classes of
$\Pca$-associative algebras: the free $\Pca$-associative algebras over
one generator where $\Pca$ are forest posets and a polynomial algebra
involving the antichains of a poset $\Pca$. The question to define
free $\Pca$-associative algebras over one generator with no assumption
on $\Pca$ is open. Also, the question to define some other interesting
$\Pca$-associative algebras has not been considered in this work and
deserves to be addressed.
\medskip

Besides, we have shown that when $\Pca$ is a forest poset, $\As(\Pca)$
is Koszul. The property of being a forest poset for $\Pca$ is only a
sufficient condition for the Koszulity of $\As(\Pca)$ and the question
to find a necessary condition is worthwhile. Notice that the strategy to
prove the Koszulity of an operad by the partition poset
method~\cite{MY91,Val07} (see also~\cite{LV12}) cannot be applied to our
context. Indeed, this strategy applies only on basic operads and we have
shown that almost all operads $\As(\Pca)$ are not basic.
\medskip

\bibliographystyle{alpha}
\bibliography{Bibliography}

\end{document}